\documentclass{amsproc}
\usepackage{amssymb}
\usepackage{color,graphicx,epic,eepic}

\usepackage{hyperref}
\copyrightinfo{2011}{Roland Bauerschmidt, Hugo Duminil-Copin, Jesse Goodman and Gordon Slade}

\newtheorem{theorem}{Theorem}[section]
\newtheorem{lemma}[theorem]{Lemma}
\newtheorem{prop}[theorem]{Proposition}
\newtheorem{conjecture}[theorem]{Conjecture}

\newtheorem{coro}[theorem]{Corollary}
\theoremstyle{definition}
\newtheorem{definition}[theorem]{Definition}
\newtheorem{example}[theorem]{Example}

\newtheorem{problem}{Problem}[section]
\theoremstyle{remark}
\newtheorem{remark}[theorem]{Remark}
\newenvironment{solution}[1]{\begin{proof}[#1]}{\end{proof}}

\newcommand{\N}{\mathbb{N}}
\newcommand{\Z}{\mathbb{Z}}
\newcommand{\Q}{\mathbb{Q}}
\newcommand{\R}{\mathbb{R}}
\newcommand{\C}{\mathbb{C}}

\newcommand{\E}{\mathbb{E}}
\renewcommand{\P}{\mathbb{P}}
\newcommand{\Walks}{\mathcal{W}}
\newcommand{\SAWs}{\mathcal{S}}
\newcommand{\SLE}{\mathrm{SLE}}
\newcommand{\abs}[1]{\left\lvert#1\right\rvert}
\newcommand{\norm}[1]{\left\lVert#1\right\rVert}
\newcommand{\floor}[1]{\left\lfloor#1\right\rfloor}
\newcommand{\set}[1]{\left\{#1\right\}}
\newcommand{\bignorm}[1]{\bigl\Vert#1\bigr\Vert}
\newcommand{\union}{\cup}
\newcommand{\bigunion}{\bigcup}
\newcommand{\intersect}{\cap}

\newcommand{\indicator}[1]{1_{\set{#1}}}
\newcommand{\defeq}{=} % no longer produces special symbol
\newcommand{\ddk}{d^d k} % no longer produces special formatting of lower d
\newcommand{\increasesto}{\nearrow}

\newcommand{\Bubble}{{\sf B}}
\newcommand{\Bridge}{B}

\newcommand{\lbeq}[1]{\label{e:#1}}
\newcommand{\refeq}[1]{\eqref{e:#1}}

\newcommand{\lbsect}[1]{\label{s:#1}}

\newcommand{\refSect}[1]{Section~\ref{s:#1}}

\newcommand{\lbssect}[1]{\label{ss:#1}}

\newcommand{\refSSect}[1]{Section~\ref{ss:#1}}

\newcommand{\lbthm}[1]{\label{t:#1}}
\newcommand{\refthm}[1]{Theorem~\ref{t:#1}}

\newcommand{\lbprop}[1]{\label{p:#1}}
\newcommand{\refprop}[1]{Proposition~\ref{p:#1}}

\newcommand{\lblemma}[1]{\label{l:#1}}
\newcommand{\reflemma}[1]{Lemma~\ref{l:#1}}

\newcommand{\lbcoro}[1]{\label{c:#1}}
\newcommand{\refcoro}[1]{Corollary~\ref{c:#1}}

\newcommand{\lbfig}[1]{\label{f:#1}}
\newcommand{\reffig}[1]{Figure~\ref{f:#1}}

\newcommand{\ddp}[2]{\frac{\partial #1}{\partial #2}}
\newcommand{\half}{\frac{1}{2}}

\numberwithin{equation}{section}

\begin{document}

\title{Lectures on Self-Avoiding Walks}

\author[Bauerschmidt]{Roland Bauerschmidt}
\address{Department of Mathematics, University of British Columbia, 1984 Mathematics Road, Vancouver, BC, Canada V6T 1Z2}
\email{brt@math.ubc.ca}

\author[Duminil-Copin]{Hugo Duminil-Copin}
\address{D\'epartement de Math\'ematiques, Universit\'e de Gen\`eve, 2-4 rue du Li\`evre, 1211 Gen\`eve, Switzerland}
\email{hugo.duminil@unige.ch}

\author[Goodman]{Jesse Goodman}
\address{EURANDOM, Technische Universiteit Eindhoven, PO Box 513, 5600 MB  Eindhoven, The Netherlands}
\email{j.a.goodman@tue.nl}

\author[Slade]{Gordon Slade}
\address{Department of Mathematics, University of British Columbia, 1984 Mathematics Road, Vancouver, BC, Canada V6T 1Z2}
\email{slade@math.ubc.ca}


\begin{abstract}
These lecture notes provide a rapid introduction to a number of rigorous
results on self-avoiding walks, with emphasis on the critical
behaviour.  Following an introductory overview of the central
problems, an account is given of the
Hammersley--Welsh bound on the number of self-avoiding walks and
its consequences for the growth rates of bridges
and self-avoiding polygons.
A detailed proof that the connective constant on the hexagonal lattice
equals $\sqrt{2+\sqrt{2}}$ is then provided.
The lace expansion for self-avoiding walks is described,
and its use in understanding the critical behaviour in
dimensions $d>4$ is discussed.
Functional integral representations of the self-avoiding walk model are
discussed and developed, and their use in a renormalisation group analysis
in dimension $4$ is sketched.
Problems and solutions from tutorials are included.
\end{abstract}

\maketitle

\tableofcontents

\section*{Foreword}

These notes are based on a course on Self-Avoiding Walks given
in B\'uzios, Brazil, in August 2010, as part of the
Clay Mathematics Institute
Summer School and the XIV Brazilian Probability School.
The course consisted of
six lectures by Gordon Slade, a lecture
by Hugo Duminil-Copin based on recent joint work with Stanislav Smirnov
(see \refSect{cchex}),
and tutorials by Roland Bauerschmidt and Jesse Goodman.
The written version of Slade's lectures
was drafted by Bauerschmidt and Goodman,
and the written version of Duminil-Copin's lecture was drafted by himself.
The final manuscript was integrated and prepared jointly by the four authors.

\section{Introduction and overview of the critical behaviour}\lbsect{Definitions}

These lecture notes focus on a number of
rigorous results for self-avoiding walks
on the $d$-dimensional integer lattice ${\mathbb Z}^d$.
The model is defined by assigning equal probability to all
paths of length $n$ starting from
the origin and without self-intersections.
This family of probability
measures is not consistent as $n$ is varied, and thus
does not define a stochastic process; the model is combinatorial in
nature.
The natural questions about self-avoiding walks concern the asymptotic behaviour as the length
of the paths tends to infinity.
Despite its simple definition, the self-avoiding walk is
difficult to study in a mathematically rigorous manner.  Many of
the important problems remain unsolved, and the
basic problems encompass many of the
features and challenges of critical phenomena.
This section gives the basic definitions and an overview of the
critical behaviour.

\subsection{Simple random walks}\lbssect{SRW}

The basic reference model is
\emph{simple random walk} (SRW).
%We primarily consider two possible step distributions which are parametrised by the set $\Omega$ of possible steps:
Let $\Omega \subset \Z^d$ be the set of possible steps. The primary examples considered in these lectures are
\begin{align}
\begin{split}
  \text{the nearest-neighbour model:}
  \qquad\quad
  \Omega &= \set{x\in\Z^d : \norm{x}_1 = 1},
   \\
  \text{the spread-out model:}
  \qquad\quad
  \Omega &= \set{x\in\Z^d : 0 < \norm{x}_\infty \leq L},
\end{split}
\lbeq{nnso}
\end{align}
where $L$ is a fixed integer, usually large.
An $n$-\emph{step walk} is a sequence $\omega=(\omega(0),\omega(1),\dotsc,\omega(n))$ with $\omega(j)-\omega(j-1)\in\Omega$ for $j=1,\dotsc,n$.
The $n$-\emph{step simple random walk} is the uniform measure on  $n$-step walks.
We define the sets
\begin{equation}
  \lbeq{WnxDefn}
  \Walks_n(0,x)=\{\omega : \omega\text{ is an $n$-step  walk with $\omega(0)=0$ and $\omega(n)=x$}\}
\end{equation}
and
\begin{equation}
  \lbeq{WnDefn}
  \Walks_n = \bigunion_{x\in\Z^d} \Walks_n(0,x).
\end{equation}

\subsection{Self-avoiding walks}\lbssect{SAW}

The \emph{weakly self-avoiding walk} and the \emph{strictly self-avoiding walk}
(the latter also called simply \emph{self-avoiding walk})
are the main subjects of these notes.  These are random paths
on $\Z^d$, defined as follows.
Given an $n$-step walk $\omega\in\Walks_n$, and integers $s,t$ with $0\leq s <t\leq n$, let
\begin{equation}
  \lbeq{UstDefn}
  U_{st} = U_{st}(\omega) = -\indicator{\omega(s)=\omega(t)}
  =
  \begin{cases}
    -1 & \text{if $\omega(s)=\omega(t)$},
    \\ \;\;\, 0 & \text{if $\omega(s) \neq \omega(t)$}.
  \end{cases}
\end{equation}
Fix $\lambda\in[0,1]$. We assign to each path $\omega\in\Walks_n$ the weighting factor
\begin{equation}
  \lbeq{WeightFactor}
  \prod_{0\leq s < t\leq n} (1+\lambda U_{st}(\omega)).
\end{equation}
The weights can also be expressed as Boltzmann weights:
\begin{equation}
  \lbeq{cnlambdaxExp}
  \prod_{0\leq s < t\leq n} (1+\lambda U_{st}(\omega)) =
  \exp\Bigl( -g \sum_{0\leq s < t\leq n} \indicator{\omega(s)=\omega(t)} \Bigr)
\end{equation}
with $g=-\log(1-\lambda) \in[0,\infty)$ for $\lambda\in[0,1)$.
Making the convention $\infty \cdot 0 = 0$, the case $\lambda=1$ corresponds to $g=\infty$.

The choice $\lambda=0$ assigns equal weight to all walks
in $\Walks_n$; this is the
case of the simple random walk.
For $\lambda\in(0,1)$, self-intersections are penalised but not forbidden,
and the model is called the \emph{weakly self-avoiding walk}.  The choice
$\lambda=1$ prevents any return to a previously visited site,
and defines the \emph{self-avoiding walk} (SAW).
More precisely, an $n$-step walk $\omega$ is a self-avoiding walk
if and only if the expression \refeq{WeightFactor} is
non-zero for $\lambda=1$, which happens if and only if $\omega$ visits
each site at most once, and for such walks the weight equals $1$.

These weights give rise to associated partition sums $c_n^{(\lambda)}(x)$ and $c_n^{(\lambda)}$ for walks in
$\Walks_n(0,x)$ and $\Walks_n$, respectively:
\begin{equation}
  \lbeq{cnlambdaxDefn}
  c_n^{(\lambda)}(x) = \sum_{\omega\in\Walks_n(0,x)} \prod_{0\leq s < t\leq n} (1+\lambda U_{st}(\omega)),
  \quad
  c_n^{(\lambda)} = \sum_{x\in\Z^d} c_n^{(\lambda)}(x).
\end{equation}
In the case $\lambda=1$, $c_n^{(1)}(x)$ counts the number of
self-avoiding walks of length $n$ ending at $x$, and $c_n^{(1)}$
counts all $n$-step self-avoiding walks. The case $\lambda=0$
reverts to simple random walk, for which $c_n^{(0)}=\abs{\Omega}^n$.
When $\lambda =1$ we will often drop the superscript $(1)$ and write
simple $c_n$ instead of $c_n^{(1)}$.

We also define
probability measures $\mathbb{Q}^{(\lambda)}_n$ on $\Walks_n$
with expectations $\E^{(\lambda)}_n$:
\begin{equation}
  \label{eq:cnlambdaProbability}
  \Q_n^{(\lambda)}(A)
  =
  \frac{1}{c_n^{(\lambda)}}\sum_{\omega\in A}
  \prod_{0\leq s<t\leq n}(1+\lambda U_{st}(\omega))
  \quad (A \subset \Walks_n),
\end{equation}
\begin{equation}
  \lbeq{cnlambdaExpectation}
  \E_n^{(\lambda)}(X) =
  \frac{1}{c_n^{(\lambda)}}
  \sum_{\omega\in\Walks_n} X(\omega)
  \prod_{0\leq s<t\leq n} (1+\lambda U_{st}(\omega))
  \quad
  (X : \Walks_n \to {\mathbb R}).
\end{equation}
The measures $\Q_n^{(\lambda)}$ define the weakly self-avoiding
walk when $\lambda \in (0,1)$ and the strictly self-avoiding walk
when $\lambda =1$.  Occasionally we will also
consider self-avoiding
walks that do not begin at the origin.
%Finally, we define
%the two-point function $G_z^{(\lambda)}(x)$ and susceptibility $\chi^{(\lambda)}(z)$,
%\begin{equation}
%  \lbeq{TwoPointSusDefn}
%  G_z^{(\lambda)}(x) = \sum_{n=0}^\infty c_n^{(\lambda)}(x) z^n,
%  \quad
%%  \lbeq{SuscDefn}
%  \chi^{(\lambda)}(z) = \sum_{x\in\Z^d} G_z^{(\lambda)}(x)
%  = \sum_{n=0}^\infty c_n^{(\lambda)} z^n.
%\end{equation}

%The partition sum for walks from $0$ to $x$ is
%
% \begin{equation}
%   \lbeq{cnlambdaxDefn}
%   c_n^{(\lambda)}(x) = \sum_{\omega\in\Walks_n(0,x)} \prod_{0\leq s < t\leq n} (1+\lambda U_{st}(\omega)).
% \end{equation}
%
%The partition sum $c_n^{(\lambda)}(x)$

%, if we make the convention that $\infty \cdot 0 = 0$.  We also write
%
% \begin{equation}
%   \lbeq{cnlambdaDefn}
%   c_n^{(\lambda)} = \sum_{x\in\Z^d} c_n^{(\lambda)}(x).
% \end{equation}
%

\subsection{Subadditivity and the connective constant}\lbsect{Subadd}

The sequence $c_n^{(\lambda)}$ has the following submultiplicativity property:
\begin{equation}
  \lbeq{cSubmult}
  c_{n+m}^{(\lambda)}
  \leq \sum_{\omega\in\Walks_{n+m}} \prod_{0\leq s<t\leq n} (1+\lambda U_{st}) \prod_{n\leq s'<t'\leq n+m} (1+\lambda U_{s't'})
  \leq c_n^{(\lambda)} c_m^{(\lambda)}.
\end{equation}
Therefore, $\log c_n^{(\lambda)}$ is a \emph{subadditive} sequence:
$\log c_{n+m}^{(\lambda)} \leq \log c_n^{(\lambda)} + \log c_m^{(\lambda)}$.

\begin{lemma}
  \lblemma{Subadd}
  If $a_1,a_2,\dotsc\in\R$ obey $a_{n+m}\leq a_n + a_m$ for every $n,m$, then
  \begin{equation}
    \lbeq{SubaddLim}
    \lim_{n\to\infty} \frac{a_n}{n} = \inf_{n\geq 1} \frac{a_n}{n} \in [-\infty, \infty).
  \end{equation}
\end{lemma}
\begin{proof}
  See Problem~\ref{problem:subadditivity}.
  The value $-\infty$ is possible, e.g., for the sequence $a_n=-n^2$.
\end{proof}

Applying \reflemma{Subadd} to $c_n^{(\lambda)}$ gives the existence
of $\mu_\lambda$ such that
$\lim \frac{1}{n}\log c_n^{(\lambda)} \defeq \log\mu_\lambda \leq \frac{1}{n}\log c_n^{(\lambda)}$ for all $n$, i.e.,
\begin{equation}
  \lbeq{mulambdaDefn}
  \mu_\lambda \defeq \lim_{n\to\infty} ( c_n^{(\lambda)} )^{1/n}
  \text{ exists, and}
  \quad c_n^{(\lambda)} \geq \mu_\lambda^n \quad \text{for all $n$}.
\end{equation}
In the special case $\lambda=1$, we write simply $\mu=\mu_1$.
This $\mu$, which depends on $d$ (and also on $L$ for the spread-out
model), is called the
\emph{connective constant}.
For the nearest-neighbour model, by counting only walks that
move in positive coordinate directions, and by counting walks that
are restricted only to prevent immediate reversals of steps, we obtain
\begin{equation}
  \lbeq{cnEasyBounds}
  d^n \leq c_n \leq 2d(2d-1)^{n-1} \qquad \text{which implies} \qquad d\leq\mu\leq 2d-1.
\end{equation}
For $d=2$, the following rigorous bounds are known:
\begin{equation}
  \lbeq{muBounds}
  \mu\in[2.625\,622, 2.679\,193].
\end{equation}
The lower bound is due to Jensen \cite{Jens04b} via bridge enumeration
(bridges are defined in \refSect{HammWelsh} below),
and the upper bound is due to P\"onitz and Tittmann
\cite{PT00} by comparison with finite-memory walks.
The estimate
\begin{equation}
  \lbeq{muEstimate}
  \mu = 2.638\,158\,530\,31(3)
\end{equation}
is given in \cite{Jens03}; here the $3$ in parentheses represents
the subjective error in the last digit.
It has been observed that
$1/\mu$ is well approximated by the smallest positive root
of $581 x^4 + 7x^2 - 13 = 0$ \cite{CEG93,JG99}, though no derivation
or explanation of this quartic polynomial is known, and
later evidence has raised doubts about its validity
\cite{Jens03}.

Even though the definition of self-avoiding walks has been restricted
to the graph $\Z^d$ thus far,
it applies more generally.
In 1982, arguments based on  a Coulomb gas formalism led
Nienhuis \cite{Nien82} to predict that on the hexagonal lattice
the connective constant is equal to
$\sqrt{2+\sqrt{2}}$.  This was very recently proved by
Duminil-Copin and Smirnov \cite{D-CS10}, whose theorem is
the following.
\begin{theorem}
  \lbthm{muHexagonal}
  The connective constant for the hexagonal lattice is
  \begin{equation}
    \lbeq{muHexagonalValue}
    \mu = \textstyle{\sqrt{2+\sqrt{2}}}.
    % previously: \mu = \sqrt{\raisebox{0pt}[0.75em]{$2+\sqrt{2}$}} -- this produces a slight difference
  \end{equation}
\end{theorem}
The proof of \refthm{muHexagonal} is presented
in \refSect{cchex} below.
Except for trivial cases, this is the only lattice for which
the connective constant is known explicitly.

Returning to ${\mathbb Z}^d$,
in 1963, Kesten  \cite{Kest63} proved that
  \begin{equation}
    \lbeq{cn2Ratio}
    \lim_{n\to\infty} \frac{c_{n+2}}{c_n}=\mu^2,
  \end{equation}
  but it remains an open problem (for $d=2,3,4$) to prove that
  \begin{equation}
    \lim\limits_{n\to\infty} \dfrac{c_{n+1}}{c_n}=\mu .
  \end{equation}
Even the proof of $c_{n+1}\geq c_n$ is a non-trivial result,
proved by O'Brien
\cite{OBri90}, though it is not hard to show that $c_{n+2}\geq c_n$.

\subsection{\texorpdfstring{$1/d$}{1/d} expansion}\lbssect{dexpansion}

It was proved by
Hara and Slade \cite{HS95} that the connective constant $\mu(d)$ for $\Z^d$
(with nearest-neighbour steps) has an asymptotic expansion in powers of $1/2d$ as $d\to\infty$:
There exist integers $a_i\in\Z$, $i=-1,0,1,\dotsc$ such that
\begin{equation}
\lbeq{1overdExpansion}
\mu(d) \sim \sum_{i=-1}^\infty \frac{a_i}{(2d)^i}
\end{equation}
in the sense that $\mu(d)=a_{-1}(2d) + a_0 +\dotsb+a_{M-1} (2d)^{-(M-1)}
+O(d^{-M})$, for each fixed $M$.
In Problem~\ref{problem:1/d} below, the first three terms
are computed.
The constant in the $O(d^{-M})$
term may depend on $M$. It is expected, though not proved,
that the asymptotic series in \refeq{1overdExpansion} has
radius of convergence $0$, so that the right-hand side of
\refeq{1overdExpansion} diverges for each fixed $d$.  The values
of $a_i$ are known for $i=-1,0,\dotsc,11$ and grow rapidly in
magnitude; see Clisby, Liang, and Slade
\cite{CLS07}.

Graham \cite{Grah10} has proved
Borel-type error bounds for the asymptotic expansion of
$z_c = z_c(d) \defeq \mu^{-1}$.  Namely,
writing the asymptotic expansion of $z_c$ as
$\sum_{i=1}^\infty \alpha_i(2d)^{-i}$, there is a constant $C$, independent
of $d$ and $M$, such that
for each $M$ and for all $d \ge 1$,
\begin{equation}
\lbeq{zcAsympBound}
\Big|z_c-\sum_{i=1}^{M-1}\frac{\alpha_i}{(2d)^i}\Big|
\leq \frac{C^M M!}{(2d)^M}.
\end{equation}
An extension of \refeq{zcAsympBound} to \emph{complex} values of the
dimension $d$ would be needed in order to
apply the method of Borel summation to recover the value of $z_c$, and
hence of $\mu(d)$, from the asymptotic series.

\subsection{Critical exponents}\lbsect{CritExp}

It is a characteristic feature of models of statistical mechanics at the critical point
that there exist \emph{critical exponents} which describe the asymptotic behaviour
on the large scale. It is a deep conjecture, not yet properly understood
mathematically, that these
critical exponents are \emph{universal}, meaning that they depend only on the
spatial dimension of the system, but not on  details such as the
specific lattice in $\R^d$.  For the case of the self-avoiding
walk, this conjecture of universality extends to lack of dependence on
the constant $\lambda$, as soon as $\lambda>0$.
%$\lambda \in (0,1]$ in the weights (the case $\lambda=0$ is omitted here).
We now introduce the critical exponents,
and in \refSSect{dimEffect} we will discuss what is known about them
in more detail.

\medskip
\subsubsection{Number of self-avoiding walks}\lbssect{cnCritExp}

It is predicted that
for each $d$ there is a constant $\gamma$ such that for all
$\lambda\in (0,1]$, and for both the nearest-neighbour and spread-out
models,
  \begin{equation}
    \lbeq{cnlambdaAsymp}
    c_n^{(\lambda)} \sim A_\lambda \mu_\lambda^n n^{\gamma-1}.
  \end{equation}
  Here $f(n) \sim g(n)$ means $\lim_{n\to\infty}f(n)/g(n)=1$.
  The predicted values of the critical exponent $\gamma$ are:
  \begin{equation}
    \lbeq{gammaPrediction}
    \gamma =
    \begin{cases}
      1 & d=1, \\
      \tfrac{43}{32} & d=2,  \\
      1.16\ldots & d=3, \\
      1 & d=4,  \\
      1 & d\geq 5.
    \end{cases}
  \end{equation}
In fact, for $d=4$, the prediction involves a logarithmic correction:
  \begin{equation}
    \lbeq{cnlambdaAsymp4}
    c_n^{(\lambda)} \sim A_\lambda \mu_\lambda^n (\log n)^{1/4}.
  \end{equation}
This situation
should be compared with simple random walk, for which $c_n^{(0)}=\abs{\Omega}^n$,
so that $\mu_0$ is equal to the degree $\abs{\Omega}$ of the lattice,
and $\gamma = 1$.

In the case of the self-avoiding walk (i.e., $\lambda=1$),
$\gamma$ has a probabilistic interpretation.
Sampling independently from two $n$-step self-avoiding walks uniformly,
\begin{equation}
\lbeq{P2SAWavoid}
\P(\omega_1\intersect \omega_2 = \set{0})
= \frac{c_{2n}}{c_n^2}
\sim {\rm const}\frac{1}{n^{\gamma-1}},
\end{equation}
so $\gamma$ is a measure of how likely it is for two self-avoiding walks
to avoid each other.
The analogous question for SRW is discussed in \cite{Lawl91}.

%(Here we have identified a walk $\omega$ with its trace $\set{\omega(s): 0\leq s\leq n}$).

Despite the precision of the prediction
%\refeq{gammaPrediction},
\refeq{cnlambdaAsymp},
the best rigorously known bounds in
dimension $d=2,3,4$ are
very far from tight and almost 50 years old.  In \cite{HW62}, Hammersley
and Welsh proved that, for all $d \ge 2$,
\begin{equation}
\lbeq{HW}
\mu^n \leq c_n \leq
\mu^n e^{\kappa \sqrt{n}}
\end{equation}
(the lower bound is just subadditivity, the upper bound is
nontrivial).  This was improved slightly
by Kesten \cite{Kest63}, who showed that for $d = 3,4,\ldots$,
\begin{equation}
\lbeq{cnBounds}
\mu^n \leq c_n \leq
\mu^n \exp\left(\kappa n^{2/(d+2)} \log n\right).
\end{equation}
The proof of the Hammersley--Welsh bound is the subject of \refSect{HammWelsh}.

\medskip
\subsubsection{Mean-square displacement}\lbssect{MeanSqDispCritExp}

% The expectation associated to the measure of weakly self-avoiding walk, $0 < \lambda < 1$,
% or strictly self-avoiding walk, $\lambda = 1$, is defined by
% %We define an expectation by
% %
% \begin{equation}
% \lbeq{cnlambdaExpectation}
% \E_n^{(\lambda)}X = \frac{1}{c_n^{(\lambda)}} \sum_{\omega\in\Walks_n} X(\omega) \prod_{0\leq s<t\leq n} (1+\lambda U_{st}(\omega)).
% \end{equation}
% (For $\lambda=0$ and $\lambda=1$ this corresponds to a uniformly chosen simple random walks and
% uniformly chosen self-avoiding walks.)
% \RB{In the first paragraph, we should either mention SRW as well, or phrase this somewhat differently.
% Also, should the expectation be introduced at this late stage or maybe earlier?}

Let $|x|$ denote the Euclidean norm of $x\in\R^d$.
It is predicted that for $\lambda\in(0,1]$, and for both the nearest-neighbour
and spread-out models,
  \begin{equation}
    % \label{}
    \E_n^{(\lambda)} \abs{\omega(n)}^2 \sim D_\lambda n^{2\nu},
  \end{equation}
  with
  \begin{equation}
    \lbeq{nuPrediction}
    \nu =
    \begin{cases}
      1 & d=1, \\
      \tfrac{3}{4} & d=2, \\
      0.588\ldots & d=3, \\
      \tfrac{1}{2} & d=4, \\
      \tfrac{1}{2} & d\geq 5.
    \end{cases}
  \end{equation}
Again, a logarithmic correction is predicted for $d=4$:
\begin{equation}
\lbeq{nuPrediction4}
    \E_n^{(\lambda)} \abs{\omega(n)}^2 \sim D_\lambda n (\log n)^{1/4}.
\end{equation}
This should be compared with the SRW, for which
$\nu=\tfrac{1}{2}$ in all dimensions.

Almost nothing is known rigorously about $\nu$ in dimensions $2,3,4$.
It is an open problem to show that the mean-square
displacement grows at least as rapidly as simple random walk,
and grows more slowly than ballistically, i.e., it has not been proved that
\begin{equation}
  c n\leq \E_n^{(1)}\abs{\omega(n)}^2 \leq C n^{2-\epsilon},
\end{equation}
or even that the endpoint is typically as far away as the surface
of a ball of volume $n$, i.e., $c n^{2/d}\leq \E_n^{(1)}\abs{\omega(n)}^2$.
Madras (unpublished) has shown $\E_n^{(1)}\abs{\omega(n)}^2 \geq c n^{4/3d}$.

\medskip
\subsubsection{Two-point function and susceptibility}

The two-point function is defined by
\begin{equation}
  \lbeq{TwoPointDefn}
  G_z^{(\lambda)}(x) = \sum_{n=0}^\infty c_n^{(\lambda)}(x) z^n,
\end{equation}
and the susceptibility by
\begin{equation}
  \lbeq{SuscDefn}
  \chi^{(\lambda)}(z) = \sum_{x\in\Z^d} G_z^{(\lambda)}(x) = \sum_{n=0}^\infty c_n^{(\lambda)} z^n.
\end{equation}
Since $\chi^{(\lambda)}$ is a power series whose coefficients satisfy
\refeq{mulambdaDefn}, its radius of convergence $z_c^{(\lambda)}$
is given by $z_c^{(\lambda)}=\mu_\lambda^{-1}$.  The value $z_c^{(\lambda)}$
is referred to as the \emph{critical point}.
\begin{prop}
\lbprop{TwoPointSubcritDecay}
Fix $\lambda\in[0,1], z\in(0,z_c^{(\lambda)})$.  Then $G_z^{(\lambda)}(x)$
decays exponentially in $x$.
\end{prop}
\begin{proof}
For simplicity, we consider only the nearest-neighbour model, and
we omit $\lambda$ from the notation.  Since $c_n(x)=0$ if $n<\norm{x}_1$,
\begin{equation}
%\label{}
G_z(x)=\sum_{n=\norm{x}_1}^\infty c_n(x) z^n \leq \sum_{n=\norm{x}_1}^\infty c_n z^n.
\end{equation}
Fix $z<z_c=1/\mu$ and choose $\epsilon>0$ such that
$z(\mu+\epsilon) <1$.  Since $c_n^{1/n}\to\mu$,
there exists $K=K(\epsilon)$ such that $c_n\leq K(\mu+\epsilon)^n$ for all $n$.
Hence
\begin{equation}
    G_z(x)\leq K\sum_{n=\norm{x}_1}^\infty (z(\mu+\epsilon))^n
    \leq K' (z(\mu+\epsilon))^{\norm{x}_1},
\end{equation}
as claimed.
\end{proof}

We restrict temporarily to $\lambda = 1$.
Much is known about $G_z(x)$ for $z<z_c$:
there is a norm $\abs{\,\cdot\,}_z$ on $\R^d$, satisfying
$\|u\|_\infty \le \abs{u}_z \le \|u\|_1$ for all $u \in \R^d$,
such that $m(z)\defeq\lim\limits_{\abs{x}_z\to\infty} -\frac{\log G_z(x)}{\abs{x}_z}$ exists and is finite.
The \emph{correlation length} is defined by $\xi(z)=1/m(z)$, and hence approximately
\begin{equation}
  % \label{}
  G_z(x) \approx e^{-\abs{x}_z/\xi(z)}.
\end{equation}
Indeed, more precise asymptotics (Ornstein--Zernike decay) are known
\cite{CC86b,MS93,CIV04}:
\begin{equation}
  % \label{}
  G_z(x) \sim \frac{c}{|x|_z^{(d-1)/2}} e^{-|x|_z/\xi(z)}
  \quad \text{as $x \to \infty$},
\end{equation}
and the arguments leading to this also prove that
\begin{equation}
\lbeq{xidiverges}
    \lim\limits_{z\increasesto z_c} \xi(z)=\infty.
\end{equation}

As a refinement of \refeq{xidiverges}, it is predicted that
as $z\increasesto z_c$,
  \begin{align}
    % \label{}
    \xi(z)
    &\sim
    \mathrm{const}\left(1-\frac{z}{z_c}\right)^{-\nu},
  \end{align}
and that, in addition,  as $\abs{x}\to\infty$ (for $d\geq 2$),
  \begin{align}
    G_{z_c}(x)
    &\sim
    \frac{\mathrm{const}}{\abs{x}^{d-2+\eta}}.
  \end{align}
The exponents $\gamma$, $\eta$ and $\nu$ are predicted to be
related to each other via \emph{Fisher's relation}
(see, e.g., \cite{MS93}):
  \begin{equation}
    \lbeq{FishersRel}
    \gamma = (2-\eta)\nu.
  \end{equation}

%As a rule of thumb,
There is typically
a correspondence between the asymptotic growth of the coefficients in a
generating function and the behaviour of the
generating function near its dominant singularity.  For our purpose we note that, under suitable hypotheses,
\begin{equation}
%\label{}
a_n \sim \frac{n^{\gamma-1}}{R^n}\text{ as $n\to\infty$}
\qquad \overset{\approx}{\Longleftrightarrow} \qquad
 \sum_n a_n z^n \sim \frac{C}{(1-z/R)^\gamma} \text{ as $z\increasesto R$}.
\end{equation}
The easier $\implies$ direction is known as an Abelian theorem,
and the more delicate
$\Longleftarrow$ direction is known as a Tauberian theorem
\cite{Hard49}.
With this in mind, our earlier prediction for $c_n^{(\lambda)}$
for $\lambda\in (0,1]$
corresponds to:
  \begin{equation}
    \lbeq{SuscAsymp}
    \chi^{(\lambda)}(z) \sim \frac{\mathrm{const}_\lambda}{(1-z/z_c)^\gamma}
  \end{equation}
  as $z\increasesto z_c$, with an additional factor
  $\abs{\log(1-z/z_c)}^{1/4}$ on the right-hand side when $d=4$.

\subsection{Effect of the dimension}\lbssect{dimEffect}

Universality asserts that self-avoiding walks on different
lattices in a fixed dimension $d$
should behave in the same way, independently of the fine details of
how the model is defined.
However,
the behaviour does depend very strongly on the dimension.

\medskip
\subsubsection{$d=1$}

For the nearest-neighbour model with $\lambda=1$ it is a triviality
that $c_n^{(1)}=2$ for all $n\geq 1$ and $\abs{\omega(n)}=n$ for
all $\omega$, since a self-avoiding walk must continue either in the
negative or in the positive direction.
Any configuration $\omega\in \Walks_n$ is possible when $\lambda\in (0,1)$,
however, and it is by no means trivial to prove
that the critical behaviour when $\lambda\in (0,1)$ is
similar to the case of $\lambda =1$.
The following theorem of K\"{o}nig \cite{Koni96}
(extending a result of Greven and den Hollander \cite{GH93})
proves that the weakly self-avoiding
walk measure \eqref{eq:cnlambdaProbability} does have ballistic
behaviour for all $\lambda \in (0,1)$.
% , and we define the law $\Q_n^{(\lambda)}$ on $\Walks_n$ by
% %
% \begin{equation}
%   % \label{}
%   \Q_n^{(\lambda)}(A)=\frac{1}{c_n^{(\lambda)}}\sum_{\omega\in A}\prod_{0\leq s<t\leq n}(1+\lambda U_{st}(\omega)).
% \end{equation}
% %
% ($\Q_n^{(\lambda)}$ is the probability measure corresponding to the expectation $\E_n^{(\lambda)}$ of \refeq{cnlambdaExpectation}.)
%
\begin{theorem}
  \lbthm{WSAW1d}
  Let $d=1$.
  For each $\lambda\in(0,1)$, there exist $\theta(\lambda)\in(0,1)$
  and $\sigma(\lambda)\in(0,\infty)$ such that for all $u \in \R$,
  \begin{equation}
    \lbeq{WSAW1dCLT}
    \lim_{n\to\infty} \Q_n^{(\lambda)} \!
    \left(\frac{\abs{\omega(n)}-n\theta}{\sigma\sqrt{n}} \leq u \right)
    = \int_{-\infty}^u \frac{e^{-t^2/2}}{\sqrt{2\pi}} \,dt.
  \end{equation}
\end{theorem}
A similar result is proved in \cite{Koni96} for the 1-dimensional
spread-out strictly self-avoiding walk.
The result of
\refthm{WSAW1d} should be contrasted to the case $\lambda=0$,
which has diffusive rather than ballistic behaviour.
It remains an open problem to prove the intuitively appealing
statement that $\theta$ should be an increasing function of $\lambda$.
A review of results for $d=1$ is given in \cite{HK01}.

\medskip
\subsubsection{$d=2$}

Based on non-rigorous Coulomb gas methods,
Nienhuis \cite{Nien82} predicted that
$\gamma=\tfrac{43}{32}$, $\nu=\tfrac{3}{4}$.
These predicted values have been confirmed numerically by Monte Carlo
simulation, e.g., \cite{LMS95},
and exact enumeration of self-avoiding walks up to
length $n= 71$ \cite{Jens04}.

Lawler, Schramm, and Werner \cite{LSW04}
have given major mathematical support to these predictions.
Roughly speaking, they show that if self-avoiding walk has a scaling
limit, and if this scaling limit has a certain
conformal invariance property, then the scaling limit
must be $\SLE_{8/3}$ (the Schramm--Loewner evolution with parameter
$\kappa = \frac 83$).
The values of $\gamma$ and $\nu$ are then recovered from an
$\SLE_{8/3}$ computation.
Numerical evidence supporting the statement that the scaling limit
is $\SLE_{8/3}$ is given in \cite{Kenn04}.
However,
until now, it remains an open problem to prove the required existence
and conformal invariance of the scaling limit.

The result of \cite{LSW04} is discussed in greater detail in the course of
Vincent Beffara  \cite{Beff11}.
Here, we describe it
only briefly, as follows.
Consider a simply connected domain $\Omega$ in the complex plane $\C$
with two points $a$ and $b$ on the boundary. Fix $\delta>0$, and
let $(\Omega_{\delta},a_{\delta},b_{\delta})$ be a discrete approximation
of $(\Omega,a,b)$ in the following sense: $\Omega_{\delta}$ is the largest
finite domain of $\delta \mathbb{Z}^2$ included in $\Omega$, $a_\delta$ and
$b_\delta$ are the closest vertices of $\delta \mathbb{Z}^2$ to $a$ and $b$ respectively.
When $\delta$ goes to 0, this provides an approximation of the domain.

For fixed $z,\delta>0$,
% $\mathbb{P}_{z,\delta}$
there is a probability measure on the set of self-avoiding walks
$\omega$ between $a_\delta$ and $b_\delta$ that remain in
$\Omega_\delta$ by assigning to $\omega$ a \emph{Boltzmann weight}
proportional to $z^{\ell(\omega)}$, where $\ell(\omega)$ denotes the length of $\omega$. We obtain a random piecewise
linear curve, denoted by $\omega_\delta$.

It is possible to prove that when $z<z_c=1/\mu$, walks are penalised
so much with respect to their length that $\omega_\delta$
becomes straight when $\delta$ goes to 0; this is
closely related to the Ornstein--Zernike decay results.
 On the other hand, it is expected that,
 when $z>z_c$, the entropy wins against the penalisation and
 $\omega_\delta$ becomes space filling when $\delta$ tends to 0.
 Finally, when $z=z_c$, the sequence of measures conjecturally
 converges to a random continuous curve.  It is for this case that
 we have the following conjecture of
 Lawler, Schramm and Werner \cite{LSW04}.

\begin{conjecture}
\label{conjecture Lawler Schramm Werner}
For $z=z_c$, the random curve $\omega_\delta$ converges to
$\SLE_{8/3}$ from $a$ and $b$ in the domain $\Omega$.
\end{conjecture}

It remains a major open problem in 2-dimensional statistical
mechanics to prove the conjecture.

%The procedure consisting of taking the limit when $\delta$ goes to
%0 is called taking the \emph{scaling limit of the self-avoiding walk}.
%It should be compared to the case of simple random walks converging to
%the Brownian motion. The critical exponents can be computed via SLE(8/3)
%(for instance $1/\nu$ is equal to the Hausdorff dimension of SLE(8/3),
%see the course of Vincent Beffara for a computation of this quantity).
%The limiting object is different from the Brownian motion.
%Thus, it is not surprising that critical exponents for the simple
%random walk and the self-avoiding walk differ in dimension 2
%(contrarily to dimensions 4 and higher).

%It has been mentioned in the previous chapters that Nienhuis' conjectures were much more precise. The number of self-avoiding walks of length $n$ should behave like
%$$c_n\sim An^{\gamma-1}\mu^n$$
%where $\gamma=43/32$. Moreover, the mean-scared displacement should grow like $n^{2\nu}$ where $\nu=3/4$. These predictions were based on the  so-called Coulomb gas formalism.

%\begin{remark}We mentioned that the critical exponents $\gamma$ and $\nu$ were predicted in the case of the hexagonal lattice. Nevertheless, the lattice is of no relevance here, since the exponents are expected to be \emph{universal}: they should depend on the lattice \emph{only} through the dimension. Therefore, the self-avoiding walk should have the same quantitative behavior on different two-dimensional lattices, for instance the hexagonal and square lattices...
%\end{remark}

\medskip
\subsubsection{$d=3$}

For $d=3$, there are no rigorous results for critical exponents,
and no mathematically well-defined candidate has been proposed for
the scaling limit.
An early prediction for the values of $\nu$, referred to as the
Flory values \cite{Flor49}, was
$\nu = \frac{3}{d+2}$ for $1\le d \le 4$.
This does give the correct answer for $d=1,2,4$, but it is not
quite accurate for $d=3$---the Flory argument is very remote from
a rigorous mathematical proof.  Flory's interest in the problem was
motivated by the use of SAWs to model polymer molecules; this application
is discussed in detail in the course of Frank den Hollander \cite{Holl11}
(see also \cite{Holl09}).

For $d=3$, there are three methods to compute the exponents approximately.
In one method, non-rigorous
field theory computations in theoretical physics  \cite{GZ98} combine
the $n \to 0$ limit
for the $O(n)$ model with an expansion in $\epsilon = 4-d$ about
dimension $d=4$, with $\epsilon = 1$.
Secondly, Monte Carlo studies have been carried out with walks of length
33,000,000 \cite{Clis10}, using the pivot algorithm \cite{MS88,Jans09}.
Finally, exact enumeration plus series analysis has been used;
currently the most extensive enumerations in dimensions
$d \ge 3$ use the lace expansion \cite{CLS07}, and for $d=3$ walks
have been enumerated to length $n=30$.
The exact enumeration estimates for $d=3$ are
$\mu = 4.684043(12)$,
$\gamma = 1.1568(8)$,
$\nu = 0.5876(5)$ \cite{CLS07}.
Monte Carlo estimates are consistent with these values:
$\gamma = 1.1575(6)$ \cite{CCP98} and
$\nu= 0.587597(7)$ \cite{Clis10}.

\medskip
\subsubsection{$d=4$}
\label{sec:d4}
Four dimensions is the \emph{upper critical dimension}
for the self-avoiding walk.
This term encapsulates the notion that for $d>4$ self-avoiding walk has
the same critical behaviour as simple random walk, while for $d<4$ it does not.
The dimension $4$ can be guessed by considering the fractal properties
of the simple random walk:
for $d\geq 2$, the path of a simple random walk is two-dimensional.
If $d>4$, two independent two-dimensional objects should generically
not intersect,
so that the effect of self-interaction between the past and the
future of a simple random walk should be negligible.
In $d=4$, the expected number of intersections between two
independent random walks tends to infinity,
but only logarithmically in the length.  Such considerations
are related to the logarithmic corrections that appear in
\refeq{cnlambdaAsymp4} and \refeq{nuPrediction4}.

The existence of logarithmic corrections to scaling has been proved
for models of weakly self-avoiding walk on a 4-dimensional \emph{hierarchical}
lattice, using rigorous renormalisation group methods
\cite{BEI92,BI03c,BI03d,HO10}.  The hierarchical lattice is a simplification
of the hypercubic lattice $\Z^4$ which is particularly amenable to the
renormalisation group approach.  Recently there has been progress in the
application of renormalisation group methods to a continuous-time
weakly self-avoiding walk model on $\Z^4$ itself, and in particular it has been
proved in this context that the critical two-point function has $|x|^{-2}$
decay \cite{BS11},
which is a statement that the critical exponent $\eta$ is equal to $0$.
This is the topic of \refSect{ctwsaw} below.

\medskip
\subsubsection{$d\geq 5$}

% The \emph{lace expansion} is a method that was first introduced by
% Brydges and Spencer \cite{BS85} and has subsequently been extended to
% a wide variety of models including lattice trees, lattice animals,
% percolation, oriented percolation, contact process, Ising model, and random
% subgraphs of finite graphs.  For a survey, see \cite{Slad06}.

Using the lace expansion, it has been proved that for the nearest-neighbour
model in dimensions $d \ge 5$ the critical exponents exist and take their
so-called \emph{mean field} values $\gamma =1$, $\nu = \frac 12$ \cite{HS92b,HS92a}
and $\eta = 0$ \cite{Hara08}, and that the scaling limit is Brownian motion
\cite{HS92a}.
The lace expansion for self-avoiding walks is discussed in \refSect{LaceExp},
and its application to prove simple random walk behaviour in dimensions
$d \ge 5$ is discussed in \refSect{ConvLace}.

\subsection{Tutorial}
\lbsect{Tut1}

\begin{problem} \label{problem:subadditivity}
  Let $(a_n)$ be a real-valued sequence that is subadditive, that is, $a_{n+m} \leq a_n + a_m$ holds for all $n,m$.
  Prove that $\lim_{n\to\infty} n^{-1} a_n$ exists in $[-\infty, \infty)$ and equals $\inf_{n} n^{-1} a_n$.
\end{problem}

\begin{problem} \label{problem:connectivec}
  Prove that the connective constant $\mu$ for the nearest-neighbour model
  on the square lattice $\Z^2$ obeys the strict inequalities $2<\mu <3$.
\end{problem}

\begin{problem} \label{problem:consistent}
  %Let $\Q_n^{(1)}$ be the uniform measure on the set $\mathcal{S}_n$ of $n$-step self-avoiding walks.
  A family of probability measures $(\P_n)$ on $\Walks_n$ is called consistent if
  $\P_n(\omega) = \sum_{\rho > \omega} \P_m(\rho)$ for all $m > n$ and for
  all $\omega \in \Walks_n$, where the sum is over all $\rho
   \in \Walks_m$ whose first $n$ steps agree with $\omega$.
  Show that $\Q_n^{(1)}$, the uniform measure on SAWs, does not provide a consistent family.
\end{problem}

\begin{problem} \label{problem:2point-1d}
  Show that the Fourier transform of the two-point function of the 1-dimensional
  strictly self-avoiding walk is given by
  \begin{equation}
    \hat G_z (k) = \frac{1 - z^2}{1 + z^2 - 2z \cos k}.
  \end{equation}
  Here $\hat f(k) = \sum_{x\in \Z^d}f(x) e^{ik\cdot x}$.
\end{problem}

\begin{problem} \label{problem:tauberian}
  Suppose that $f(z) = \sum_{n=0}^\infty a_n z^n$ has radius of convergence $1$.
  Suppose that $|f(z)| \leq c |1-z|^{-b}$ uniformly in $|z| < 1$, with $b \geq 1$.
  Prove that, for some constant $C$,
  $|a_n| \leq C n^{b-1}$ if $b > 1$, and that $|a_n|  \leq C \log n$ if $b = 1$. Hint:
  \begin{equation}
    a_n = \frac{1}{2\pi i} \oint_{\Gamma_n} \frac{f(z)}{z^{n+1}} \; dz,
  \end{equation}
  where $\Gamma_n = \{ z \in \C: |z| = 1- \frac{1}{n} \}$.
\end{problem}

\begin{problem} \label{problem:transience}
  Consider the nearest-neighbour
  simple random walk $(X_n)_{n\geq 0}$ on $\Z^d$ started at the origin.
  Let $D(x) = (2d)^{-1}\indicator{\|x\|_1=1}$ denote its step distribution.
  The two-point function for simple random walk is defined by
    \begin{equation}
      C_z(x) = \sum_{n\geq 0} c^{(0)}_n(x) z^n
      = \sum_{n\geq 0} D^{*n}(x) (2dz)^n,
    \end{equation}
    where $D^{*n}$ denotes the $n$-fold convolution of $D$ with itself.
\medskip

  %\begin{enumerate}
  %\item
  (a)
    Let $u$ denote the probability that the walk ever returns to the origin. The
    walk is recurrent if $u = 1$ and transient if $u < 1$. Let $N$ denote the random
    number of visits to the origin, including the initial visit at time $0$, and let
    $m = {\mathbb E}(N)$. Show that $m = (1-u)^{-1}$; so the walk is recurrent if and only if $m = \infty$.
    \medskip
  %\item

    (b)
    Show that
    \begin{equation}
      m= \sum_{n\geq 0} {\mathbb P}(X_n = 0) = \int_{[-\pi,\pi]^d} \frac{1}{1-\hat D(k)} \frac{d^dk}{(2\pi)^d} .
    \end{equation}
    Thus transience is characterised by the integrability of
    $\hat C_{z_0}(k)$, where $z_0 = (2d)^{-1}$.
    % XXX: c_n is #SRW not #SAW
    \medskip
  %\item

    (c)
    %Consider the nearest-neighbour model, with step distribution defined by
    %$\Omega = \{ x \in \Z^d: |x|_1 = 1\}$.
    Show that the walk is recurrent in dimensions $d \leq 2$
    and transient for $d > 2$.
  %\end{enumerate}
\end{problem}

\begin{problem} \label{problem:intersection}
  Let $X^1= (X^1_i)_{i\geq 0}$ and $X^2=(X^2_i)_{i\geq 0}$ be two independent
  nearest-neighbour simple random walks on $\Z^d$
  started at the origin, and let
  \begin{equation}
    I = \sum_{i\geq 0}\sum_{j\geq 0} \indicator{X^1_i = X^2_j}
  \end{equation}
  be the random number of intersections of the two walks. Show that
  \begin{equation}
    {\mathbb E}(I) =  \int_{[-\pi,\pi]^d} \frac{1}{[1-\hat D(k)]^2} \frac{d^dk}{(2\pi)^d}.
  \end{equation}
  Thus ${\mathbb E}(I)$ is finite if and only if $\hat C_{z_0}$ is square integrable.
  Conclude %for the nearest-neighbour model
  that the expected number of
  intersections is finite if $d > 4$ and infinite if $d \leq 4$.
\end{problem}

\section{Bridges and polygons}

Throughout this section, we consider only the nearest-neighbour strictly
self-avoiding walk on $\Z^d$.  We will introduce a class of self-avoiding
walks called bridges, and will show that the number of bridges grows
with the same exponential rate as the number of self-avoiding walks, namely
as $\mu^n$.
The analogous fact for the hexagonal lattice ${\mathbb H}$
will be used in \refSect{cchex}
as an ingredient in the proof that the connective constant for  ${\mathbb H}$
is $\sqrt{2+\sqrt{2}}$.
The study of bridges will also lead to the proof of the
Hammersley--Welsh bound \refeq{HW} on $c_n$.
Finally, we will study self-avoiding polygons, and show that they too
grow in number as $\mu^n$.

\subsection{Bridges and the Hammersley--Welsh bound}\lbsect{HammWelsh}

For a self-avoiding walk $\omega$, denote by $\omega_1(i)$ the first spatial coordinate of $\omega(i)$.
\begin{definition}
\label{def:bridge}
An $n$-step \emph{bridge} is an $n$-step SAW $\omega$ such that
\begin{equation}
\lbeq{BridgeDefn}
\omega_1(0) < \omega_1(i) \leq \omega_1(n) \qquad\qquad\text{for $i=1,2,\dotsc,n$.}
\end{equation}
Let $b_n$ be the number of $n$-step bridges with $\omega(0)=0$ for $n > 1$, and $b_0=1$.
\end{definition}

While the number of self-avoiding walks is a \emph{sub}multiplicative sequence,
the number of bridges is \emph{super}multiplicative:
\begin{equation}
\lbeq{BridgeSupermult}
b_{n+m}\geq b_n b_m.
\end{equation}
Thus, applying \reflemma{Subadd} to $-\log b_n$, we obtain the existence
of the bridge growth constant $\mu_{\text{Bridge}}$ defined by
\begin{equation}
    \mu_{\text{Bridge}} = \lim_{n\to\infty} b_n^{1/n} = \sup_{n\geq 1} b_n^{1/n}.
\end{equation}
Using the trivial inequality $\mu_{\text{Bridge}}\leq\mu$ we conclude that
\begin{equation}
\lbeq{bnmuBound}
b_n\leq\mu_{\text{Bridge}}^n\leq \mu^n.
\end{equation}
\begin{definition}
An $n$-step \emph{half-space walk} is an $n$-step SAW $\omega$ with
\begin{equation}
\lbeq{HSWDefn}
\omega_1(0) < \omega_1(i)\qquad\qquad\text{for $i=1,2,\dotsc,n$.}
\end{equation}
Let $h_0=1$, and for $n \ge 1$,
let $h_n$ denote the number of $n$-step half-space walks with
$\omega(0)=0$.
\end{definition}
\begin{definition}
The \emph{span} of an $n$-step SAW $\omega$ is
\begin{equation}
%\label{}
    \max_{0\leq i\leq n} \omega_1(i)-\min_{0\leq i\leq n} \omega_1(i).
\end{equation}
Let $b_{n,A}$ be the number of $n$-step bridges with span $A$.
\end{definition}

We will use the following result on integer partitions
which dates back to 1917, due to
Hardy and Ramanujan \cite{HR17}.

\begin{theorem}
\lbthm{IntParts}
For an integer $A\geq 1$, let $P_D(A)$ denote the number of ways of
writing $A=A_1+\dotsb+A_k$ with $A_1>\dotsb>A_k\geq 1$, for any $k\geq 1$.
Then
\begin{equation}
\lbeq{IntPartAsymp}
\log P_D(A) \sim \pi \left( \frac{A}{3} \right)^{1/2}
\end{equation}
as $A\to\infty$.
\end{theorem}
\begin{prop}
\lbprop{HSWPartBound}
$h_n\leq P_D(n) b_n$ for all $n\geq 1$.
\end{prop}
\begin{proof}
Set $n_0=0$ and inductively define
\begin{equation}
%\label{}
A_{i+1}=\max_{j>n_i} (-1)^i (\omega_1(j)-\omega_1(n_i))
\end{equation}
and
\begin{equation}
%\label{}
n_{i+1}=\max\set{j>n_i : (-1)^i(\omega_1(j)-\omega_1(n_i))=A_{i+1}}.
\end{equation}
In words, $j=n_1$ maximises $\omega_1(j)$, $j=n_2$ minimises $\omega_1(j)$ for $j>n_1$, $n_3$ maximises $\omega_1(j)$ for $j>n_2$, and so on in an alternating pattern.  In addition $A_1=\omega_1(n_1)-\omega_1(n_0)$, $A_2=\omega_1(n_1)-\omega_1(n_2)$ and so on.  Moreover, the $n_i$ are chosen to be the last times these extrema are attained.
\begin{figure}[t]
\begin{center}
\input{HSWalkToBridge.pspdftex}
\end{center}
\caption{A half-space walk is decomposed into bridges, which are reflected to form a single bridge.}
\lbfig{hBoundByb}
\end{figure}

This procedure stops at some step $K\geq 1$ when $n_K=n$.
Since the $n_i$ are chosen maximal, it follows that $A_{i+1}<A_i$.
Note that $K=1$ if and only if $\omega$ is a bridge, and in that case $A_1$
is the span of $\omega$.
Let $h_n[a_1,\dotsc,a_k]$ denote the number of $n$-step half-space walks with $K=k$, $A_i=a_i$ for $i=1,\dotsc,k$.  We observe that
\begin{equation}
%\label{}
h_n[a_1,a_2,a_3,\dotsc,a_k] \leq h_n[a_1+a_2,a_3,\dotsc,a_n].
\end{equation}
To obtain this, reflect the part of the walk $(\omega(j))_{j\geq n_1}$ across the line $\omega_1 = A_1$; see \reffig{hBoundByb}.  Repeating this inequality gives
\begin{equation}
%\label{}
h_n[a_1,\dotsc,a_k] \leq h_n[a_1+\cdots+a_k] = b_{n,a_1+\cdots+a_k}.
\end{equation}
So we can bound
\begin{align}
h_n
&=
\sum_{k\geq 1} \, \sum_{a_1>\dotsb>a_k>0} h_n[a_1,\dotsc,a_k]
\notag\\
&\leq
\sum_{k\geq 1} \, \sum_{a_1>\dotsb>a_k>0} b_{n,a_1+\cdots+a_k}
\notag\\
&=
\sum_{A=1}^n P_D(A) b_{n,A}.
\end{align}
Bounding $P_D(A)$ by $P_D(n)$, we obtain $h_n\leq P_D(n)\sum\limits_{A=1}^n b_{n,A}=P_D(n) b_n$ as claimed.
\end{proof}

We can now prove the Hammersley--Welsh bound \refeq{HW}, from \cite{HW62}.
\begin{theorem}
\lbthm{HammWelsh}
Fix $B>\pi(\tfrac{2}{3})^{1/2}$.  Then there is $n_0=n_0(B)$ independent of the dimension $d\geq 2$ such that
\begin{equation}
\lbeq{HammWelshBound}
c_n \leq b_{n+1} e^{B\sqrt{n}} \leq \mu^{n+1} e^{B\sqrt{n}} \qquad\text{for $n\geq n_0$.}
\end{equation}
\end{theorem}
%
%We will prove this result.
Note that \refeq{HammWelshBound}, though an improvement over $c_n\leq \mu^n e^{o(n)}$ which follows from the definition
\refeq{mulambdaDefn} of $\mu$, is still much larger than the
predicted growth $c_n\sim A\mu^n n^{\gamma-1}$
from \refeq{cnlambdaAsymp}.
It is an open problem to improve \refthm{HammWelsh} in $d=2,3,4$
beyond the result of Kesten \cite{Kest63}
shown in \refeq{cnBounds}.
%Nonetheless, the Hammersley--Welsh argument involves a
%number of interesting ideas and is therefore
%the subject of the remainder of this lecture.

\begin{figure}[h]
\begin{center}
\input{HWcAndh.pspdftex}
\end{center}
\caption{The decomposition of a self-avoiding walk into two half-space walks.}
\lbfig{cBoundByh}
\end{figure}
\begin{proof}[Proof of \refthm{HammWelsh}]
We first prove
\begin{equation}
\lbeq{cnHalfSpaceBound}
c_n \leq \sum_{m=0}^n h_{n-m} h_{m+1},
\end{equation}
using the decomposition depicted in \reffig{cBoundByh}, as follows.
Given an $n$-step SAW $\omega$, let
\begin{equation}
%\label{}
x_1=\min_{0\leq i\leq n} \omega_1(i), \qquad m=\max\set{i : \omega_1(i)=x_1}.
\end{equation}
Write $e_1$ for the unit vector in the first coordinate direction of $\Z^d$.   Then (after translating by $\omega(m)$) the walk $(\omega(m),\omega(m+1),\dotsc,\omega(n))$ is an $(n-m)$-step half-space walk, and (after translating by $\omega(m)-e_1$) the walk $(\omega(m)-e_1,\omega(m),\omega(m-1),\linebreak\dotsc,\omega(1),\omega(0))$ is an $(m+1)$-step half-space walk.  This proves \refeq{cnHalfSpaceBound}.

Next, we apply \refprop{HSWPartBound} in \refeq{cnHalfSpaceBound} and use \refeq{BridgeSupermult} to get
\begin{align}
c_n
&\leq
\sum_{m=0}^n P_D(n-m) P_D(m+1) b_{n-m} b_{m+1}
\notag\\
&\leq
b_{n+1} \sum_{m=0}^n P_D(n-m) P_D(m+1).
\end{align}
Fix $B>B'>\pi (\tfrac{2}{3})^{1/2}$.  By \refthm{IntParts}, there is $K>0$ such that $P_D(A)\leq K \exp\left( B' (A/2)^{1/2} \right)$ and consequently
\begin{equation}
%\label{}
P_D(n-m) P_D(m+1) \leq K^2 \exp\left[ B'\left( \sqrt{\frac{n-m}{2}}+\sqrt{\frac{m+1}{2}} \right) \right] \! .
\end{equation}
The bound $x^{1/2}+y^{1/2} \leq (2x+2y)^{1/2}$ now gives
\begin{equation}
\lbeq{HWcnFinalBound}
c_n \leq (n+1)K^2 e^{B' \sqrt{n+1}} b_{n+1} \leq e^{B\sqrt{n}} b_{n+1}
\end{equation}
if $n\geq n_0(B)$.  By \refeq{bnmuBound}, the result follows.
\end{proof}
\begin{coro}
\lbcoro{bnLowerBound}
For $n\geq n_0(B)$,
\begin{equation}
\lbeq{bnLowerBound}
b_n \geq c_{n-1} e^{-B\sqrt{n-1}} \geq \mu^{n-1} e^{-B\sqrt{n-1}}.
\end{equation}
In particular, $b_n^{1/n}\to\mu$ and so $\mu_{\text{Bridge}}=\mu$.
\end{coro}
\begin{coro}
\lbcoro{chiBridgeBound}
Define the bridge generating function $\Bridge(z)=\sum_{n=0}^\infty b_n z^n$.  Then
\begin{equation}
%\label{}
\chi(z) \leq \frac{1}{z} e^{2(\Bridge(z)-1)}
\end{equation}
and in particular $\Bridge(1/\mu)=\infty$.
\end{coro}
\begin{proof}
In the proof of \refprop{HSWPartBound}, we decomposed a half-space walk
into subwalks on $[n_{i-1},n_i]$ for $i=1,\dotsc,K$.
Note that each such subwalk was in fact a bridge of span $A_i$.
With this observation, we conclude that
\begin{equation}
%\label{}
h_n \leq \sum_{k=1}^\infty \sum_{A_1>\dotsb>A_k}
\sum_{0=n_0<n_1<\dotsb<n_k=n} \prod_{i=1}^k b_{n_i-n_{i-1},A_i}
\end{equation}
(the second sum is over $A_1$ when $k=1$).
The choice of a descending sequence $A_1>\dotsb>A_k$ of arbitrary length is
equivalent to the choice of a subset of $\N$, so that taking generating functions gives
\begin{equation}
%\label{}
\sum_{n=0}^\infty h_n z^n \leq \prod_{A=1}^\infty \left( 1+\sum_{m=1}^\infty b_{m,A} z^m \right)  \! .
\end{equation}
Using the inequality $1+x\leq e^x$, we obtain
\begin{equation}
%\label{}
\sum_{n=0}^\infty h_n z^n
\leq
\exp \left( \sum_{A=1}^\infty \sum_{m=1}^\infty b_{m,A} z^m \right)
=
e^{\Bridge(z)-1}.
\end{equation}
Now using \refeq{cnHalfSpaceBound} gives
\begin{align}
\chi(z)
=
\sum_{n=0}^\infty c_n z^n
&\leq
\frac{1}{z}\sum_{n=0}^\infty \sum_{m=0}^n h_{n-m} z^{n-m} h_{m+1} z^{m+1}
\notag\\
&=
\frac{1}{z}\left( \sum_{n=0}^\infty h_n z^n \right)\left( \sum_{n=1}^\infty h_n z^n \right)
\notag\\
&\leq
\frac{1}{z}e^{2(\Bridge(z)-1)},
\end{align}
as required.
\end{proof}

\subsection{Self-avoiding polygons}\lbsect{SAPolygons}

A $2n$-step \emph{self-avoiding return} is a walk $\omega \in \Walks_{2n}$
with $\omega(2n)=\omega(0)=0$ and with $\omega(i)\ne \omega(j)$ for
distinct pairs $i,j$ other than the pair $0,2n$.  A
\emph{self-avoiding polygon} is a self-avoiding return with both the
orientation and the location of the origin forgotten.
Thus we can count self-avoiding polygons by counting self-avoiding returns
up to orientation and translation invariance, and their number is
\begin{equation}
  \lbeq{qnDefn}
  q_{2n} = \frac{2d c_{2n-1}(e_1)}{2\cdot 2n}, \qquad \text{$n\geq 2$,}
\end{equation}
where $e_1=(1,0,\dotsc,0)$ is the first standard basis vector.
Here, the $2$ in the denominator cancels the choice of orientation,
and the $2n$ cancels the choice of origin in the polygon.

\setlength{\unitlength}{0.4mm}
\begin{figure}[hbt]
\begin{center}
\begin{picture}(155,40)
\put(0,5){\begin{picture}(20,30)
\put(0,0){\line(1,0){20}}
\put(0,30){\line(1,0){10}}
\put(10,20){\line(1,0){10}}
\put(0,0){\line(0,1){30}}
\put(10,20){\line(0,1){10}}
\put(20,0){\line(0,1){20}}
\put(20,10){\circle*{1.5}}
\put(20,20){\circle*{1.5}}
\end{picture}}
\put(35,5){\begin{picture}(30,30)
\put(10,0){\line(1,0){20}}
\put(10,10){\line(1,0){10}}
\put(0,20){\line(1,0){20}}
\put(0,30){\line(1,0){30}}
\put(10,0){\line(0,1){10}}
\put(30,0){\line(0,1){30}}
\put(0,20){\line(0,1){10}}
\put(20,10){\line(0,1){10}}
\put(0,20){\circle*{1.5}}
\put(0,30){\circle*{1.5}}
\end{picture}}
\put(95,0){\begin{picture}(60,40)
\put(0,10){\line(1,0){20}}
\put(0,40){\line(1,0){10}}
\put(10,30){\line(1,0){50}}
\put(40,0){\line(1,0){20}}
\put(40,10){\line(1,0){10}}
\put(20,20){\line(1,0){30}}
\put(0,10){\line(0,1){30}}
\put(10,30){\line(0,1){10}}
\put(60,0){\line(0,1){30}}
\put(40,0){\line(0,1){10}}
\put(50,10){\line(0,1){10}}
\put(20,10){\line(0,1){10}}
\put(20,20){\circle*{1.5}}
\put(20,30){\circle*{1.5}}
\put(30,30){\circle*{1.5}}
\put(30,20){\circle*{1.5}}
\end{picture}}
\put(75,20){\vector(1,0){10}}
\end{picture}
\end{center}
\caption{\label{figpolycat}
Concatenation of a 10-step polygon and a 14-step polygon to
produce a 24-step polygon in ${\mathbb Z}^2$.
}
\end{figure}
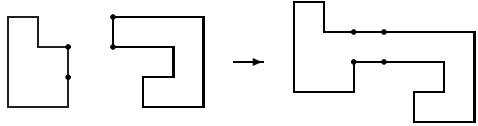

We first observe that
two self-avoiding polygons can be concatenated to form a larger
self-avoiding polygon.  Consider first the case of $d=2$.
The procedure is as in
Figure~\ref{figpolycat}, namely we
join a ``rightmost'' bond of one polygon to a ``leftmost'' bond of the other.
This shows that
for even integers $m,n\geq 4$, and for $d=2$, $q_m q_n\leq q_{m+n}$.
With a little thought (see \cite{MS93} for details), in general dimensions
$d \ge 2$ one obtains
\begin{equation}
\lbeq{qnSuperMult}
\frac{q_m q_n}{d-1}\leq q_{m+n},
\end{equation}
and if we set $q_2=1$
and make the easy observation that $q_n\leq q_{n+2}$,
then \refeq{qnSuperMult} holds for all even $m,n\geq 2$.
It follows from \refeq{qnSuperMult} that
\begin{equation}
\lbeq{muPolygonAndq2nBound}
q_{2n}^{1/2n}\to\mu_{\text{Polygon}}\leq \mu,\qquad q_{2n}\leq\mu_{\text{Polygon}}^{2n}\leq\mu^{2n} \quad\text{for all $n\geq 2$}.
\end{equation}

\begin{theorem}
\lbthm{PolygonLowerBound}
There is a constant $K=K(d)$ such that, for all $n\geq 1$,
\begin{equation}
%\label{}
c_{2n+1}(e_1) \geq \frac{K}{n^{d+2}} b_n^2.
\end{equation}
\end{theorem}
\begin{figure}
  \begin{center}
    \input{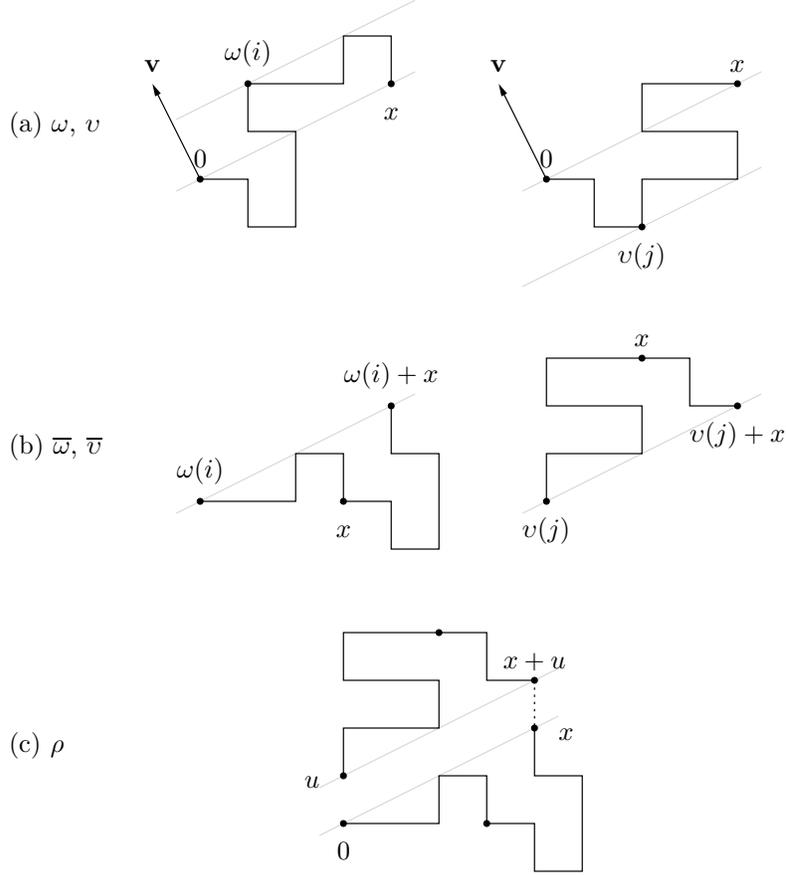}
  \end{center}
  \caption{Proof of Theorem~\ref{t:PolygonLowerBound}.
    Here $n=12$.  (a) The
    $n$-step bridges $\omega$ and $\upsilon$, and the vector ${\bf v}$.
    (b) The derived walks $\overline{\omega}$ and $\overline{\upsilon}$.
    (c) The $(2n+1)$-step walk $\rho$; here $u=(1,0)$.  The shaded lines are the hyperplanes orthogonal
    to ${\bf v}$.}
  \label{figure:PolygonLowerBound}
\end{figure}
\begin{proof}
We first show the inequality
\begin{equation}
\lbeq{bnxSquaredIneq}
\sum_{x\in \Z^d} b_n(x)^2 \leq 2d(n+1)^2 c_{2n+1}(e_1)
\end{equation}
where $b_n(x)$ denotes the number of $n$-step bridges ending at $x$.
The proof is illustrated in Figure~\ref{figure:PolygonLowerBound}.
Namely, given $n$-step bridges $\omega$ and $\upsilon$ with
$\omega(n)=\upsilon(n)=x\in\Z^d$, let ${\bf v}\in \R^d$ be some non-zero
vector orthogonal to $x$, and fix some unit direction $u\in\Z^d$
with $u\cdot {\bf v}>0$.  Let $i\in\set{0,1,\dotsc,n}$ be the
smallest index maximising $\omega(i)\cdot {\bf v}$ and
$j\in\set{0,1,\dotsc,n}$ the smallest index minimising
$\upsilon(j)\cdot {\bf v}$.  Split $\omega$ into the pieces
before and after $i$ and interchange them to produce a walk
$\overline{\omega}$, as in Figure~\ref{figure:PolygonLowerBound}(b).
Do the same for $\upsilon$ and $j$.  Finally combine
$\overline{\omega}$ and $\overline{\upsilon}$ with an inserted
step $u$ to produce a SAW $\rho$ with $\rho(2n+1)=u$, as in
Figure~\ref{figure:PolygonLowerBound}(c).  The resulting
map $(\omega,\upsilon)\mapsto(\rho,i,j)$ is one-to-one, which
proves \refeq{bnxSquaredIneq}.

Now, applying the Cauchy-Schwarz inequality to \refeq{bnxSquaredIneq} gives
\begin{align}
b_n^2
&=
\left( \sum_{x\in\Z^d} b_n(x) \indicator{b_n(x)\neq 0}\right)^2
\leq
\sum_{x\in\Z^d} b_n(x)^2 \sum_{x\in\Z^d}\indicator{b_n(x)\neq 0}
\notag\\
&\leq
n(2n+1)^{d-1}\sum_{x\in\Z^d} b_n(x)^2.
\end{align}
Thus $2d c_{2n+1}(e_1)\geq \frac{b_n^2}{n(n+1)^2(2n+1)^{d-1}}$, which completes the proof.
\end{proof}
\begin{coro}
There is a $C>0$ such that
\begin{equation}
%\label{}
\mu^{2n} e^{-C\sqrt{n}} \leq c_{2n+1}(e_1) \leq (n+1)\mu^{2n+2}.
\end{equation}
 In particular, $\mu_{\rm{Polygon}}=\mu$. % not \mu_{\text{Polygon}} which produces italics.
\end{coro}
\begin{proof}
The lower bound follows from \refthm{PolygonLowerBound}
and \refcoro{bnLowerBound}.  The upper bound follows from \refeq{qnDefn}
and \refeq{muPolygonAndq2nBound} (using $d\ge 2$).
\end{proof}
With a little more work, it can be shown that
for any fixed $x\neq 0$, $c_n(x)^{1/n}\to\mu$ as $n\to\infty$ along the
subsequence of integers whose parity agrees with
$\norm{x}_1$.  The details can be found in \cite{MS93}.
Thus the radius of convergence of the two-point function
$G_z(x)=\sum_{n=0}^\infty c_n(x)z^n$ is equal to $z_c=1/\mu$ for all $x$.

%    Absolute value notation
% \newcommand{\abs}[1]{\lvert#1\rvert}

%    Blank box placeholder for figures (to avoid requiring any
%    particular graphics capabilities for printing this document).
\newcommand{\blankbox}[2]{%
  \parbox{\columnwidth}{\centering
%    Set fboxsep to 0 so that the actual size of the box will match the
%    given measurements more closely.
    \setlength{\fboxsep}{0pt}%
    \fbox{\raisebox{0pt}[#2]{\hspace{#1}}}%
  }%
}

\section{The connective constant on the hexagonal lattice}
\lbsect{cchex}

Throughout this section,
we consider self-avoiding walks on the hexagonal lattice $\mathbb{H}$.
Our first and primary goal is to prove the following theorem from
\cite{D-CS10}.
The proof makes use of a certain observable of broader significance,
and following the proof we discuss this in the context of the
$O(n)$ models.

\begin{theorem}\label{thm:cchex}
For the hexagonal lattice ${\mathbb H}$,
\begin{equation}
%\label{}
\mu=\textstyle{\sqrt{2+\sqrt{2}}}.
\end{equation}
\end{theorem}

As a matter of convenience, we extend walks at their extremities
by two half-edges
in such a way that they start and end at \emph{mid-edges},
i.e., centres of edges of $\mathbb{H}$.
The set of mid-edges will be called $H$.
We position the hexagonal lattice $\mathbb{H}$ of mesh size 1 in $\mathbb{C}$
so that there exists a horizontal edge $e$ with mid-edge $a$ being $0$.
We now write $c_n$ for the number of $n$-step
SAWs on the hexagonal lattice $\mathbb{H}$ which start at $0$,
and $\chi(z) = \sum_{n=0}^\infty c_nz^n$ for the susceptibility.

We first point out that it suffices to count bridges.
On the hexagonal lattice, a bridge is defined by the following
adaptation of Definition~\ref{def:bridge}:
a \emph{bridge} on ${\mathbb H}$ is a SAW which never revisits
the vertical line through its starting point, never visits a vertical
line to the right of the vertical line through its endpoint, and moreover
starts and ends at the midpoint of a horizontal edge.
We now use $b_n$ to denote the number of $n$-step bridges on ${\mathbb H}$
which start at $0$.
It is straightforward to adapt the arguments used to prove
\refcoro{bnLowerBound} to the hexagonal lattice, leading to the conclusion
that $\mu_{\text{Bridge}}=\mu$ also on ${\mathbb H}$.  Thus it suffices to
show that
\begin{equation}
\lbeq{mubridgehex}
    \mu_{\text{Bridge}}=\textstyle{\sqrt{2+\sqrt{2}}}.
\end{equation}
Using notation which anticipates our conclusion but which should not
create confusion, we will write
\begin{equation}
    z_c = \frac{1}{\sqrt{2+\sqrt{2}}}.
\end{equation}
We also write
$B(z)=\sum_{n=0}^\infty b_nz^n$ for $z>0$.
To prove \refeq{mubridgehex}, it suffices to
prove that
$B(z_c)=\infty$ or $\chi(z_c)=\infty$, and that $B(z)<\infty$ whenever $z<z_c$.
This is what we will prove.

\subsection{The holomorphic observable}

The proof is based on a generalisation of the two-point function that
we call the \emph{holomorphic observable}.  In this section, we introduce
the holomorphic observable and prove its discrete analyticity.
Some preliminary definitions are required.

%In this section, we introduce the following observable, and prove its
%discrete analyticity. It allows us to use a counting argument in the
%next section. This is the main step of the proof.

A \emph{domain} $\Omega\subset H$ is
a union of all mid-edges emanating from a given
connected collection of
vertices $V(\Omega)$; see Figure~\ref{fig:SAWpicture}.
In other words, a mid-edge
$x$ belongs to $\Omega$ if at least one end-point of its associated
edge is in $V(\Omega)$.  The boundary $\partial \Omega$ consists of mid-edges
whose associated edge has exactly one endpoint
in $\Omega$. We further assume $\Omega$ to be simply connected,
i.e., having a connected complement.

\begin{figure}[h]
\includegraphics[width=1.00\hsize]{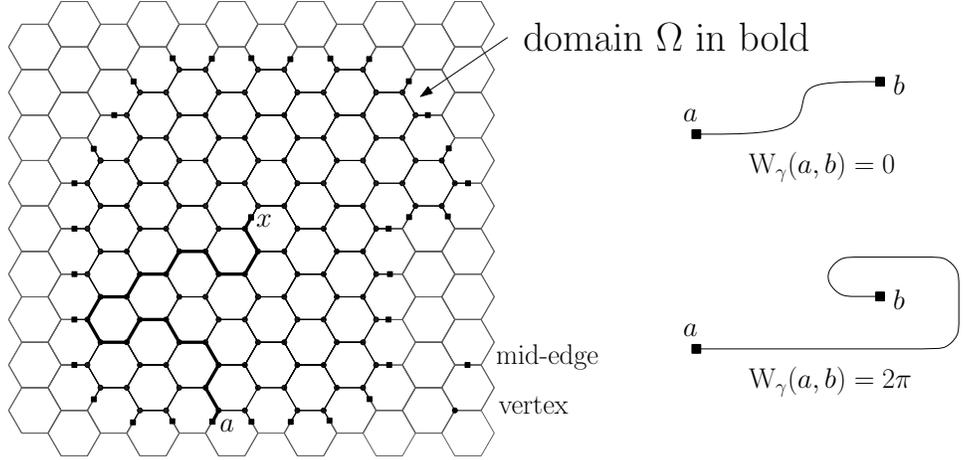}
\caption{Left: A domain $\Omega$ whose boundary mid-edges are pictured by
small black squares. Vertices of $V(\Omega)$ correspond to circles.
Right: Winding of a SAW $\omega$.}
\label{fig:SAWpicture}
\end{figure}

\begin{definition}
The winding ${\rm W}_\omega(a,b)$ of a SAW $\omega$
between mid-edges $a$ and $b$ (not necessarily the start and end of $\omega$)
is the total rotation in radians when $\omega$ is
traversed from $a$ to $b$; see Figure~\ref{fig:SAWpicture}.
\end{definition}

We write $\omega:a\rightarrow E$ if
a walk $\omega$ starts at mid-edge $a$ and ends at some mid-edge of $E\subset H$.
In the case where $E=\{b\}$, we simply write $\omega:a\rightarrow b$.
The \emph{length} $\ell(\omega)$ of the walk is the number of vertices
belonging to $\omega$.
The following definition provides a generalisation of the two-point function
$G_z(x)$.

\begin{definition}
Fix $a\in \partial \Omega$ and $\sigma \in \R$. For $x\in \Omega$  and $z \ge 0$,
the \emph{holomorphic observable} is defined to be
\begin{equation}\label{defparafermion}
    F_z(x)=\sum_{\omega\subset \Omega:\ a\rightarrow x} e^{-i\sigma {\rm W}_\omega(a,x)} z^{\ell(\omega)}.
\end{equation}
\end{definition}

In contrast to the two-point function,
the weights in the holomorphic observable need not be positive.
For the special case $z=z_c$ and $\sigma=\frac 58$,
$F_{z_c}$ satisfies the relation in the following lemma,
a relation which can be regarded as a weak form of discrete analyticity,
and which will be crucial in the rest of the proof.

\begin{lemma}\label{integral contour}
    If $z=z_c$ and $\sigma=\frac 58$, then, for every vertex $v\in V(\Omega)$,
    \begin{equation}\label{relation around vertex}
        (p-v)F_{z_c}(p)+(q-v)F_{z_c}(q)+(r-v)F_{z_c}(r)=0,
    \end{equation}
    where $p,q,r$ are the mid-edges of the three edges adjacent to $v$.
\end{lemma}

\begin{proof}
Let $z \ge 0$ and $\sigma \in \R$.  We will
specialise later to $z=z_c$ and $\sigma = \frac 58$.
We assume without loss of generality
that $p,q$ and $r$ are oriented counter-clockwise around $v$.
By definition, $(p-v)F_z(p)+(q-v)F_z(q)+(r-v)F_z(r)$ is a sum of contributions
$c(\omega)$ over all possible SAWs $\omega$ ending at $p,q$ or $r$.
For instance, if $\omega$ ends at the mid-edge $p$,
then its contribution will be
\begin{equation}
    c(\omega)=(p-v) e^{-i\sigma {\rm W}_{\omega}(a,p)}
    z^{\ell(\omega)}.
\end{equation}
The set of walks $\omega$ finishing at $p,q$ or $r$
can be partitioned into pairs and triplets of walks as depicted in
Figure~\ref{fig:pairs}, in the following way:
\begin{itemize}
\item If a SAW $\omega_1$ visits all three mid-edges $p,q,r$,
then the edges belonging to $\omega_1$ form a
SAW plus (up to a half-edge) a self-avoiding return from $v$ to $v$.
One can associate to $\omega_1$ the walk
$\omega_2$ passing through the same edges,
but traversing the return from $v$ to $v$ in the opposite direction. Thus,
walks visiting the three mid-edges can be grouped in pairs.

\item If a walk $\omega_1$ visits only one mid-edge, it can be
associated to two walks $\omega_2$ and $\omega_3$ that visit exactly
two mid-edges by prolonging the walk one step further (there are two
possible choices). The reverse is true: a walk visiting exactly two
mid-edges is naturally associated to a walk visiting only one mid-edge
by erasing the last step. Thus, walks visiting one or two mid-edges
can be grouped in triplets.
\end{itemize}
We will prove that when $\sigma = \frac 58$ and
$z=z_c$ the sum of contributions for each pair and each
triplet vanishes, and therefore the total sum is zero.

\begin{figure}[h]
    \begin{center}
      \includegraphics[width=1.00\hsize]{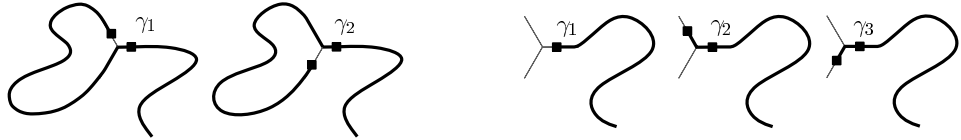}
     \end{center}
    \caption{Left: a pair of walks visiting the three mid-edges
    and matched together.
    Right: a triplet of walks, one visiting one mid-edge,
    the two others visiting two mid-edges, which are matched together.}
    \label{fig:pairs}
\end{figure}

Let $\omega_1$ and $\omega_2$ be two walks that are grouped as in the first case.
Without loss of generality, we assume that $\omega_1$ ends at $q$ and
$\omega_2$ ends at $r$. Note that $\omega_1$ and $\omega_2$ coincide
up to the mid-edge $p$ since $(\omega_1,\omega_2)$ are matched together.
Then
\begin{equation}
%\label{}
\ell(\omega_1)=\ell(\omega_2)\quad\quad \text{and}\quad \quad \left\{\ \substack{
    {\rm W}_{\omega_1}(a,q)={\rm W}_{\omega_1}(a,p)+{\rm W}_{\omega_1}(p,q)
    ={\rm W}_{\omega_1}(a,p)-\frac{4\pi}{3}\\ \ \\
    {\rm W}_{\omega_2}(a,r)={\rm W}_{\omega_2}(a,p)+{\rm W}_{\omega_2}(p,r)
    ={\rm W}_{\omega_1}(a,p)+\frac{4\pi}{3}.}\right.
\end{equation}
In evaluating the winding of $\omega_1$ between $p$ and $q$,
we used the fact that $a\in\partial \Omega$ and $\Omega$ is simply connected.
The term
$e^{-i\sigma {\rm W}_\omega(a,x)}$ gives a weight
$\lambda$ or $\bar{\lambda}$ per left or right turn of $\omega$,
where
\begin{equation}
    \lambda=\exp \left(- i \sigma \frac{\pi}{3}\right).
\end{equation}
Writing $j=e^{i 2\pi/3}$, we obtain
\begin{align}c(\omega_1)+c(\omega_2)
&=(q-v)  e^{- i\sigma {\rm W}_{\omega_1}(a,q)}z^{\ell(\omega_1)}
+(r-v) e^{- i\sigma {\rm W}_{\omega_2}(a,r)} z^{\ell(\omega_2)}
\nonumber
\\
&=(p-v) e^{- i\sigma {\rm W}_{\omega_1}(a,p)}z^{\ell(\omega_1)}
\left(j\bar{\lambda}^4+\bar{j}\lambda^4\right).
\end{align}
Now we set $\sigma = \frac 58$ so that $j\bar{\lambda}^4+\bar{j}\lambda^4
= 2\cos (\frac{3\pi}{2}) = 0$, and hence
\begin{align}c(\omega_1)+c(\omega_2) &=0.
\end{align}

Let $\omega_1,\omega_2,\omega_3$ be three walks matched as in the second case.
Without loss of generality, we assume that $\omega_1$ ends at $p$ and that
$\omega_2$ and $\omega_3$ extend $\omega_1$ to $q$ and $r$ respectively.
As before, we easily find that
\begin{equation}
%\label{}
\ell(\omega_2)=\ell(\omega_3)=\ell(\omega_1)+1\quad\quad\text{and}\quad\quad\left\{\ \substack{{\rm W}_{\omega_2}(a,r)={\rm W}_{\omega_2}(a,p)+{\rm W}_{\omega_2}(p,q)={\rm W}_{\omega_1}(a,p)-\frac{\pi}{3}\\ \ \\
    {\rm W}_{\omega_3}(a,r)={\rm W}_{\omega_3}(a,p)+{\rm W}_{\omega_3}(p,r)
    ={\rm W}_{\omega_1}(a,p)+\frac{\pi}{3},}\right.
\end{equation}
and thus
\begin{align}
c(\omega_1)+c(\omega_2)+c(\omega_3)
&=(p-v) e^{- i\sigma {\rm W}_{\omega_1}(a,p)}z^{\ell(\omega_1)}
\left(1+zj\bar{\lambda}+z\bar{j}\lambda\right).
\end{align}
Now we choose $z$ such that $1+z j\bar{\lambda}+z \bar{j}\lambda=0$.
Due to our choice $\sigma = \frac 58$, we have
$\lambda = \exp(- i\frac{5\pi}{24})$.
Thus we choose $z_c^{-1}=2\cos \frac{\pi}{8}=\sqrt{2+\sqrt2}$.

Now the desired identity \eqref{relation around vertex} follows
immediately by summing over all the pairs and triplets of walks.
\end{proof}

The last step of the proof of Lemma~\ref{integral contour}
is the \emph{only} place where the choice $z=z_c=1/\sqrt{2+\sqrt2}$
is used in the proof
of Theorem~\ref{thm:cchex}.

\subsection{Proof of Theorem~\ref{thm:cchex} completed.}
Now we will apply Lemma~\ref{integral contour} to prove
Theorem~\ref{thm:cchex}.

 We consider a vertical strip domain $S_T$ composed of the vertices
 of $T$ strips of hexagons,
 and its finite version $S_{T,L}$ cut at height $L$ at an angle of $\frac{\pi}{3}$;
 see Figure~\ref{fig:domain}.
% Then
%\begin{align*}
%    V(S_T)&=\{x\in V(\mathbb{H}):0\leq {\rm Re}(x)
%    \leq {\textstyle\frac{1}{2}}(3T+1) \},\\
%    V(S_{T,L})&=\{x\in V(S_T):|\sqrt 3 {\rm Im}(x)-{\rm Re}(x)|\leq 3L\}.
%\end{align*}
We denote the left and right boundaries of $S_{T}$ by $\alpha$ and
$\beta$, respectively,
and the top and bottom boundaries of $S_{T,L}$
by $\epsilon$ and $\bar{\epsilon}$, respectively.
We also introduce the positive quantities:
\begin{align}
    A_{T,L}(z)&\defeq\sum_{\substack{\omega\subset S_{T,L} :\  a\rightarrow\alpha\setminus \{a\}}}z^{\ell(\omega)},\\
    B_{T,L}(z)&\defeq\sum_{\substack{\omega\subset S_{T,L}:\ a\rightarrow\beta}}z^{\ell(\omega)},\\
    E_{T,L}(z)&\defeq\sum_{\substack{\omega\subset S_{T,L}:\ a\rightarrow\epsilon\cup\bar{\epsilon}}}z^{\ell(\omega)}.
\end{align}

\begin{figure}[h]
    \begin{center}
      \includegraphics[width=0.40\hsize]{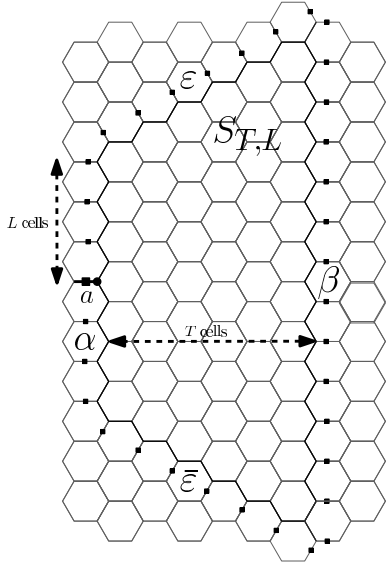}
    \end{center}
    \caption{Domain $S_{T,L}$ and boundary parts $\alpha$, $\beta$, $\epsilon$ and $\bar{\epsilon}$.}
    \label{fig:domain}
\end{figure}

\begin{lemma}
    For $z=z_c$,
    \begin{equation}\label{equation box}
        1=c_{\alpha}A_{T,L}(z_c)+B_{T,L}(z_c)+c_{\epsilon}E_{T,L}(z_c),
    \end{equation}
    where $c_{\alpha}=\cos \left(\frac {3\pi}{8}\right)$ and $c_{\epsilon}=\cos \left(\frac{\pi}{4}\right)$.
\end{lemma}

\begin{proof}
We fix $z=z_c$ and drop it from the notation.
We sum the relation \eqref{relation around vertex}
over all vertices in $V(S_{T,L})$.
Contributions at interior mid-edges vanish and we arrive at
\begin{equation}\label{sum}
    -\sum_{x\in \alpha}F(x)+\sum_{x\in \beta}F(x)
    +j\sum_{x\in \epsilon}F(x)+\bar{j}\sum_{x\in \bar{\epsilon}}F(x)
    =0.
\end{equation}
%Using the symmetry of the domain, we deduce that
%$F(\bar{x})=\bar{F}(x)$, where $\bar{x}$ is the image of $x$
%under reflection in the real axis.
The winding of any SAW
from $a$ to the bottom part of $\alpha$ is $-\pi$, while the winding
to the top part is $\pi$. Using this and symmetry,
together with the fact that the only SAW from $a$ to $a$ has
length $0$, we conclude that
\begin{equation}
\lbeq{Falpha}
    \sum_{x\in \alpha}F(x)=F(a)+\sum_{x\in \alpha\setminus \{a\}}F(x)
    =1+\frac{ e^{- i\sigma \pi}
    + e^{i\sigma \pi}}{2}A_{T,L}=1-c_{\alpha}A_{T,L}.
\end{equation}
Similarly, the winding from $a$ to any half-edge
in $\beta$, $\epsilon$ or $\bar{\epsilon}$ is respectively $0$,
$\frac {2\pi}3$ or $-\frac{2\pi}3$.  Therefore, again using symmetry,
\begin{equation}
\lbeq{Fbeta}
\sum_{x\in \beta}F(x)=B_{T,L},\quad
\quad j\sum_{x\in \epsilon}F(x)+\bar{j}\sum_{x\in \bar{\epsilon}}F(x)
%=\cos  \left(\frac\pi 4\right)\  E_{T,L}
=c_{\epsilon}E_{T,L}.
\end{equation}
The proof is completed by inserting \refeq{Falpha}--\refeq{Fbeta}
into \eqref{sum}.
\end{proof}

The sequences $(A_{T,L}(z))_{L>0}$ and $(B_{T,L}(z))_{L>0}$ are increasing
in $L$ and are bounded for $z\leq z_c$, thanks to \eqref{equation box}
and the monotonicity in $z$. Thus they have limits
\begin{align}
    A_T(z)&=\lim_{L\rightarrow \infty}A_{T,L}(z)
    =\sum_{\omega\subset S_T:\ a\rightarrow\alpha\setminus \{a\}} z^{\ell(\omega)},\\
    B_T(z)&=\lim_{L\rightarrow \infty}B_{T,L}(z)=\sum_{\omega\subset S_T:\ a\rightarrow \beta} z^{\ell(\omega)}.
\end{align}
When $z=z_c$, via \eqref{equation box} again, we conclude that
$(E_{T,L}(z_c))_{L>0}$ is decreasing
and converges to a limit $E_T(z_c)=\lim_{L\rightarrow \infty}E_{T,L}(z_c)$.
Thus, by \eqref{equation box},
\begin{equation}\label{equation strip}
    1=c_{\alpha}A_T(z_c) +B_T(z_c) +c_{\epsilon}E_T(z_c).
\end{equation}

\begin{proof}[Proof of Theorem~\ref{thm:cchex}]
The bridge generating function is given by $B(z)=\sum_{T=0}^\infty B_T(z)$.
Recall that it suffices to show that $B(z)<\infty$ for $z<z_c$,
and that $B(z_c)=\infty$ or $\chi(z_c)=\infty$.

We first assume $z<z_c$.
Since $B_T(z)$ involves only bridges of length at least $T$, it follows from
\eqref{equation strip} that
\begin{equation}
    B_T(z)\leq \left(\frac{z}{z_c}\right)^TB_T(z_c)
    \leq \left(\frac{z}{z_c}\right)^T,
\end{equation}
and hence $B(z)$ is finite since the right-hand side is summable.

It remains to prove that $B(z_c)=\infty$ or $\chi(z_c)=\infty$.
We do this by considering two
separate cases.
Suppose first that,
for some $T$, $E_T(z_c)>0$. As noted previously, $E_{T,L}(z_c)$
is decreasing in $L$. Therefore, as required,
\begin{equation}
    \chi(z_c)\geq \sum_{L=1}^\infty E_{T,L}(z_c)
    \geq \sum_{L=1}^\infty E_{T}(z_c)=\infty.
\end{equation}
It remains to consider the case that
 $E^{z_c}_T=0$ for every $T$.
In this case, \eqref{equation strip} simplifies to
\begin{equation}\label{infinite strip}
    1=c_{\alpha}A_T(z_c)+B_T(z_c).
\end{equation}
Observe that walks contributing to $A_{T+1}(z_c)$ but not to $A_T(z_c)$
must visit some vertex adjacent to the right edge of $S_{T+1}$.
Cutting such a walk at the first such point (and adding half-edges to the
two halves), we obtain two bridges of span $T+1$ in $S_{T+1}$. We conclude
from this that
\begin{equation}\label{rec relation}
    A_{T+1}(z_c)-A_T(z_c)\leq z_c\left(B_{T+1}(z_c)\right)^2.
\end{equation}
Combining \eqref{infinite strip} for $T$ and $T+1$ with \eqref{rec relation},
we can write
\begin{align}
0
&=[c_{\alpha}A_{T+1}(z_c)+B_{T+1}(z_c)]-[c_{\alpha}A_{T}(z_c)+B_{T}(z_c)]
\nonumber \\
&\leq c_{\alpha}z_c\left(B_{T+1}(z_c)\right)^2+B_{T+1}(z_c)-B_T(z_c),
\end{align}
so
\begin{equation}
c_{\alpha}z_c\left(B_{T+1}(z_c)\right)^2+B_{T+1}(z_c)\geq B_T(z_c).
\end{equation}
It is an easy exercise to verify by induction that
\begin{equation} % XXX
B_T(z_c)\geq \min \{B_1(z_c),1/(c_{\alpha}z_c)\}\frac {1}{T}
\end{equation}
for every $T\geq1$.  This implies, as required, that
\begin{equation}
B(z_c)\geq \sum_{T=1}^\infty B_T(z_c)=\infty.
\end{equation}
This completes the proof.
\end{proof}

\subsection{Conjecture~\ref{conjecture Lawler Schramm Werner}
and the holomorphic observable}

Recall the statement of Conjecture~\ref{conjecture Lawler Schramm Werner}.
When formulated on ${\mathbb H}$, this conjecture concerns
a simply connected domain $\Omega$ in the complex plane $\C$
with two points $a$ and $b$ on the boundary, with a discrete approximation
given by the largest
finite domain $\Omega_{\delta}$ of $\delta {\mathbb H}$ included in $\Omega$,
and with $a_\delta$ and
$b_\delta$ the closest vertices of $\delta {\mathbb H}$ to $a$ and $b$
respectively.
A probability measure
$\mathbb{P}_{z,\delta}$ is defined on the set of SAWs
$\omega$ between $a_\delta$ and $b_\delta$ that remain in
$\Omega_\delta$ by assigning to $\omega$ a  weight
proportional to $z_c^{\ell(\omega)}$.
 We obtain a random curve denoted $\omega_\delta$.
 We can also
define the observable in this context, and we denote it by
$F_\delta$.
Conjecture~\ref{conjecture Lawler Schramm Werner} then asserts that
the random curve $\omega_\delta$ converges to
$\SLE_{8/3}$ from $a$ and $b$ in the domain $\Omega$.

A possible approach to proving
Conjecture~\ref{conjecture Lawler Schramm Werner}
might be the following.
First, prove a precompactness result for self-avoiding walks.  Then,
by taking a subsequence, we could
assume that the curve $\gamma_\delta$ converges to a continuous curve
(in fact, the limiting object would need to be a Loewner chain,
see \cite{Beff11}).
The second step would consist in identifying the possible limits.
The holomorphic observable should play a crucial role in this step.
Indeed, if $F_\delta$ converges when rescaled to an explicit function,
one could use the \emph{martingale technique} introduced in \cite{Smir06}
to verify that the only possible limit is $\SLE_{8/3}$.

Regarding the convergence of $F_\delta$,
we first recall that in the discrete setting contour integrals
should be performed along dual edges.  For ${\mathbb H}$, the
dual edges form a triangular lattice, and
Lemma~\ref{integral contour} has the enlightening interpretation that
the contour integral vanishes
along any elementary dual triangle.
Any area enclosed
by a discrete closed dual contour is a union of elementary
triangles, and hence the
integral along any discrete closed
contour also vanishes.
This is a discrete analogue of Morera's theorem. It implies that if the
limit of $F_\delta$ (properly rescaled) exists and is continuous, then it
is automatically holomorphic.
By studying the boundary conditions, it is even
possible to identify the limit.  This leads to the following
conjecture, which is based on ideas in \cite{Smir06}.

\begin{conjecture}
Let $\Omega$ be a simply connected domain (not equal to $\mathbb{C}$),
let $z \in \Omega$,
and let $a,b$ be two distinct points on the boundary of $\Omega$.
We assume that the boundary of $\Omega$ is smooth near $b$.
For
$\delta>0$, let $F_\delta$ be the holomorphic observable
in the domain
$(\Omega_\delta,a_\delta,b_\delta)$ approximating $(\Omega,a,b)$,
and let $z_\delta$ be the closest point in $\Omega_\delta$ to $z$.
Then
\begin{equation}
\lbeq{Fconj}
\lim_{\delta\rightarrow 0} \frac {F_\delta(a_\delta,z_\delta)}
{F_\delta(a_\delta,b_\delta)}
=\left(\frac{\Phi'(z)}{\Phi'(b)}\right)^{5/8},
\end{equation}
where $\Phi$ is a conformal map from $\Omega$ to the upper half-plane
mapping a to $\infty$ and $b$ to 0.
\end{conjecture}

The right-hand side of \refeq{Fconj} is well-defined,
since the conformal map $\Phi$ is unique up to multiplication by
a real factor.

%Let $\Omega$ be a simply connected domain with $a,b \in \partial\Omega$.
%Then
%\begin{equation}
%    \lim_{\delta \to 0}
%    \delta^{-5/8}F_\delta(a_\delta,z_\delta)= K'(a,z)^{\frac58}
%    ,
%\end{equation}
%where $K$ is the Poisson kernel in $\Omega$, centred at $a$,
%and $z_\delta$ is the closest point to $z$ in $\Omega_\delta$.

\subsection{Loop models and holomorphic observables.}
\label{sec:loopholo}

The original motivation for the introduction of the
holomorphic observable stems from a
more general context, which we now discuss.
The \emph{loop} $O(n)$ \emph{model} is a lattice model on a domain $\Omega$.
We restrict attention in this discussion to the hexagonal lattice ${\mathbb H}$.
A configuration $\omega$ is a family of self-avoiding loops, and its
probability is proportional to $z^{\#\text{edges}}n^{\#\text{loops}}$.
The \emph{loop parameter} $n$ is taken in $[0,2]$. There are other variants of
the model; for instance, one can introduce an interface going from one point $a$ on the boundary to the inside, or one interface between two points of the boundary.
The case $n=1$ corresponds to the Ising model, while the case $n=0$ corresponds
to the self-avoiding walk
(when allowing one interface).
% while those with integers values of $n$
%are the \emph{high-temperature expansions} of the vector $O(n)$ models on
%the hexagonal lattice.

Fix $n\in[0,2]$. It is a
non-rigorous prediction of \cite{Nien82}
that the model has the following three phases distinguished
by the value of $z$:
\begin{itemize}
\item If $z<1/\sqrt{2+\sqrt{2-n}}$, the loops are sparse
(typically of logarithmic size in the size of the domain).
This phase is subcritical.
\item If $z=1/\sqrt{2+\sqrt{2-n}}$, the loops are dilute
(there are loops of the size of the domain  which
are typically separated be a distance of the size of the domain).
This phase is critical.
\item If $z>1/\sqrt{2+\sqrt{2-n}}$, the loops are dense
(there are loops of the size of the domain  which
are typically separated be a distance much smaller
than the size of the domain).
This phase is critical as well.
\end{itemize}

Consider the special case of the Ising model at its critical
value $z_c=1/\sqrt 3$. Let $E$ denote the set of configurations
consisting only of self-avoiding loops, and let
$E(a,x)$ denote the set of configurations
with self-avoiding loops plus an interface $\gamma$ from $a$ to $x$.
Then, ignoring the issue of boundary conditions,
the Ising spin-spin correlation is given in terms of the loop model by
\begin{equation}
\label{eq:disordered}
\left\langle \sigma(a)\sigma(x)\right\rangle
=\frac{\sum_{\omega\in E(a,x)} z_c^{\#\text{edges}} }
{\sum_{\omega\in E} z_c^{\#\text{edges}}}.
\end{equation}
A natural operation in physics consists in
flipping the sign of the coupling constant
of the Ising model along a path from $a$ to $x$, in such a way that a
monodromy is introduced: if we follow a path turning around $x$,
spins are reversed after one whole turn.
See, e.g., \cite{RC06}.
In terms of the loop representation, the spin-spin correlation
$\left\langle \sigma(a)\sigma(x)\right\rangle_{\text{monodromy}}$ in
this new Ising model is
\begin{equation}
\lbeq{newIsing}
    \left\langle \sigma(a)\sigma(x)\right\rangle_{\text{monodromy}}
    =\frac{\sum_{\omega\in E(a,x)}
    (-1)^{\#\text{turns of }\gamma\text{ around } x}z_c^{\#\text{edges}} }
    {\sum_{\omega\in E} z_c^{\#\text{edges}}}
\end{equation}
where $\gamma$ is the interface between $a$ and $x$.

The numerator of the right-hand side of \refeq{newIsing} can be rewritten as
\begin{equation}
%\label{}
\sum_{\omega\in E(a,x)}
 e^{- i\frac12 {\rm W}_{\gamma}(a,x)}z^{\# \text{edges}}n^{\# \text{loops}}
\end{equation}
with $n=1$.
This is of the same form as the holomorphic observable \eqref{defparafermion}.
With general values of $n$, and with the freedom to choose the
value of $\sigma \in [0,1]$, we obtain the observable
\begin{equation}
    F_z(x)\defeq\sum_{\omega\in E(a,x)}  e^{- i\sigma
    {\rm W}_{\gamma}(a,x)}z^{\# \text{edges}}n^{\# \text{loops}}.
\end{equation}
The values of $\sigma$ and $z$
need to be chosen according to the value of $n$. If $\sigma=\sigma(n)$
satisfies $2\cos [(1+2\sigma)2\pi/3]=-n$
and $z=z(n)=1/\sqrt{2+\sqrt{2-n}}$, then the proof of
Lemma~\ref{integral contour} can be modified to yield its conclusion
in this more general context.

To conclude this discussion,
consider the loop $O(n)$ model with a family of self-avoiding loops
and a single interface between two boundary points $a$ and $b$.
For $n=1$ and $z=1/\sqrt{3}$, it has been proved that
the interface converges to $\SLE_{3}$ \cite{CS10}.
For other values of $z$ and $n$,
the following behaviour is conjectured \cite{Smir06}.
\begin{conjecture}
\label{conj:interface}
Fix $n\in [0,2]$.  For $z=1/\sqrt{2+\sqrt{2-n}}$, the interface between $a$
and $b$ converges, as the lattice spacing goes to zero, to
\begin{equation}
%\label{}
\SLE_\kappa\quad\text{with}\quad\kappa
=\frac{4\pi}{2\pi-\arccos (-n/2)}.
\end{equation}
For $z>1/\sqrt{2+\sqrt{2-n}}$, the interface between
$a$ and $b$ converges, as the lattice spacing goes to zero,
to
\begin{equation}
%\label{}
\SLE_\kappa\quad\text{with}\quad\kappa=\frac{4\pi}{\arccos (-n/2)}.
\end{equation}
\end{conjecture}

\begin{figure}
\includegraphics[width=0.90\hsize]{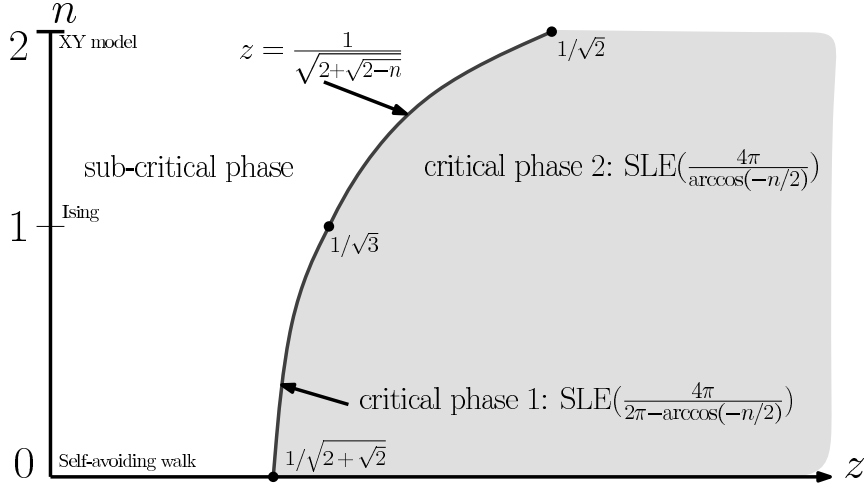}
\caption{Phase diagram for $O(n)$ models.}
\label{fig:O(n)}
\end{figure}

Conjecture~\ref{conj:interface} is summarised in Figure~\ref{fig:O(n)}.
The value of $\arccos$ is in $[0,\pi]$, so the first regime corresponds to
$\kappa \in [\frac 83, 4]$ and the second to $\kappa \in [4,8]$.
These two critical regimes do not belong to the same
universality class, in the sense that the scaling limit of
the interface is not the same. In particular,
since $\SLE_\kappa$ curves are simple
for $\kappa\leq 4$ but not for $\kappa>4$ (see \cite{Beff11}),
in the dilute phase
the interface is conjectured to be simple in the scaling limit,
but not in
the dense phase.  In addition, all the
 $\SLE_\kappa$ models for $\frac 83 \le \kappa \le 8$ arise in
 these $O(n)$ models. This rich
 behaviour is at the heart of the mathematical
 interest in $O(n)$ models. To prove
 the conjecture remains a major challenge in 2-dimensional statistical
 mechanics.

\section{The lace expansion}\lbsect{LaceExp}

\subsection{Main results}

In dimensions $d \ge 5$, it has been proved
that SAW has the same scaling behaviour as
SRW.  The following two theorems, due to Hara and Slade
\cite{HS92a,HS92b} and to Hara \cite{Hara08}, respectively, show
that the critical exponents $\gamma,\nu,\eta$ exist and take the
values $\gamma=1$, $\nu = \frac 12$, $\eta =0$, and that the scaling
limit is Brownian motion.
\begin{theorem}
  \label{thm:SAWLaceConclusion}
  Fix $d \ge 5$, and consider the nearest-neighbour SAW on
  $\Z^d$.
  There exist constants $A, D, \epsilon >0$ such that,
  as $n \to \infty$,
  \begin{align}
    c_n
    &=
    A\mu^n [1+O(n^{-\epsilon})],
    % \notag
    \\
    \E_n\abs{\omega(n)}^2
    &=
    D n [1+O(n^{-\epsilon})].
    \end{align}
    Also,
    \begin{align}
    \left( \frac{\omega(\floor{nt})}{\sqrt{Dn}} \right)_{t\geq 0}
    &\to (B_t)_{t\geq 0},
  \end{align}
  where $B_t$ denotes Brownian motion and the convergence is in distribution.
\end{theorem}
\begin{theorem}
  \label{thm:SAWLaceGreenConclusion}
  Fix $d \ge 5$, and consider the nearest-neighbour SAW on
  $\Z^d$.
  There are constants $c,\epsilon>0$ such that,
  as $x\to\infty$,
  \begin{equation}
    % \lbeq{}
    G_{z_c}(x)=\frac{c}{\abs{x}^{d-2}}
    \left[1+O\left( \abs{x}^{-\epsilon} \right) \right].
  \end{equation}
\end{theorem}

The proofs are based on the lace expansion, a technique that
was introduced by Brydges and Spencer \cite{BS85} to
study the weakly SAW in dimensions $d > 4$.
Since 1985, the method of lace expansion has been highly developed
and extended to several other models: percolation ($d > 6$), oriented percolation
($d > 4$ \emph{spatial} dimensions), the
contact process ($d > 4$), lattice trees and lattice animals ($d > 8$),
the Ising model ($d > 4$), and to random subgraphs of high-dimensional
transitive graphs such as the Boolean cube.
For a review and references, see \cite{Slad06}.

Versions of Theorems~\ref{thm:SAWLaceConclusion}--\ref{thm:SAWLaceGreenConclusion}
have been proved also for spread-out models; see \cite{MS93,HHS03}.
More recently, the above two theorems have been extended also to
study long-range SAWs based on simple random walks which take steps
of length $r$ with probability proportional to $r^{-d-\alpha}$ for some
$\alpha$.  For $\alpha \in (0,2)$, the upper critical dimension
(recall Section~\ref{sec:d4})
is reduced from $4$ to $2\alpha$, and the Brownian limit is replaced by
a stable law in dimensions $d>2\alpha$ \cite{Heyd09}.
Further results in this direction can be found in
\cite{HHS08,CS11}.

Our goal now is modest.
In this section, we will derive the lace expansion.
In \refSect{ConvLace}, we will sketch a proof
of how it can be used to prove that $\gamma=1$, in the sense that
\begin{equation}
\lbeq{chiAsymp}
\chi(z) \asymp \left( 1-z/z_c \right)^{-1} \qquad\text{as $z\increasesto z_c$,}
\end{equation}
both for the nearest-neighbour model with $d\geq d_0 \gg 4$,
and for the spread-out model with $L\geq L_0(d) \gg 1$ and any $d>4$.
Here, the notation $f(z) \asymp g(z)$ means that there exist positive
$c_1,c_2$ such that $c_1g(z) \le f(z) \le c_2g(z)$ holds uniformly in $z$.
The lower bound in \refeq{chiAsymp} holds in all dimensions
and follows immediately from the elementary
observation in \refeq{mulambdaDefn} that $c_n\geq \mu^n = z_c^{-n}$, since
\begin{equation}
  \chi(z)
  =
  \sum_{n=0}^\infty c_n z^n
  \geq
  \sum_{n=0}^\infty (\mu z)^n
  =
  \frac{1}{1-z/z_c}
  \lbeq{chiLowerBound}
\end{equation}
for $z<z_c$.
%The inequality \refeq{chiLowerBound} holds in all dimensions, but
%only for $d\geq 5$ it is relatively sharp.
It therefore suffices to prove that in high dimensions we have
the complementary upper bound
\begin{equation}
\lbeq{chiUpperBound}
\chi(z)\leq \frac{C}{1-z/z_c}
\end{equation}
for some finite constant $C$.

\subsection{The differential inequality for
\texorpdfstring{$\chi(z)$}{the susceptibility}}\lbssect{DiffIneq}

We prove \refeq{chiUpperBound} by means of a
\emph{differential inequality}---an inequality relating
$\tfrac{d}{dz}\chi(z)$ to $\chi(z)$.
The derivation of the differential inequality and its implication for
\refeq{chiUpperBound} first appeared in \cite{BFF84}.

The differential inequality is expressed in terms of the quantity
\begin{equation}
  \lbeq{BubbleDefn}
  \Bubble(z)=\sum_{x\in \Z^d} G_z(x)^2
\end{equation}
for $z\leq z_c$.
\refprop{TwoPointSubcritDecay} ensures that $\Bubble(z)$ is finite for $z<z_c$.
If we assume, as usual, that $G_{z_c} \sim c|x|^{-(d-2+\eta)}$, then
$\Bubble(z_c)$ will be finite precisely when $d>4-2\eta$.  With Fisher's
relation \refeq{FishersRel} and the predicted values of $\gamma$ and $\nu$
from \refeq{gammaPrediction} and \refeq{nuPrediction}, this inequality
can be expected to hold, and
correspondingly $\Bubble(z_c)<\infty$, only for $d>4$
(this is a prediction, not a theorem).
We refer to
$\Bubble(z)$ as the \emph{bubble diagram} because we
express \refeq{BubbleDefn} diagramatically as
\begin{equation}
  \lbeq{BubbleDiagram}
  \Bubble(z)=\raisebox{-0.26in}{\begin{picture}(0,0)%
\includegraphics{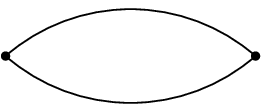}%
\end{picture}%
\setlength{\unitlength}{3947sp}%
\begingroup\makeatletter\ifx\SetFigFont\undefined%
\gdef\SetFigFont#1#2#3#4#5{%
  \reset@font\fontsize{#1}{#2pt}%
  \fontfamily{#3}\fontseries{#4}\fontshape{#5}%
  \selectfont}%
\fi\endgroup%
\begin{picture}(1252,589)(1175,-716)
\put(1201,-661){\makebox(0,0)[b]{\smash{{\SetFigFont{10}{12.0}{\familydefault}{\mddefault}{\updefault}{\color[rgb]{0,0,0}$0$}%
}}}}
\end{picture}%
}
  \, .
\end{equation}
In this diagram, each line represents a factor $G_z(x)$ and the unlabelled
vertex is summed over $x\in\Z^d$.  The condition that $\Bubble(z_c)<\infty$
will be referred to as the \emph{bubble condition}.

We now derive the differential inequality
\begin{equation}
  \lbeq{DiffIneq}
%  \tag{DI}
  \frac{d}{dz}\left( z \chi(z) \right) \geq \frac{\chi(z)^2}{\Bubble(z)}
%  \geq \frac{\chi(z)^2}{\Bubble(z_c)}
  .
\end{equation}
Assuming \refeq{DiffIneq}, we obtain \refeq{chiUpperBound} as
if we were solving a differential equation.
Namely, using the monotonicity of $\Bubble$, we first replace
$\Bubble(z)$ by $\Bubble(z_c)$ in \refeq{DiffIneq}.
We then rearrange and integrate from
$z$ to $z_c$, using the terminal value $\chi(z_c)=\infty$ from
\refeq{chiLowerBound}, to obtain
\begin{align}
  \frac{1}{z^2 \chi(z)^2} \frac{d}{dz}\left( z \chi(z) \right)
  &\geq
  \frac{1}{z^2 \Bubble(z_c)}
  \notag\\
  -\frac{d}{dz}\left( \frac{1}{z\chi(z)} \right)
  &\geq
  \frac{d}{dz}\left( \frac{-1}{z \Bubble(z_c)} \right)
  \notag\\
  -0+\frac{1}{z\chi(z)}
  &\geq
  \frac{1}{\Bubble(z_c)} \left( -\frac{1}{z_c} + \frac{1}{z} \right)
  \notag\\
  \frac{\Bubble(z_c)}{1-z/z_c}
  &\geq
  \chi(z).
\end{align}
Thus we have reduced the proof of \refeq{chiAsymp}
to verifying \refeq{DiffIneq} and showing
that $\Bubble(z_c)<\infty$ in high dimensions.
We will prove \refeq{DiffIneq} now, and in \refSect{ConvLace}
we will sketch the proof of the bubble condition in high dimensions.

We will use diagrams to derive \refeq{DiffIneq}.  A proof using more
conventional mathematical notation can be found, e.g.,\ in \cite{Slad06}.
In the diagrams in the next two paragraphs, each
dot denotes a point in $\Z^d$, and if a dot is unlabelled then
it is summed over all points in $\Z^d$.
Each arc (or line) in a diagram represents
a generating function for a SAW connecting the endpoints.
At times SAWs corresponding to distinct lines must be mutually-avoiding.
We will indicate this condition by labelling diagram lines and listing
in groups those that mutually avoid.

With these conventions, we can describe the two-point function and the
susceptibility succinctly by
\begin{align}
  \lbeq{GchiAsDiagrams}
  G_z(x)=
  \raisebox{-0.2in}{\input{Gzofx.pspdftex}} \, ,
  &&\chi(z)=
  \raisebox{-0.2in}{\begin{picture}(0,0)%
\includegraphics{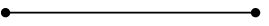}%
\end{picture}%
\setlength{\unitlength}{3947sp}%
\begingroup\makeatletter\ifx\SetFigFont\undefined%
\gdef\SetFigFont#1#2#3#4#5{%
  \reset@font\fontsize{#1}{#2pt}%
  \fontfamily{#3}\fontseries{#4}\fontshape{#5}%
  \selectfont}%
\fi\endgroup%
\begin{picture}(1252,305)(1175,-641)
\put(1201,-586){\makebox(0,0)[b]{\smash{{\SetFigFont{10}{12.0}{\familydefault}{\mddefault}{\updefault}{\color[rgb]{0,0,0}$0$}%
}}}}
\end{picture}%
} \, .
\end{align}
In order to obtain \refeq{DiffIneq}, let us consider
$Q(z)=\frac{d}{dz}(z \chi(z))$.  Note that $Q(z)$ can be regarded as
the generating function for
SAWs weighted by the number of vertices visited in the walk.
We represent this diagrammatically as:
\begin{equation}
  \lbeq{QDiagram}
  Q(z)=\sum_{n=0}^\infty (n+1)c_n z^n=
  \raisebox{-0.33in}{\input{Qofz.pspdftex}} .
\end{equation}
In \refeq{QDiagram}, each segment represents a SAW path, and the
notation [12] indicates that SAWs 1 and 2 must be mutually avoiding,
apart from one shared vertex.

We apply inclusion-exclusion to \refeq{QDiagram},
first summing over all pairs of SAWs, mutually avoiding or not,
and then subtracting configurations where SAWs 1 and 2 intersect.
We parametrise the subtracted term according to the \emph{last} intersection
point along the second walk.  Renumbering the subwalks, we have
\begin{equation}
  \lbeq{QIncExcl}
  Q(z)
  =
  \raisebox{-0.2in}{\begin{picture}(0,0)%
\includegraphics{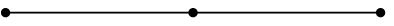}%
\end{picture}%
\setlength{\unitlength}{3947sp}%
\begingroup\makeatletter\ifx\SetFigFont\undefined%
\gdef\SetFigFont#1#2#3#4#5{%
  \reset@font\fontsize{#1}{#2pt}%
  \fontfamily{#3}\fontseries{#4}\fontshape{#5}%
  \selectfont}%
\fi\endgroup%
\begin{picture}(1852,305)(1175,-641)
\put(1201,-586){\makebox(0,0)[b]{\smash{{\SetFigFont{10}{12.0}{\familydefault}{\mddefault}{\updefault}{\color[rgb]{0,0,0}$0$}%
}}}}
\end{picture}%
}
  \: - \:
  \raisebox{-0.39in}{\input{QExcl.pspdftex}}
\end{equation}
where the notation $[124][34]$ means that
walks 1, 2 and 4 must be mutually avoiding except at the endpoints, whereas
walk 3 must avoid walk 4 but is allowed to intersect walks 1 and 2.
Also, SAWs 2 and 3 must each take at least one step.
We obtain an inequality by relaxing the avoidance pattern to [14],
keeping the requirement that the walk 23 should be non-empty:
\begin{align}
  Q(z)
  &\geq
  \raisebox{-0.2in}{\begin{picture}(0,0)%
\includegraphics{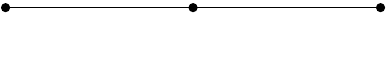}%
\end{picture}%
\setlength{\unitlength}{3947sp}%
\begingroup\makeatletter\ifx\SetFigFont\undefined%
\gdef\SetFigFont#1#2#3#4#5{%
  \reset@font\fontsize{#1}{#2pt}%
  \fontfamily{#3}\fontseries{#4}\fontshape{#5}%
  \selectfont}%
\fi\endgroup%
\begin{picture}(1852,305)(1175,-641)
\put(1201,-586){\makebox(0,0)[b]{\smash{{\SetFigFont{10}{12.0}{\familydefault}{\mddefault}{\updefault}{\color[rgb]{0,0,0}$0$}%
}}}}
\end{picture}%
}
  -
  \raisebox{-0.39in}{\input{QExclRelax.pspdftex}}
  \notag\\
  &=
\chi(z)^2-Q(z)(\Bubble(z)-1).
\lbeq{QIneq}
\end{align}
Rearranging gives the inequality \refeq{DiffIneq}.

\subsection{The lace expansion by inclusion-exclusion}\lbssect{LEInclExcl}

The proof of the bubble condition is based on the lace expansion.
The original derivation of the lace expansion by Brydges and Spencer
\cite{BS85} made
use of a certain graphical construction called a \emph{lace}.
Later, it was realised that repeated inclusion-exclusion leads to the same
expansion \cite{Slad91}.  We present the inclusion-exclusion approach
now; the approach via laces is treated in the problems of
\refSect{Tut2}.  The underlying graph plays little role in the derivation,
and the following discussion pertains to either nearest-neighbour
or spread-out SAWs.  Indeed, with minor modifications, the discussion
also applies on general graphs \cite{CS09}.

We use the convolution $(f  *  g)(x)=\sum_{y\in\Z^d} f(y) g(x-y)$
of two functions $f,g$ on $\Z^d$.
The lace expansion gives rise to a formula for $c_n(x)$, for $n \ge 1$,
of the form
\begin{align}
\lbeq{cnLaceExpansion}
\begin{split}
c_n(x)
&=
(c_1 * c_{n-1})(x) + \sum_{m=2}^n (\pi_m * c_{n-m})(x)
\\
&=
\sum_{y\in\Z^d} c_1(y) c_{n-1}(x-y) + \sum_{m=2}^n \sum_{y\in\Z^d} \pi_m(y) c_{n-m}(x-y)
,
\end{split}
\end{align}
in which the coefficients
$\pi_m(y)$ are certain combinatorial integers that we
will define below.
Note that the identity \refeq{cnLaceExpansion} would hold for
SRW with $\pi\equiv 0$.  The quantity $\pi_m(y)$ can therefore be understood as
a correction factor determining to what degree SAWs fail
to behave like SRWs.  In this sense, the lace expansion studies the
SAW as a perturbation of the SRW.
%This statement must be interpreted heuristically, since the values
%$\pi_m(y)$ are uniquely determined by the values $c_n(x)$ (see equations
%\refeq{GhatRelation}--\refeq{GhatFormula}).  The usefulness of
%\refeq{cnLaceExpansion} will depend upon (a) our ability to describe and
%bound the values $\pi_m(y)$; and (b) the ``smallness'' of the $\pi_m(y)$,
%in a suitable sense.

%In addition to the inclusion-exclusion analysis presented here, the original
%derivation of the lace expansion by Brydges, Spencer (1985) \cite{BS85} gives
%explicit algebraic formulas for the terms $\pi_m^{(N)}(y)$ in terms of
%``laces''; see Problem~\ref{problem:J-laces}.

Our starting point is similar to that of the derivation of the differential
inequality \refeq{DiffIneq}, but now we will work with identities rather
than inequalities.  Also, rather than working with generating functions,
we will work instead with walks with a fixed number of steps and without
factors $z$: diagrams now arise from walks of fixed length.
We begin by dividing an $n$-step SAW ($n\geq 1$) into its first step and
the remainder of the walk.  Because of self-avoidance, these two parts
must be mutually avoiding, and we perform inclusion-exclusion on this
condition:
\begin{align}
\raisebox{-0.2in}{\input{Gzofx.pspdftex}}
&=
\raisebox{-0.33in}{\input{c1ConvcAvoid.pspdftex}}
\notag\\
&=
\raisebox{-0.2in}{\input{c1Convc.pspdftex}}
\: - \:
\raisebox{-0.33in}{\input{LoopConvcAvoid.pspdftex}}
\lbeq{LaceExpStep1}
\end{align}
where \begin{picture}(0,0)%
\includegraphics{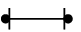}%
\end{picture}%
\setlength{\unitlength}{3947sp}%
\begingroup\makeatletter\ifx\SetFigFont\undefined%
\gdef\SetFigFont#1#2#3#4#5{%
  \reset@font\fontsize{#1}{#2pt}%
  \fontfamily{#3}\fontseries{#4}\fontshape{#5}%
  \selectfont}%
\fi\endgroup%
\begin{picture}(352,126)(1025,-724)
\end{picture}%
 indicates a single step.  In more detail, the
first term on the right-hand side represents $(c_1*c_{n-1})(x)$, and
the subtracted term represents the number of $n$-step walks from $0$ to $x$
which are self-avoiding apart from a single required return to $0$.
We again perform inclusion-exclusion,
first on the avoidance $[12]$ in
the second term of \refeq{LaceExpStep1}
(noting now the \emph{first} time along walk 2 that walk 1 is hit):
\begin{align}
\raisebox{-0.33in}{\input{LoopConvcAvoid.pspdftex}}
&=
\raisebox{-0.2in}{\input{LoopConvc.pspdftex}}
\: - \:
\raisebox{-0.33in}{\input{ThetaConvcAvoid.pspdftex}}
\lbeq{LaceExpStep2}
\intertext{and then on the avoidance $[34]$ in the second term of
\refeq{LaceExpStep2} (noting the \emph{first} time along walk 4 that walk 3 is hit):}
\raisebox{-0.33in}{\input{ThetaConvcAvoid.pspdftex}}
&=
\raisebox{-0.33in}{\input{ThetaConvc.pspdftex}}
\: - \:
\raisebox{-0.33in}{\input{ThirdLaceConvcAvoid.pspdftex}}.
\end{align}
The process is continued recursively.   Since the total number $n$
of steps is finite, the above process
terminates after a finite number of applications of inclusion-exclusion,
because each application uses at least one step.  The result is
\begin{align}
\raisebox{-0.2in}{\input{Gzofx.pspdftex}}
&=
\raisebox{-0.2in}{\input{c1Convc.pspdftex}}
\: - \:
\raisebox{-0.2in}{\input{LoopConvc.pspdftex}}
\notag\\
&\quad+ \:
\raisebox{-0.33in}{\input{ThetaConvc.pspdftex}}
\: - \:
\raisebox{-0.33in}{\input{ThirdLaceConvc.pspdftex}}
\: +\dots
\lbeq{lacepictures}
\end{align}

The first term on the right-hand side is just $(c_1*c_{n-1})(x)$.
In the remaining terms on the right-hand side, we regard the line ending
at $x$ as having length $n-m$, so that $m$ steps are used by the other
lines.  We also regard the line ending at $x$ as starting at $y$.
A crucial fact is that the line ending at $x$ has no dependence on the
other lines, so it represents $c_{n-m}(x-y)$.
Thus, if we define the coefficients $\pi_m(y)$ as
\begin{align}
\pi_m(y)
&=
- \:
\raisebox{-0.2in}{\begin{picture}(0,0)%
\includegraphics{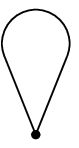}%
\end{picture}%
\setlength{\unitlength}{3947sp}%
\begingroup\makeatletter\ifx\SetFigFont\undefined%
\gdef\SetFigFont#1#2#3#4#5{%
  \reset@font\fontsize{#1}{#2pt}%
  \fontfamily{#3}\fontseries{#4}\fontshape{#5}%
  \selectfont}%
\fi\endgroup%
\begin{picture}(343,889)(1029,-641)
\put(1201,-586){\makebox(0,0)[b]{\smash{{\SetFigFont{10}{12.0}{\familydefault}{\mddefault}{\updefault}{\color[rgb]{0,0,0}$0$}%
}}}}
\end{picture}%
}
\delta_{0y}
+ \:
\raisebox{-0.33in}{\input{Theta.pspdftex}}
\: - \:
\raisebox{-0.33in}{\input{ThirdLace.pspdftex}}
\: +\dotsc
\notag\\
&=
\sum_{N=1}^\infty (-1)^N \pi_m^{(N)}(y),
\lbeq{pimy}
\end{align}
where $\delta_{xy}=\indicator{x=y}$ denotes the Kronecker delta, then \refeq{lacepictures} becomes \refeq{cnLaceExpansion}, namely
\begin{align}
c_n(x)
&=
\sum_{y\in\Z^d} c_1(y) c_{n-1}(x-y) + \sum_{m=2}^n \sum_{y\in\Z^d} \pi_m(y) c_{n-m}(x-y)
.
\end{align}

By definition, $\pi_m^{(1)}(y)$ counts the number of $m$-step
self-avoiding returns if $y=0$, and is otherwise $0$.  Also,
$\pi_m^{(2)}(y)$ counts the number of $m$-step ``$\theta$-diagrams''
with vertices $0$ and $y$, i.e., the number of $m$-step walks which
start at zero, end at $y$, and are self-avoiding apart from a required
return to $0$ and a visit to $y$ before terminating at $y$.
With more attention to the inclusion-exclusion procedure, it can be
seen that in the three diagrams on the right-hand side of \refeq{pimy}
all the individual subwalks must have length at least $1$ except for
subwalk 3 of the third term which may have length $0$.
As noted above, the inclusion-exclusion procedure terminates after
a finite number of steps, so the terms in the series \refeq{pimy}
are eventually all zero, but as $m$ increases more and more terms are
non-zero.
If the diagrams make you uncomfortable, formulas for $\pi_m(y)$ are
given in \refSect{Tut2}.  This completes the derivation of the
lace expansion.

Our next task is to relate $\pi_m(y)$ to our goal of proving the bubble
condition.
Equation \refeq{cnLaceExpansion} contains two convolutions:
a convolution in space given by the sum over $y$, and a convolution
in time given by the sum over $m$.  To eliminate these and facilitate
analysis, we pass to
generating functions and Fourier transforms.  By definition of the two-point
function,
\begin{align}
G_z(x)
&=
\sum_{n=0}^\infty c_n(x) z^n = \delta_{0x} + \sum_{n=1}^\infty c_n(x) z^n,
\end{align}
and we define
\begin{align}
\Pi_z(x)
&=
\sum_{m=2}^\infty \pi_m(x) z^m.
\end{align}
From \refeq{cnLaceExpansion}, we obtain
\begin{align}
\lbeq{GzxRelation}
\begin{split}
G_z(x)
&=
\delta_{0x}+\sum_{y\in\Z^d} z c_1(y) G_z(x-y) + \sum_{y\in\Z^d} \Pi_z(y) G_z(x-y)
\\
&=
\delta_{0x}+z (c_1  *  G_z)(x) + (\Pi_z  *  G_z)(x).
\end{split}
\end{align}
Given an
absolutely summable function $f: \Z^d \to \C$, we write its Fourier transform as
\begin{equation}
%\label{}
\hat{f}(k)=\sum_{x\in\Z^d} f(x) e^{ik\cdot x},
\end{equation}
with $k =(k_1,\ldots,k_d)\in [-\pi,\pi]^d$.
Then \refeq{GzxRelation} gives
\begin{equation}
\lbeq{GhatRelation}
\hat{G}_z(k) = 1 + z\hat{c}_1(k)\hat{G}_z(k) + \hat{\Pi}_z(k)\hat{G}_z(k).
\end{equation}
We solve for $\hat{G}_z(k)$ to obtain
\begin{equation}
\lbeq{GhatFormula}
\hat{G}_z(k) = \frac{1}{1-z\hat{c}_1(k)-\hat{\Pi}_z(k)}.
\end{equation}

It is convenient to express $c_1(y)$ in terms of the probability distribution for the steps of the corresponding SRW model:
\begin{equation}
%\label{}
D(y) \defeq \frac{c_1(y)}{\abs{\Omega}}, \qquad\qquad
\hat{c}_1(k)=\abs{\Omega} \hat{D}(k),
\end{equation}
where $\abs{\Omega}$ denotes the cardinality of either option for
the set $\Omega$
defined in \refeq{nnso}.
For the nearest-neighbour model, $\abs{\Omega}=2d$ and
\begin{equation}
%\label{}
\hat{D}(k)=\frac{1}{2d}\sum_{j=1}^d \bigl( e^{i k_j}+e^{-i k_j} \bigr) =\frac{1}{d}\sum_{j=1}^d \cos k_j.
\end{equation}
To simplify the notation, we define $\hat F_z(k)$ by
\begin{equation}
\lbeq{FhatDefn}
\hat{G}_z(k)=\frac{1}{1-z\abs{\Omega}\hat{D}(k)-\hat{\Pi}_z(k)}
\defeq \frac{1}{\hat{F}_z(k)}.
\end{equation}
Notice that $\hat{G}_z(0)=\sum_{x\in\Z^d}\sum_{n=0}^\infty c_n(x) z^n =\chi(z)$,
so that $\hat{G}_z(0)$ will have a singularity at $z=z_c$.  To emphasise this,
we will write
%
%\begin{align}
%\hat{G}_z(k)
%&=
%\frac{1}{\hat{F}_z(0)+\bigl(\hat{F}_z(k)-\hat{F}_z(0)\bigr)}
%\notag\\
%&=
%\frac{1}{\chi(z)^{-1} + z\abs{\Omega}\bigl[1-\hat{D}(k)\bigr] + \bigl[\hat{\Pi}_z(0)-\hat{\Pi}_z(k)\bigr]}.
%\lbeq{GhatFormula2}
%\end{align}
%
%
\begin{align}
\hat{F}_z(k)
&=
\hat{F}_z(0)+\bigl(\hat{F}_z(k)-\hat{F}_z(0)\bigr)
\notag\\
&=
\chi(z)^{-1} + z\abs{\Omega}\bigl(1-\hat{D}(k)\bigr) + \bigl(\hat{\Pi}_z(0)-\hat{\Pi}_z(k)\bigr).
\lbeq{FhatFormula}
\end{align}

Now we can make contact with our goal of proving the bubble condition.
By Parseval's relation,
\begin{equation}
\lbeq{BzIntegral}
\Bubble(z)=\sum_{x\in\Z^d} G_z(x)^2
= \int_{[-\pi,\pi]^d}|\hat{G}_z(k)|^2 \frac{\ddk}{(2\pi)^d}
\end{equation}
(this includes the case where one side of the equality, and hence both,
are infinite).  The issue of whether $\Bubble(z_c)<\infty$ or not
boils down to the question of whether the singularity of the integrand
is integrable or not,
so we will need to understand the asymptotics of the terms
in \refeq{FhatFormula} as $k\to 0$ and $z \nearrow z_c$.
In principle there could be other singularities when $z=z_c$,
 but for the nearest-neighbour
and spread-out models $1-\hat{D}(k) > 0$ for non-zero $k$, and one of the
goals of the analysis will be to prove that the term
$\hat{\Pi}_z(0)-\hat{\Pi}_z(k)$ cannot create a cancellation.

The term $1-\hat{D}(k)$ is explicit, and for the nearest-neighbour model
has asymptotic behaviour
\begin{equation}
%\label{}
1-\hat{D}(k)=\frac{1}{d}\sum_{j=1}^d (1-\cos k_j) \sim \frac{\abs{k}^2}{2d}
\end{equation}
as $k\to 0$.
We need to see that the term
$\hat{\Pi}_z(0)-\hat{\Pi}_z(k)$ is relatively small in high dimensions.
By symmetry, we can write this term as
\begin{align}
\hat{\Pi}_z(0)-\hat{\Pi}_z(k)
&=
\sum_{x\in\Z^d} (1-e^{ik\cdot x})\Pi_z(x)
=
\sum_{x\in\Z^d} (1-\cos k\cdot x) \Pi_z(x).
\lbeq{PihatDifference}
\end{align}

%For the corresponding integral,
%%
%\begin{equation}
%\lbeq{d4Integral}
%\int_{[-\pi,\pi]^d} \left( \abs{k}^{-2} \right)^2 \ddk < \infty
%\quad \text{is equivalent to} \quad
%d>4.
%\end{equation}
%%
%Because of \refeq{d4Integral}, we expect that the lace expansion for SAW will \emph{converge} when $d>4$, and give no information otherwise.  \JG{Brief explanation of what convergence means for the lace expansion?}

Finally, we note that the equation $\chi(z_c)=\infty$ can be rewritten as
$0=\chi(z_c)^{-1}=1-z_c \abs{\Omega} - \hat{\Pi}_{z_c}(0)$, from which
we see that the critical point $z_c$ is given implicitly by
\begin{equation}
\lbeq{zcEquation}
z_c = \frac{1}{\abs{\Omega}} \bigl(1-\hat{\Pi}_{z_c}(0)\bigr).
\end{equation}
This equation has been the starting point for the study of $z_c$,
in particular
for the derivation of the $1/d$ expansion for the connective constant
discussed in \refSSect{dexpansion}.  Problem~\ref{problem:1/d} below
indicates how the first terms are obtained.

\subsection{Tutorial}
\lbsect{Tut2}

These problems develop the original derivation of the lace expansion by
Brydges and Spencer \cite{BS85}.
All this material can also be found in \cite{Slad06}.

We require a notion of graphs on integer intervals, and
connectivity of these graphs. We emphasise in advance that the notion of
connectivity is
\emph{not} the usual graph theoretic one, but that it is the right notion
in this context.

\begin{definition}
  (i)
      Let $I=[a,b]$ be an interval of non-negative integers.
      An \emph{edge} is a pair $st = \{s,t\}$ with $s,t\in\Z$
      and $a\leq s<t\leq b$.
      A \emph{graph} on $[a,b]$ is a set of edges.
      We denote the set of all graphs on $[a,b]$ by $\mathcal{B}[a,b]$.

  (ii)
      A graph $\Gamma \in \mathcal{B}[a,b]$ is \emph{connected} if $a,b$ are endpoints of edges, and if
      for any $c \in (a,b)$, there are $s,t \in [a,b]$ such that $c \in (s,t)$ and $st \in \Gamma$.
      Equivalently, $\Gamma$ is connected if $(a,b) = \cup_{st \in \Gamma} (s,t)$. The set of all connected
      graphs on $[a,b]$ is denoted by $\mathcal{G}[a,b]$.
\end{definition}

\begin{problem}
  \label{problem:graph-connected-not-path-connected}
  Give an example of a graph which is connected in the above sense, but not path-connected
  in the usual graph theoretic sense, and give an example which is path-connected, but not connected in the above
  sense.
\end{problem}

Let $U_{st}(\omega) = -\indicator{\omega(s) \neq \omega(t)}$, and
for $a<b$ define
\begin{equation}
  K[a,b](\omega) = \prod_{a \leq s < t \leq b} (1+U_{st}(\omega)), \quad K[a,a](\omega) = 1,
\end{equation}
so that
\begin{equation}
  c_n(x) = \sum_{\omega \in \mathcal{W}_n(0,x)} K[0,n](\omega).
\end{equation}

\begin{problem}
  \label{problem:Ksumprod}
  Show that
  \begin{equation}
    K[a,b](\omega) = \sum_{\Gamma \in \mathcal{B}[a,b]} \prod_{st \in \Gamma} U_{st}(\omega).
  \end{equation}
\end{problem}

\begin{problem} \label{problem:KJidentity}
  For $a<b$, let
  \begin{equation}
    J[a,b](\omega) =
    \sum_{\Gamma \in \mathcal{G}[a,b]} \prod_{st \in \Gamma} U_{st}(\omega).%,   \quad J[a,a](\omega) = 1
  \end{equation}
  Show that
  \begin{equation}
    K[a,b] = K[a+1,b] + \sum_{j=a+1}^b J[a,j]K[j,b].
  \end{equation}
\end{problem}

\begin{problem} \label{problem:c_n-convolution}
  Define
  \begin{equation}
    \pi_m(x) = \sum_{\omega \in \mathcal{W}_m(0,x)} J[0,m](\omega)
  \end{equation}
  for $m\geq 1$.
  Use Problem~\ref{problem:KJidentity} to show that, for $n \ge 1$,
  \begin{equation}
    c_n(x) = (c_1 * c_{n-1})(x) + \sum_{m=1}^n (\pi_m * c_{n-m})(x).
  \end{equation}
  (Compared to \refeq{cnLaceExpansion}, the sum here starts at $m=1$ instead
of $m=2$.  In fact, we will see that $\pi_1(x)=0$ for the self-avoiding walk, since walks cannot self-intersect in 1 step.)
\end{problem}

\begin{definition}
  A \emph{lace} is a minimally connected graph, that is, a connected graph for which the removal of
  any edge would result in a disconnected graph. The set of laces on $[a,b]$ is denoted $\mathcal{L}[a,b]$.
\end{definition}

\begin{problem} \label{problem:LaceIntervals}
  Let $L = \{ s_1t_1, \dots, s_N t_N\}$, where $s_l < t_l$ and $s_l\leq s_{l+1}$ for all $l$ (and all the edges are different).  Show that $L$ is a lace if
  and only if
  \begin{equation}
    a= s_1 < s_2, \quad s_N < t_{N-1} < t_N = b, \quad
    s_{l+1} < t_l \leq s_{l+2} \quad (1 \leq l \leq  N-2),
  \end{equation}
  or $L=\set{ab}$ if $N=1$.  In particular, for $N>1$, $L$ divides $[a,b]$ into $2N-1$ subintervals,
  \begin{equation}
    [s_1, s_2], [s_2, t_1], [t_1,s_3] [s_3, t_2], \dots, [t_{N-2},s_N] [s_N, t_{N-1}], [t_{N-1}, t_N].
  \end{equation}
  Determine which of these intervals must have length at least $1$, and which can have length $0$.
\end{problem}

Let $\Gamma \in \mathcal{G}[a,b]$ be a connected graph.
We associate a unique lace ${\sf L}_\Gamma$ to $\Gamma$ as follows: Let
\begin{equation}
\begin{split}
  t_1 = \max\{t: at\in \Gamma\}, &\quad s_1 = a,\\
  t_{i+1} = \max\{t: \exists s < t_i \text{ such that } st \in \Gamma \}, &\quad s_{i+1} = \min\{ s: st_{i+1} \in \Gamma\}.
\end{split}
\end{equation}
The procedure terminates when $t_{N} = b$ for some $N$,  and we then define
${\sf L}_\Gamma = \{s_1t_1, \dots, s_Nt_N\}$.
We define the set of edges \emph{compatible} with a lace
$L \in \mathcal{L}[a,b]$ to be
\begin{equation}
  \mathcal{C}(L) = \{ st: {\sf L}_{L\cup\{st\}} = L, st\notin L \}.
\end{equation}

\begin{problem} \label{problem:Compatible}
  Show that ${\sf L}_\Gamma = L$ if and only if
  $L \subset \Gamma$ and $\Gamma \setminus L \subset \mathcal{C}(L)$.
\end{problem}

\begin{problem} \label{problem:Jsumprod}
  Show that
  \begin{equation}
    J[a,b](\omega)
    = \sum_{L \in \mathcal{L}[a,b]} \prod_{st \in  L} U_{st}(\omega)
     \sum_{\Gamma: L_\Gamma = L}
     \prod_{s't' \in \Gamma \setminus L} U_{s't'}(\omega).
  \end{equation}
  Conclude from the previous exercise that
  \begin{equation}
    \sum_{\Gamma: {\sf L}_\Gamma = L}\prod_{s't' \in \Gamma \setminus L} U_{s't'}(\omega) = \prod_{s't' \in \mathcal{C}(L)}(1+U_{s't'}(\omega)),
  \end{equation}
  and thus
  \begin{equation}
    J[a,b](\omega)
    = \sum_{L \in \mathcal{L}[a,b]} \prod_{st \in  L} U_{st}(\omega)
    \prod_{s't' \in \mathcal{C}(L)}(1+U_{s't'}(\omega)).
  \end{equation}
\end{problem}

\begin{problem} \label{problem:J-laces}
  Let $\mathcal{L}^{(N)}[a,b]$ denote the set of laces on $[a,b]$
  which consist of exactly $N$ edges.  Define
  \begin{equation}
    J^{(N)}[a,b](\omega) = \sum_{L \in \mathcal{L}^{(N)}[a,b]} \prod_{st \in  L} U_{st}(\omega) \prod_{s't' \in \mathcal{C}(L)}(1+U_{s't'}(\omega))
  \end{equation}
  and
  \begin{equation}
    % \label{}
    \pi_m^{(N)}(x)=(-1)^N\sum_{\omega\in\mathcal{W}_m(0,x)} J^{(N)}[0,m](\omega).
  \end{equation}
  \begin{enumerate}
  \item
    Prove that
    \begin{equation}
      % \label{}
      \pi_m(x)=\sum_{N=1}^\infty (-1)^N \pi_m^{(N)}(x)
    \end{equation}
    with $\pi_m^{(N)}(x)\geq 0$.
  \item
    Describe the walk configurations that correspond to non-zero terms in $\pi_m^{(N)}(x)$, for $N=1,2,3,4$.  What parts of the walk must be mutually avoiding?
  \item
    What is the interpretation of the possibly empty intervals in Problem \ref{problem:LaceIntervals}?
  \end{enumerate}
\end{problem}

\section{Lace expansion analysis in dimensions \texorpdfstring{$d>4$}{d > 4}}\lbsect{ConvLace}

In this section, we outline a proof that the bubble condition holds
for the nearest-neighbour model in sufficiently high dimensions, and for
the spread-out model in dimensions $d>4$ provided $L$ is large enough.
As noted above, the bubble condition implies that $\gamma = 1$ in the
sense that the susceptibility diverges linearly at the critical point
as in \refeq{chiAsymp}.  Proving the bubble condition
will require control of the generating
function $\hat\Pi_z(k)$ at the critical value $z=z_c$.  According to
\refeq{pimy} (see also Problem~\ref{problem:J-laces}), $\Pi_z$ is
given by an infinite series
\begin{equation}
\lbeq{PiNsum}
\Pi_z(x) = \sum_{N=1}^\infty (-1)^N \Pi_z^{(N)}(x),
\qquad \Pi_z^{(N)}(x)=\sum_{m=2}^\infty \pi_m^{(N)}(x) z^m.
\end{equation}
The lace expansion is said to converge
if $\Pi_z(x)$ is absolutely summable when $z=z_c$, in the
strong sense that
\begin{equation}
\sum_{x \in \Z^d}\sum_{N=1}^\infty
\Pi_{z_c}^{(N)}(x)  < \infty.
\end{equation}

There are now several different approaches to proving convergence of
the lace expansion.  In particular, a powerful but technically demanding method
involves the study of \refeq{cnLaceExpansion} by induction on $n$
\cite{HS02}.  Here we will follow the relatively
simple approach of \cite{Slad06}, which was
inspired by a similar argument for percolation in \cite{BCHSS05b}.
Some details are omitted below; these can all be found in \cite{Slad06}.

We will make use of the usual $\ell^p$ norms on functions on $\Z^d$,
for $p=1,2,\infty$.
In addition, when dealing with functions on the torus $[-\pi,\pi]^d$, we
will use the usual $L^p$ norms with respect to the probability measure
$(2\pi)^{-d}\ddk$ on the torus, for $p=1,2$.
To simplify the notation, we will sometimes
omit the measure, and write, e.g.,
$\Bubble(z)=\int \hat{G}_z^2=\|\hat{G}_z\|_2^2$.

\subsection{Diagrammatic estimates}\lbssect{DiagEst}

%Recall that
%%
%\begin{equation}
%%\label{}
%\Pi_z(x) = \sum_{N=1}^\infty (-1)^N \Pi_z^{(N)}(x), \qquad\qquad \Pi_z^{(N)}(x)=\sum_{m=2}^\infty \pi_m^{(N)}(x) z^m
%\end{equation}
%%
%where \JG{diagrams for $\pi_m^{(N)}(x)$.....}.

We will obtain bounds on $\Pi_z(x)$ in terms of $G_z(x)$
and the closely related quantity $H_z(x)$ defined by
\begin{equation}
\lbeq{HzxDefn}
H_z(x)=G_z(x)-\delta_{0x} = \sum_{n=1}^\infty c_n(x) z^n.
\end{equation}
The trivial term $c_0(x)=\delta_{0x}$ in $G_z(x)$
gives rise to a contribution $1$ in
the bubble diagram, and it will be important in the following that this
contribution sometimes be omitted.
It is for this reason that we use $H_z$ as well as $G_z$.

The following \emph{diagrammatic estimates} bound $\Pi_z$ in terms of
$H_z$ and $G_z$.  Once this theorem has been proved,
the details of the definition of
$\Pi_z$ are no longer needed---the rest of the argument is analysis
that uses the diagrammatic estimates.

\begin{theorem}
\lbthm{PiBounds}
For any $z\geq 0$,
\begin{align}
\lbeq{Pi1Bound}
\sum_{x\in\Z^d} \Pi_z^{(1)}(x)
&\leq
z\abs{\Omega}\norm{H_z}_\infty,
\end{align}
\begin{align}
\lbeq{Pi1cosBound}
\sum_{x\in\Z^d} (1-\cos k\cdot x) \Pi_z^{(1)}(x)
&=
0,
\end{align}
and for $N\geq 2$,
\begin{align}
\lbeq{Pi2Bound}
\sum_{x\in\Z^d} \Pi_z^{(N)}(x)
&\leq
\norm{H_z}_\infty \norm{G_z  *  H_z}_\infty^{N-1}
%\qquad (N\geq 2)
,
\end{align}
\begin{align}
\lbeq{Pi2cosBound}
\sum_{x\in\Z^d} (1-\cos k\cdot x) \Pi_z^{(N)}(x)
&\leq
N^2 \norm{(1-\cos k\cdot x)H_z}_\infty \norm{G_z  *  H_z}_\infty^{N-1}
%\qquad (N\geq 2)
.
\end{align}
\end{theorem}
\begin{proof}
We prove just the cases $N=1,2$ here; the complete proof can
be found in \cite[Theorem~4.1]{Slad06}.

For $N=1$, since $\pi_m^{(1)}(x)$ is equal to $\delta_{0x}$ times
the number $\sum_{y\in \Omega}c_{m-1}(y)$ of self-avoiding returns,
we have
\begin{equation}
    \sum_{x\in\Z^d} \Pi_z^{(1)}(x)
    =
    \sum_{y\in \Omega}
    \sum_{m=2}^\infty c_{m-1}(y) z^m
    =
    \sum_{y\in \Omega} z H_z(y),
\end{equation}
which implies \refeq{Pi1Bound}.   Also, \refeq{Pi1cosBound} follows from
\begin{equation}
    \sum_{x\in\Z^d} (1-\cos k\cdot x) \Pi_z^{(1)}(x)
    =
    (1-\cos k \cdot 0) \Pi_z^{(1)}(0) = 0.
\end{equation}

For $N=2$,
dropping the mutual avoidance constraint between the three lines
in $\pi_m^{(2)}(x)$ in \refeq{pimy} gives
\begin{align}
\sum_{x\in\Z^d} \Pi_z^{(2)}(x)
&\leq
\sum_{x\in\Z^d} H_z(x)^3 \leq \norm{H_z}_\infty (H_z  *  H_z)(0)
\notag\\
&\leq
\norm{H_z}_\infty \norm{H_z *  H_z}_\infty
\end{align}
and
\begin{align}
\sum_{x\in\Z^d} (1-\cos k\cdot x) \Pi_z^{(2)}(x)
&\leq
\norm{(1-\cos k\cdot x)H_z}_\infty \norm{H_z *  H_z}_\infty.
\end{align}
Since $0\leq H_z(x)\leq G_z(x)$, this is stronger than
\refeq{Pi2Bound} and \refeq{Pi2cosBound}.
\end{proof}

\subsection{The small parameter}\lbssect{SmallParam}

\refthm{PiBounds} shows that the sum over $N$ in \refeq{PiNsum} can be
dominated by the sum of a geometric series with ratio $\norm{G_z  *  H_z}_\infty$.
Ideally, we would like this ratio to be small.
A Cauchy--Schwarz estimate gives
\begin{align}
%\label{}
\norm{H_z *  G_z}_\infty & \le \norm{H_z }_\infty+ \norm{H_z *  H_z}_\infty
 \leq \norm{H_z }_\infty+ \norm{H_z}_2^2
 \nonumber \\ &\leq \norm{H_z }_\infty+ \norm{G_z}_2^2
 = \norm{H_z }_\infty+ \Bubble(z),
\end{align}
but this looks problematic because the upper bound involves
the bubble diagram
---the very quantity we are trying to prove is finite at the critical point!
So we will need some insight to make good use of the diagrammatic estimates.

An important idea will be to use not just the finiteness,
but also the smallness of $H_z$.  Specifically, we might hope that
$\norm{H_{z_c}}_2^2=\norm{\smash[t]{\hat{H}_{z_c}}}_2^2=\norm{\smash[t]{\hat{G}_{z_c}-1}}_2^2$ % change size of norm bars.....?
should be small when the corresponding quantity for SRW is small.

Let $C_z(x)=\sum_{n=0} c_n^{(0)}(x) z^n$ be the analogue of $G_z(x)$ for the
SRW model.  Its critical value is $z_0=\abs{\Omega}^{-1}$, and
\begin{equation}
%\label{}
\hat{C}_z(k) = \frac{1}{1-z\abs{\Omega}\hat{D}(k)},\qquad\qquad\hat{C}_{z_0}(k)=\frac{1}{1-\hat{D}(k)}.
\end{equation}
The SRW analogue of $\norm{\smash[t]{\hat{G}_{z_c}-1}}_2^2$ is
\begin{equation}
%\label{}
\norm{\smash[t]{\hat{C}_{z_0}-1}}_2^2
%=\int_{[-\pi,\pi]^d}\left( \frac{1}{1-\hat{D}(k)}-1 \right)^2\frac{\ddk}{(2\pi)^d}
%=\int_{[-\pi,\pi]^d} \frac{\hat{D}(k)^2}{\bigl( 1-\hat{D}(k) \bigr)^2} \frac{\ddk}{(2\pi)^d}.
=\int \left( \frac{1}{1-\hat{D}}-1 \right)^2
=\int \frac{\hat{D}^2}{\bigl( 1-\hat{D} \bigr)^2}.
\end{equation}
The following elementary proposition shows that the above integral is
small for the models we are studying.
The hypothesis $d>4$ is needed for convergence, due
to the $(|k|^{-2})^2$ singularity at the origin.
%On the other hand, if
%we expand $(1-\hat D)^{-2}$ in powers of $\hat D$ and keep only the constant
%term $1$, then we are left with the integral $\int\hat{D}^2 = (D*D)(0)$, and this
%is the probability that SRW returns to the origin after two steps, namely
%$|\Omega|^{-1}$.  This motivates the following result.

%
\begin{prop}
\lbprop{betaBound}
Let $d>4$.  Then
\begin{equation}
\lbeq{betaBound}
%\int_{[-\pi,\pi]^d} \frac{\hat{D}(k)^2}{\bigl( 1-\hat{D}(k) \bigr)^2} \frac{\ddk}{(2\pi)^d}
\int \frac{\hat{D}^2}{\bigl( 1-\hat{D} \bigr)^2}
\leq \beta
\end{equation}
where, for some constant $K$,
\begin{equation}
\lbeq{betaFormula}
\beta \defeq
\begin{cases}
\dfrac{K}{d-4}&\text{for the nearest-neighbour model,}\\
\text{\raisebox{0pt}[1.75em]{}} \dfrac{K}{L^d}&\text{for the spread-out model.}
\end{cases}
\end{equation}
\end{prop}
\begin{proof}
This is a calculus problem.  For the nearest-neighbour model, see
\cite[Lemma~A.3]{MS93}, and for the spread-out model see
\cite[Proposition~5.3]{Slad06}.
\end{proof}
%

%%
%\begin{remark}
%$\int \bigl(1-\hat{D}\bigr)^{-2}$ counts the expected number of intersections \JG{counting multiplicity; should we add what $\int \hat{D}^2/(1-\hat{D})^2$ counts?} of 2 independent SRWs started from $0$, and $\int \bigl(1-\hat{D}\bigr)^{-1} $ counts the expected number of visits to $0$ of a SRW started from $0$.
%\end{remark}
%%

We will prove the following theorem.
\begin{theorem}
\lbthm{betaThreshold}
There are constants $\beta_0$ and $C$, independent of $d$ and $L$, such that
when \refeq{betaBound} holds with $\beta\leq\beta_0$
we have $\Bubble(z_c)\leq 1+C\beta$.
\end{theorem}
\refthm{betaThreshold} achieves
our goal of proving the bubble condition for the nearest-neighbour model
in sufficiently high dimensions, and for the spread-out model with $L$
sufficiently large in dimensions $d>4$.  As noted previously, this
gives the following corollary that $\gamma=1$ in high dimensions.

\begin{coro}
\lbcoro{chiAsympFollows}
When \refeq{betaBound} holds with $\beta\leq\beta_0$, then as $z\increasesto z_c$,
\begin{equation}
\lbeq{chiAsympFollows}
\chi(z) \asymp \frac{1}{1-z/z_c}.
\end{equation}
\end{coro}

\subsection{Proof of \refthm{betaThreshold}}%\lbssect{betaThreshProof}

%In the rest of this section, we will give the main ideas of the proof of
%\refthm{betaThreshold}.
We begin with the following elementary lemma, which will be a principal
ingredient in the proof.

\begin{lemma}
\lblemma{ForbiddenValues}
Let $a<b$ be real numbers and let $f$ be a continuous
real-valued function on $[z_1,z_2)$ such that $f(z_1)\leq a$.
Suppose that, for each $z\in(z_1,z_2)$, we have the implication
\begin{equation}
\lbeq{SmallVerySmall}
f(z)\leq b \qquad \implies \qquad f(z)\leq a.
\end{equation}
Then $f(z)\leq a$ for all $z\in[z_1,z_2)$.
\end{lemma}
\begin{proof}
The result is a straightforward application of the Intermediate Value Theorem.
\end{proof}

\begin{figure}[h]
  \input{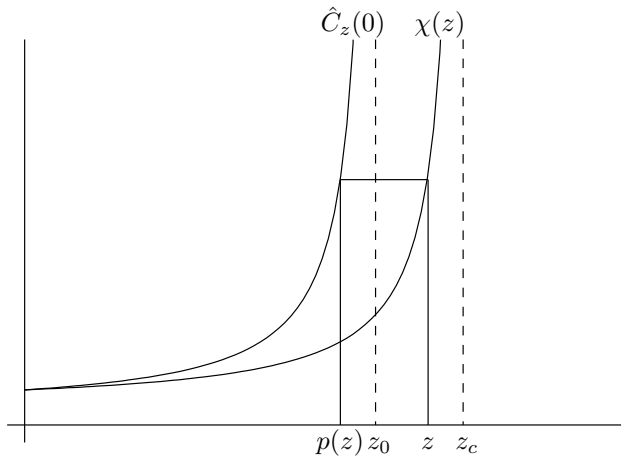}
  \caption{The definition of $p(z)$.}
  \label{fig:pofz}
\end{figure}

We will apply \reflemma{ForbiddenValues} to a carefully chosen function $f$,
based on a \emph{coupling} between $\hat{G}$ on the parameter range $[0,z_c)$,
and the SRW analogue $\hat{C}$ on the parameter range $[0,z_0)$.
To define the coupling, let $z\in[0,z_c)$ and define $p(z)\in[0,z_0)$ by
\begin{equation}
\lbeq{pzDefn}
\hat{G}_z(0)=\chi(z)=\hat{C}_{p(z)}(0)=\frac{1}{1-p(z)\abs{\Omega}},
\end{equation}
i.e.,
\begin{equation}
\lbeq{pzFormula}
p(z)\abs{\Omega} = 1-\chi(z)^{-1} = z\abs{\Omega} + \hat{\Pi}_z(0).
\end{equation}
See Figure~\ref{fig:pofz}.
We expect (or hope!) that $\hat{G}_z(k)\approx \hat{C}_{p(z)}(k)$ for
\emph{all} $k$, not just for $k=0$,
as well as an additional condition that expresses another form of similarity
between $\hat{G}_z(k)$ and $\hat{C}_{p(z)}(k)$.  For the latter,
we define
\begin{equation}
%\label{}
-\tfrac{1}{2}\Delta_k \hat{G}_z(l)
= \hat{G}_z(l)-\tfrac{1}{2}\bigl(\hat{G}_z(l+k)+\hat{G}_z(l-k)\bigr);
\end{equation}
this is the Fourier transform of
$(1-\cos k\cdot x) G_z(x)$ with $l$ as the dual variable.
We aim to apply \reflemma{ForbiddenValues} with
$z_1=0$, $z_2=z_c$, $a=1+{\rm const} \cdot \beta$ (with a constant whose
value is determined in \refeq{f4Impl} below),
$b=4$ (in fact, any fixed $b>1$ will do here), and
\begin{equation}
\lbeq{fDefn}
f(z)=\max\set{f_1(z),f_2(z),f_3(z)}
\end{equation}
where
\begin{equation}
%
%\begin{split}
f_1(z)=z\abs{\Omega},\qquad f_2(z)
=\sup_{k\in[-\pi,\pi]^d} \frac{|\hat{G}_z(k)|}{|\hat{C}_{p(z)}(k)|},
\end{equation}
and
\begin{equation}
f_3(z)= \sup_{k,l\in[-\pi,\pi]^d}
\frac{\frac 12 |\Delta_k \hat{G}_z(l)|}{|U_{p(z)}(k,l)|},
\lbeq{f3Defn}
%\end{split}
%
\end{equation}
with
\begin{align}
    U_{p(z)}(k,l)
   &= 16\hat{C}_{p(z)}(k)^{-1}\left( \hat{C}_{p(z)}(l-k)\hat{C}_{p(z)}(l)
    + \hat{C}_{p(z)}(l+k)\hat{C}_{p(z)}(l)
    \right. \nonumber \\ & \hspace{28mm} \left.
    + \hat{C}_{p(z)}(l-k)\hat{C}_{p(z)}(l+k) \right).
\end{align}
The choice of $U_{p(z)}(k,l)$ is made for technical reasons not explained
here, and should be regarded as a useful replacement for the more natural
choice $\frac 12 |\Delta_k \hat{C}_{p(z)}(l)|$.

The conclusion from \reflemma{ForbiddenValues} would
be that $f(z)\leq a$ for all $z\in[0,z_c)$.
The inequality \refeq{betaBound} can be used to show
that $\|(1-\hat D)^{-1}\|_2^2
\le 1+3\beta$ (see \cite[(5.10)]{Slad06}), and hence
we may assume that $\|(1-\hat D)^{-1}\|_2^2
\le 2$.
Using $f_2(z)\leq a$ we therefore conclude that
\begin{align}
B(z_c)
&=
\lim_{z\increasesto z_c} B(z)
=
\lim_{z\increasesto z_c} \norm{\smash[t]{\hat{G}_z}}_2^2
\notag\\
&\leq
a^2\lim_{z\increasesto z_c} \|\hat{C}_{p(z)}\|_2^2
=
a^2 \bignorm{\bigl(1-\hat{D}\bigr)^{-1}}_2^2
\notag\\
&\leq 2 a^2 <\infty
\end{align}
which is our goal.
Thus it suffices to verify the hypotheses on $f(z)$ in
\reflemma{ForbiddenValues}.
This is the content of the following lemma.

\begin{lemma}\lblemma{fHypos}
The function $f(z)$ defined by \refeq{fDefn}--\refeq{f3Defn} is
continuous on $[0,z_c)$, with $f(0)=1$, and for each $z\in(0,z_c)$,
\begin{equation}
\lbeq{f4Impl}
f(z)\leq 4 \qquad \implies \qquad f(z)\leq 1+O(\beta).
\end{equation}
\end{lemma}
\begin{proof}
It is relatively easy to verify the continuity of $f$, and we omit the details.
To see that $f(0)=1\leq a$, we observe that $f_1(0)=0$,
$p(0) =0$ and hence $f_2(0)=1/1=1$, and $f_3(0)=0$.
The difficult step is to prove the implication \refeq{f4Impl}, and the
remainder of the proof concerns this step.
We assume throughout that $f(z)\leq 4$.

%This is the remaining task for this section, and we first outline
%the strategy of the proof.
%\begin{itemize}
%\item
%The assumption \refeq{betaBound} will lead to bounds on $C$.
%\item
%The assumption $f(z)\leq 4$, together with the bounds on $C$, give bounds on norms of $H$.
%\item
%In turn, from \refthm{PiBounds} we obtain bounds on $\Pi$ of order $\beta$.  In particular, the coupling of $G$ to $C$ has allowed us to remove the effect of the poor constant $4$.
%\item
%These bounds on $\Pi$ lead to $\hat{G}_z(k)/\hat{C}_{p(z)}(k) \leq 1 + O(\beta)$, so that $f_2(z)\leq 1+O(\beta)=a$.  Similar arguments apply for $f_1$ and $f_3$.
%\end{itemize}
%%
%\JG{This outline does not entirely seem to match up with the order of presentation that follows.}

We consider first $f_1(z)=z\abs{\Omega}$.
Our goal is to prove that $f_1(z) \le 1+O(\beta)$, and for this we
will only use the assumptions $f_1(z) \le 4$ and $f_2(z) \le 4$;
we do not yet need $f_3$.
Since $0<\chi(z)<\infty$, we have
$\chi(z)^{-1}=1-z\abs{\Omega}-\hat{\Pi}_z(0)>0$, i.e.,
\begin{equation}
%\label{}
f_1(z)=z\abs{\Omega} < 1-\hat{\Pi}_z(0) \leq 1+|\hat{\Pi}_z(0)|.
\end{equation}
The required bound for $f_1(z)$ will follow once we
show that for all $z\in(0,z_c)$ and for all $k\in[-\pi,\pi]^d$,
\begin{equation}
\lbeq{PiObeta}
|\hat{\Pi}_z(k)| \leq O(\beta).
\end{equation}
To prove \refeq{PiObeta} we use
\refthm{PiBounds} (more precisely, \refeq{Pi1Bound} and \refeq{Pi2Bound}),
to obtain
\begin{align}
|\hat{\Pi}_z(k)|
&\leq
\sum_{N=1}^\infty \sum_{x\in\Z^d} \Pi_z^{(N)}(x)
\notag\\
\lbeq{PihatUnifBound}
&\leq
\norm{H_z}_\infty \left( f_1(z) + \sum_{N=2}^\infty \norm{G_z * H_z}_\infty^{N-1} \right).
\end{align}
For the first term, we use
$f_1(z)\leq 4$.  For the second term, we need a bound on $\norm{G_z * H_z}_\infty$
in order to bound the sum.   By definition,
\begin{equation}
\lbeq{GstarHBound}
\norm{G_z * H_z}_\infty \leq \norm{H_z}_\infty + \norm{H_z * H_z}_\infty \leq \norm{H_z}_\infty +\norm{H_z}_2^2.
\end{equation}
Now $H_z$ is the generating function for SAWs which take at least one step.
By omitting the avoidance constraint between the first step and subsequent
steps, we obtain
\begin{equation}
%\label{}
H_z(x)\leq z\abs{\Omega} (D * G_z)(x) \leq 4 (D*G_z)(x).
\end{equation}
Thus we can bound the second term in \refeq{GstarHBound}, using $f_2(z)\le 4$
and \refprop{betaBound},
as
\begin{align}
\norm{H_z}_2^2
&\leq
4^2 \norm{D*G_z}_2^2 = 4^2 \bignorm{\hat{D}\hat{G}_z}_2^2
\notag\\
&\leq
4^4 \bignorm{\hat{D}\hat{C}_{p(z)}}_2^2 = 4^4 \norm{D*C_{p(z)}}_2^2
\notag\\
&\leq
4^4 \norm{D*C_{z_0}}_2^2 = 4^4\bignorm{\hat{D}\bigl(1-\hat{D}\bigr)^{-1}}_2^2
\notag\\
&\leq 4^4 \beta.
\end{align}
Similar estimates show $\norm{H_z}_\infty \leq O(\beta)$.  If we
substitute these estimates into \refeq{PihatUnifBound}, we obtain
\begin{equation}
|\hat{\Pi}_z(k)|
\leq
C\beta \left( 4 + \sum_{N=2}^\infty (C\beta)^{N-1} \right)
\end{equation}
for some constant $C$,
so that
\refeq{PiObeta} will hold for $\beta$ sufficiently small.
This completes the proof for $f_1(z)$.

We next sketch the proof that $f_2(z)\leq 1+O(\beta)$.
Recalling the notation $\hat{F}_z(k)=\hat{G}_z(k)^{-1}$ introduced in
\refeq{FhatDefn},  and using the formulas
\refeq{pzDefn} and \refeq{pzFormula}
for $p(z)$, we obtain
\begin{align}
\frac{\hat{G}_z(k)}{\hat{C}_{p(z)}(k)} -1
&=
\frac{1-p(z)\abs{\Omega}\hat{D}(k)}{\hat{F}_z(k)} -1
\notag\\
&=
\frac{1-\bigl(z\abs{\Omega}+\hat{\Pi}_z(0)\bigr)\hat{D}(k)-\hat{F}_z(k)}{\hat{F}_z(k)}
\notag\\
&=
\frac{-\hat{\Pi}_z(0)\hat{D}(k)+\hat{\Pi}_z(k)}{\hat{F}_z(k)}
\notag\\
\lbeq{GhatOverChatBound}
&=
\frac{\hat{\Pi}_z(0)\bigl(1-\hat{D}(k)\bigr)-\bigl(\hat{\Pi}_z(0)-\hat{\Pi}_z(k)\bigr)}{\hat{F}_z(k)}.
\end{align}
The bound $f_3(z)\leq 4$ and \refeq{Pi2cosBound} can be used to show that
$|\hat{\Pi}_z(0)-\hat{\Pi}_z(k)|\leq \linebreak O(\beta)\bigl(1-\hat{D}(k)\bigr)$
(see \cite{Slad06} for details); it is precisely at this point
that the need to include $f_3$ in the definition of $f$ arises.
Together with \refeq{PiObeta}, this shows that the numerator of
\refeq{GhatOverChatBound} is $O(\beta)\bigl(1-\hat{D}(k)\bigr)$.

For the denominator, we recall the formula \refeq{FhatFormula}:
\begin{equation}
\hat{F}_z(k)=\chi(z)^{-1} + z\abs{\Omega}\bigl(1-\hat{D}(k)\bigr) + \bigl(\hat{\Pi}_z(0)-\hat{\Pi}_z(k)\bigr).
\end{equation}
To bound $\hat{F}_z(k)$ from below, we consider two parameter ranges for $z$.
If $z\leq\tfrac{1}{2}\abs{\Omega}^{-1}$, we can make the trivial
estimate $\chi(z)^{-1} \geq \hat{C}_z(0)^{-1}
=1-z\abs{\Omega}\geq \tfrac{1}{2}$, so that
$\hat{F}_z(k)\geq\tfrac{1}{2}+0-O(\beta)\geq \tfrac{1}{4}$ for small $\beta$.
Since the numerator of \refeq{GhatOverChatBound} is itself $O(\beta)$,
this proves that $f_2(z)\leq 1+O(\beta)$ for this range of $z$.

It remains to consider $\tfrac{1}{2}\abs{\Omega}^{-1} \leq z\leq z_c$.  Now we estimate
\begin{equation}
%\label{}
\hat{F}_z(k)\geq 0+\tfrac{1}{2}\bigl(1-\hat{D}(k)\bigr)-O(\beta)\bigl(1-\hat{D}(k)\bigr)\geq \tfrac{1}{4}\bigl(1-\hat{D}(k)\bigr).
\end{equation}
The factors $1-\hat{D}(k)$ in the numerator and denominator of \refeq{GhatOverChatBound} cancel, leaving $O(\beta)$ as desired.

Finally the proof for $f_3(z)$ is similar to the proof for $f_2(z)$,
and we refer to \cite{Slad06} for the details.
\end{proof}

\subsection{Tutorial}
\lbsect{Tut3}

For simplicity, we restrict our attention now to the nearest-neighbour model
of SAWs in dimensions sufficiently high that the preceding arguments and
conclusions apply.
In \reflemma{fHypos}, we found that $f_2(z) \le a=1+O(d^{-1})$,
since  $\beta \leq O((d-4)^{-1})= O(d^{-1})$.  This
estimate, which states that
\begin{equation} \label{eq:infrared}
  \hat G_z(k) \leq a \hat C_{p(z)}(k) \qquad k \in [\-\pi,\pi]^d, \;\;\;
  z \in (0,z_c),
\end{equation}
is most important for $k \approx 0$, the \emph{small frequencies}, and it
is referred to as the \emph{infrared bound}.  Other bounds obtained in
\reflemma{fHypos} can be framed as follows:
%The convergence result relies on a crucial comparision between simple random walks and self-avoiding walks:
%There is $\beta_0 > 0$ and $c > 0$ such that if $\beta = \|\hat C_1 - 1\|_2^2 \leq \beta_0$, then
%%$\|\hat G_{z_c} - 1\|_2^2 \leq c \|\hat C_1 -1\|_2^2$.
%$\|H_{z_c}\|_2^2 = \|\hat G_{z_c} - 1\|_2^2 \leq c \|\hat C_1 -1\|_2^2$.
%More precisely,
%%Recall that $\beta$ and $\|H_{z_c}\|_2^2$ are analogous quantities:
there is a constant $c$, independent of $z\le z_c$, such that
\begin{equation}
  \|H_z\|_2^2 \leq c d^{-1},
  \quad
  \|H_z\|_\infty \leq c d^{-1},
  \quad
  \|\Pi_z\|_1 \leq  c d^{-1},
\end{equation}
% \begin{equation}
%   \|\Pi_z\|_1 = \sum_{x\in\Z^d} |\Pi_z(x)|\leq \bar c_a \beta,
% \end{equation}
and
%The \emph{diagrammatic estimates} proved are
%\begin{equation}
%  \sum_{x\in\Z^d} \Pi_z^{(1)}(x) \leq z|\Omega| \|H_z\|_\infty,
%  \quad
%  \sum_{x\in\Z^d} \Pi_z^{(N)}(x) \leq \|H_z\|_\infty \|H_z * G_z\|_\infty^{N-1}.
%\end{equation}
%The estimates on $H_z$ and $\|H_z * G_z\|_\infty \leq \|H_z\|_\infty + \|H_z\|_2^2$ then imply
\begin{equation} \label{eq:Pi-lace-bound}
  \|\Pi_z^{(N)}\|_1 \leq (cd^{-1})^N, \quad
  \sum_{N= M}^\infty \|\Pi_z^{(N)}\|_1 \leq cd^{-M}.
\end{equation}
We also recall that the Fourier transform of the two-point function can be
written as
\begin{equation}
  \hat G_z(k) = \frac{1}{1-z|\Omega|\hat D(k) - \hat \Pi_z(k)}.
\end{equation}
Since $\hat G_z(0) \to \infty$ as $z \to z_c$, we obtain the equation
\begin{equation} \label{eq:G-z_c}
  1-z_c|\Omega|-\hat\Pi_{z_c}(0) = 0.
\end{equation}
This equation provides a starting point to study the connective constant
$\mu = z_c^{-1}$.

\begin{problem} \label{problem:1/d}
  In this problem, we show that the connective constant obeys
  \begin{equation}
    \mu = 2d -1 - (2d)^{-1} + O((2d)^{-2}) \qquad \text{as $d\to \infty$.}
  \end{equation}
  This special case of the results discussed in \refSSect{dexpansion}
  was first proved by Kesten \cite{Kest64}, by very different means.
  \medskip

  (a) Let $m \geq 1$ be an integer.
  Show that $\|(1-\hat D)^{-m}\|_1$ is non-increasing in $d > 2m$.
  In particular, it follows that $\|\hat C_{z_0}\|_2$ is
  bounded uniformly in $d>4$.

  \emph{Hint:} $A^{-m} = \Gamma(m)^{-1} \int_0^\infty u^{m-1} e^{-uA} \; du$.
  \medskip

  (b) Let $H^{(j)}_z(x) = \sum_{m = j}^\infty c_m(x) z^m$ be the generating
  function for
  SAWs that take at least $j$ steps. By relaxing the condition of mutual
  self-avoidance for the first $j$ steps, show that
  \begin{equation}
    \|H^{(j)}_{z_c}\|_\infty \leq O((2d)^{-j/2}), \quad j > 1.
  \end{equation}

  \emph{Hint:} Use the infrared bound for the two-point function \eqref{eq:infrared},
  and that the probability that
  a $2j$-step simple random walk which starts at $0$ also ends at $0$ is
  \begin{equation}
    \|\hat D^{2j}\|_1 \leq O((2d)^{-j}).
  \end{equation}
  \medskip

  (c)
  Recall that $\pi^{(1)}_n(x) = 0$ if $x \neq 0$, so that
  $\hat \Pi^{(1)}_z(0) = \sum_{x\in\Z^d} \Pi^{(1)}_z(x) = \Pi^{(1)}_z(0)$
  is the generating function for all self-avoiding returns. Prove that
  \begin{equation}
    \hat\Pi^{(1)}_{z_c}(0) = (2d)^{-1} + 3 (2d)^{-2} + O((2d)^{-3})
  \end{equation}
  \medskip

  (d) Note that $\hat \Pi^{(2)}_z(0)$ is the generating function for all
  \emph{$\theta$-walks}:
  paths that visit their eventual endpoint, return to the origin, then return to
  their endpoint, and are otherwise self-avoiding.
%paths that that are self-avoiding except for a visit to its eventual endpoint followed by a return to its origin before reaching that endpoint.
  Prove that
  \begin{equation}
    \hat \Pi^{(2)}_z(0) = (2d)^{-2} + O((2d)^{-3}).
  \end{equation}
  \medskip

  (e) Conclude from (c) and (d) that
  \begin{equation}
    \hat \Pi_z(0) = - (2d)^{-1} - 2(2d)^{-2} + O((2d)^{-3}),
  \end{equation}
  and use this to show
  \begin{equation}
    \mu = 2d -1 - (2d)^{-1} + O((2d)^{-2}).
  \end{equation}
\end{problem}

%For the next problem, recall the definition of the susceptibility,
%\begin{equation}
%  \chi(z) = \sum_{x\in\Z^d} G_z(x) = \hat G_z(0),
%\end{equation}
%and that the critical value $z_c$ was defined such that $\chi(z) \to \infty$ as $z \nearrow z_c$.
We have seen in \refSSect{DiffIneq} that $\chi(z) \asymp (1-z/z_c)^{-1}$ in high dimensions, assuming the bubble condition.
The next problem shows that this bound can be improved to an asymptotic formula.

\begin{problem} \label{problem:chi-asymptotic}
  (a) Show that
  \begin{equation}
    \frac{d[z\chi(z)]}{dz} = V(z) \chi(z)^2,
    \quad \text{where }
    V(z) = 1- \hat \Pi_z(0) + z\frac{d\hat\Pi_z(0)}{dz}.
  \end{equation}
  \emph{Hint:} Let $\hat F_z(0) = \chi(z)^{-1} = 1-z|\Omega| - \hat \Pi_z(0)$
  and express the left-hand side in terms of $\hat F_z(0)$.
  \medskip
  % \begin{equation*}
  %   \hat F_z(0) = \frac{1}{\chi(z)} =  \frac{1}{\hat G_z(0)} = 1-z|\Omega| - \hat \Pi_z(0).
  % \end{equation*}
  % and show that
  % \begin{equation*}
  %   \frac{d[z\chi(z)]}{dz} = \hat F_z(0)^{-2} [\hat F_z(0) - z \frac{d\hat F_z(0)}{dz}] = V(z) \chi(z)^2
  % \end{equation*}
  % where
  % \begin{equation*}
  %   V(z) = 1- \hat \Pi_z(0) + z\frac{d\hat\Pi_z(0)}{dz}.
  % \end{equation*}

  (b) Show that $\hat \Pi_{z_c}(0)$, $\frac{d}{dz}\hat\Pi_{z_c}(0)$ and thus $V(z_c)$ are finite. It follows that
  \begin{equation}
    \frac{d[z\chi(z)]}{dz} = V(z) \chi(z)^2 \sim V(z_c) \chi(z)^{2} \quad \text{ as $z\nearrow z_c$,}
  \end{equation}
  where $f(z) \sim g(z)$ means $\lim_{z \nearrow z_c} f(z)/g(z) = 1$.
  \medskip

  (c) Prove that $\chi(z) \sim A (1-z/z_c)^{-1}$ as $z \nearrow z_c$,
    where the constant $A$ is given by
    $A = z_c^{-1} [2d+ \frac{d}{dz}|_{z=z_c} \hat\Pi_z(0)]^{-1}$.
\end{problem}

\section{Integral representation for walk models}
\lbsect{intrep}

It has long been understood by physicists that it is sometimes possible
to represent random fields by random walks.
Ideas in this direction due to Symanzik \cite{Syma69} were influential
among mathematicians, and inspired,  e.g., the analysis of
\cite{BFS83II,BFS82}
who showed how to use random walks to represent and analyse ferromagnetic
lattice spin systems.
In this section, we develop representations of two random walk models
in terms of random fields, via functional integrals.
Our ultimate goal is rather the opposite to
that of \cite{BFS83II,BFS82}, namely
we wish to study models of random walks via studying their
integral representations.
This will be the topic of \refSect{ctwsaw}.

We begin in Section~\ref{s:Gaussian-integrals}
with some background material about Gaussian integrals.
In Section~\ref{sec:G-saw-loops}, we use these Gaussian integrals
to represent a model of SAWs in a background of
self-avoiding loops, a model closely related to the $O(n)$ loop
model discussed in Section~\ref{sec:loopholo}.
The random field in these Gaussian integrals is called a \emph{boson}
field in physics.  It was realised in the physics literature
\cite{McKa80,PS80}
that the loops in the loop model could be eliminated by the
use of anti-commuting variables, referred to as a
\emph{fermion} field, thereby providing a representation for models of
SAWs.  The anti-commuting variables can
be understood in terms of differential forms with their
anti-commuting wedge product, and in
Sections~\ref{s:Gaussian-forms}--\ref{s:functions-forms}
we provide the relevant background on differential forms and
their integration.  Finally, in \refSect{repsaw}, we obtain
an integral representation for SAWs.
The ideas in this section are developed in
further detail in \cite{BIS09}.

\subsection{Gaussian integrals}
\label{s:Gaussian-integrals}

%The basic reference fields, corresponding to simple random walk, are Gaussian fields.
%We summarise some facts about complex Gaussian variables, presented in a way to suit our
%subsequent applications.

Fix a positive integer $M$.
Later, we identify the
set $\{1, \dots, M\}$ with a finite set $\Lambda$ on which the walks
related to the fields take place, e.g., $\Lambda \subset \Z^d$.
Consider a two-compo\-nent real field
\begin{equation}
%\label{}
(u,v) = (u_x, v_x)_{x \in \{1, \dots, M\}} \in \R^M \times \R^M.
\end{equation}
%
%$(u_x, v_x)_{x \in \{1, \dots, M\}}$, i.e., a vector in $\R^M \times \R^M$,
From this, we obtain
the associated complex field
$(\varphi, \bar\varphi) = (\varphi_x, \bar\varphi_x)_{x\in\{1,\dots, M\}}$,
where
\begin{equation}
%\label{}
\varphi_x = u_x + iv_x, \quad \bar\varphi_x = u_x -iv_x;
\end{equation}
this is the so-called \emph{boson field}.
We wish to integrate with respect to the
variables $(\varphi_x, \bar\varphi_x)$, and for this we will
use the differentials $d\varphi_x = du_x+i\, dv_x$ and
$d\bar\varphi_x = du_x-i\, dv_x$.
As we will discuss in more detail in Section~\ref{s:Gaussian-forms},
differentials are multiplied
using an anti-commuting product, so in particular
$du_x \, dv_x = - dv_x \, du_x$, $du_x \, du_x=dv_x \, dv_x=0$, and
$d\bar\varphi_x \, d\varphi_x = 2i \, du_x \, dv_x$.
%
%$\varphi_x = u_x + iv_x$, $\bar\varphi_x = u_x -iv_x$.

Let $C = (C_{xy})_{x,y\in \{1,\dots,M\}}$ be an $M \times M$ complex matrix with positive Hermitian part,
meaning that
\begin{equation}
  \sum_{x,y=1}^M \varphi_x (C_{xy} + \bar C_{yx}) \bar\varphi_y > 0
  \quad \text{for all $\varphi \neq 0$ in $\C^M$.}
\end{equation}
It is not difficult to see that this implies that $A = C^{-1}$ exists.
%[Indeed, if $C\varphi=0$ then $0 = (\varphi,C\varphi)$ and
%$0 = (C\varphi, \varphi) = (\varphi, C^*\varphi)$, and thus $(\varphi, (C+C^*)\varphi) = 0$ which
%implies $\varphi = 0$.]
%\RB{The following definition would be more natural with forms, but they are only
%introduced later. Verify that the standard notation for this kind of measure is
%really complex Gaussian measure because there is also the complex Gaussian function
%$e^{ix^2}$.}
The (complex) Gaussian measure with covariance $C$ is defined by
\begin{equation}
  d\mu_C(\varphi, \bar\varphi) = \frac{1}{Z_C}  e^{-\varphi A \bar\varphi} \; d\bar\varphi \, d\varphi,
\end{equation}
where $\varphi A \bar\varphi = \sum_{x,y=1}^M \varphi_x A_{xy} \bar\varphi_y$,
and
\begin{equation}
    d\bar\varphi \, d\varphi = d\bar\varphi_1 \, d\varphi_1 \cdots
    d\bar\varphi_M \, d\varphi_M
    = (2i)^M du_1\, dv_1 \cdots du_M\, dv_M
\end{equation}
is a multiple of the
Lebesgue measure on $\R^{2M}$.
The normalisation constant
\begin{equation}
  Z_C = \int_{\R^{2M}} e^{-\varphi A \bar\varphi} \; d\bar\varphi \, d\varphi
\end{equation}
can be computed explicitly.

\begin{lemma}
\label{lem:ZC}
For $C$ with positive Hermitian part,
the normalisation of the Gaussian integral is given by
  \begin{equation}
    Z_C = \frac{(2\pi i)^M}{\det A}.
  \end{equation}
\end{lemma}

\begin{proof}[Proof]
In this proof, we make the simplifying assumption that
$C$ and thus also $A$ are Hermitian, though the result holds more generally;
see \cite{BIS09}.
  %Since $C$ is Hermitian, $A=C^{-1}$ is Hermitian,
  By the spectral theorem for Hermitian matrices, there is a positive
  diagonal matrix $D = \mathrm{diag}(d_x)$
  and a unitary matrix $U$ such that $A=U^{-1}D U$.
  Then, $\varphi A \bar\varphi = \rho D \bar\rho$ where $\rho = \bar U \varphi$
  ($\bar U$ is the complex conjugate of $U$).
  By a change of variables in the integral and explicit computation of
  the resulting 1-dimensional integral,
  \begin{equation}
    Z_C = \prod_{x=1}^M
    \int_{\R^2} e^{-d_x(u_x^2+v_x^2)} \; 2i \, du_x \, dv_x
    = \frac{(2\pi i)^M}{\prod_{x=1}^M d_x} = \frac{(2\pi i)^M}{\det A}.
    \qedhere
  \end{equation}
\end{proof}

We define the differential operators
\begin{equation}
  \ddp{}{\varphi_x} = \half \left( \ddp{}{u_x}-i\ddp{}{v_x} \right)
  ,\quad
  \ddp{}{\bar\varphi_x} = \half \left( \ddp{}{u_x}+i\ddp{}{v_x} \right).
\end{equation}
It is easy to check that
\begin{equation}
  \ddp{\varphi_y}{\varphi_x} = \ddp{\bar\varphi_y}{\bar\varphi_x} = \delta_{xy}
  ,\quad
  \ddp{\bar\varphi_y}{\varphi_x} = \ddp{\varphi_y}{\bar\varphi_x} = 0
  .
\end{equation}
The following integration by parts formula will be useful.

\begin{lemma} \label{lem:IBP}
  For $C$ with positive Hermitian part, and for nice functions $F$,
  \begin{equation}
  \lbeq{IBP}
    \int \bar\varphi_a F \; d\mu_C(\varphi, \bar\varphi) = \sum_{x=1}^M C_{ax} \int \ddp{F}{\varphi_x} \; d\mu_C(\varphi, \bar\varphi).
  \end{equation}
\end{lemma}

\begin{proof}
  Integrating by parts, we obtain
  \begin{align}
    \int \ddp{F}{\varphi_x} e^{-\varphi A \bar\varphi} \; d\bar\varphi \, d\varphi
    &= - \int F \ddp{}{\varphi_x} e^{-\varphi A \bar\varphi} \; d\bar\varphi \, d\varphi
    \notag\\
    &=  \int F \sum_{y} A_{xy} \bar\varphi_y e^{-\varphi A \bar\varphi}
    \; d\bar\varphi \,d\varphi .
  \end{align}
  It follows from the fact that $C=A^{-1}$ that
  \begin{equation}
    \sum_{x=1}^M C_{ax} \int \ddp{F}{\varphi_x} \; d\mu_C
    = \int \sum_{x,y} C_{ax}A_{xy} \bar\varphi_y F \; d\mu_C
    = \int \bar\varphi_a F \; d\mu_C .
   \qedhere
  \end{equation}
\end{proof}

The following application of Lemma~\ref{lem:IBP} is a special case of
Wick's Theorem.
The quantity appearing on the right-hand side
of \refeq{Wick} is the \emph{permanent} of the submatrix of $C$
indexed by $(x_i, y_j)_{i,j=1}^k$.

\begin{lemma}
\label{lem:Wick}
  Let $\{x_1, \dots, x_k\}$ and $\{y_1, \dots, y_k\}$ each be sets with $k$ distinct elements from $\{1, \dots, M\}$.
  Then
  \begin{equation}
  \lbeq{Wick}
    \int \prod_{l=1}^k \bar\varphi_{x_l}\varphi_{y_l} d\mu_C
    = \sum_{\sigma \in S_k} \prod_{l=1}^k C_{x_l, y_{\sigma(l)}},
  \end{equation}
  where the sum is over the set $S_k$ of permutations of $\{1,\ldots,k\}$.
\end{lemma}

\begin{proof}
  This follows by repeated application of the integration by parts
  formula in Lemma~\ref{lem:IBP}.  Each time the formula is applied,
  one factor of $\bar\varphi$ disappears on the right-hand side of
  \refeq{IBP}, and the partial differentiation
  eliminates one factor $\varphi$ as well.
\end{proof}

\subsection{Integral representation for a loop model}
\label{sec:G-saw-loops}

Let $\Lambda$ be a finite set of cardinality $M$.
Fix $a,b \in \Lambda$ and a subset $X \subset \Lambda \setminus \{a,b\}$.
An example we have in mind
is $\Lambda \subset \Z^d$ and $X=\Lambda \setminus \{a,b\}$.
We define the integral
\begin{equation}
\lbeq{loopab}
  G_{ab,X} = \int \bar\varphi_a \varphi_b \prod_{x \in X} (1+\varphi_x\bar\varphi_x) \; d\mu_C.
\end{equation}
As we now explain, this can be interpreted as a loop model whose configurations
consist of a self-avoiding walk from $a$ to $b$ whose intermediate steps
lie in $X$, together with a background of closed loops in $X$.
We denote by $\SAWs_{ab}(X)$ the set of sequences $(a,x_1,\ldots,x_{n-1},b)$
with $n\ge 1$ arbitrary and the $x_i \in X$ distinct---these are SAWs with
rather general steps.

Repeated integration by parts gives
\begin{equation}
  G_{ab,X} = \sum_{\omega \in \SAWs_{ab}(X)} C^\omega \int \prod_{x \in X \setminus \omega} (1+\varphi_x\bar\varphi_x) \; d\mu_C,
\end{equation}
where $C^\omega = \prod_{i=1}^{\ell(\omega)} C_{w(i-1),w(i)}$. Also,
by expanding the product and applying Lemma~\ref{lem:Wick}, we obtain
\begin{align}
  \int \prod_{x \in X \setminus \omega} (1+\varphi_x\bar\varphi_x) \; d\mu_C
  &= \sum_{Z \subset X \setminus \omega} \int \prod_{x \in Z} \varphi_x\bar\varphi_x \; d\mu_C \nonumber \\
  &= \sum_{Z \subset X \setminus \omega} \sum_{\sigma \in S(Z)} \prod_{z \in Z} C_{z,\sigma(z)},
\end{align}
with $S(Z)$ is the set of permutations of the set $Z$.
Altogether, this gives
\begin{align}
  G_{ab,X} &= \sum_{\omega \in \SAWs_{ab}(X)} C^\omega
  \sum_{Z \subset X \setminus \omega} \sum_{\sigma \in S(Z)} \prod_{z \in Z} C_{z,\sigma(z)}.
\end{align}
%The last product factors over the cycles of the permutation $\sigma$.

%Thus \refeq{loopab} can be interpreted as the generating function for
%self-avoiding walks from $a$ to $b$ in a background of loops
%with weight $C_{xy}$ for every step between $x$ and $y$,
%and one loop for each cycle.
%See Figure~\ref{fig:SAW-loops}.

Thus, by decomposing the permutation $\sigma$ into cycles, we can interpret \refeq{loopab} as the generating function for
self-avoiding walks from $a$ to $b$ in a background of loops
with weight $C_{xy}$ for every step between $x$ and $y$ (with each loop corresponding to a cycle of $\sigma$).
See Figure~\ref{fig:SAW-loops}.

\begin{figure}[htb]
  \centering
  \input{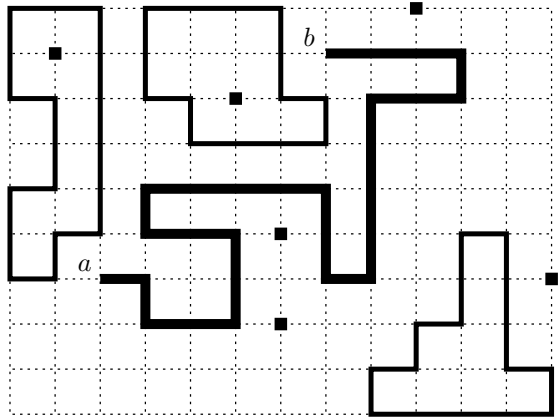}
  \caption{Self-avoiding walk from $a$ to $b$ with loop background.
    Loops can have length zero.
    The loops will be eliminated by the use of differential forms.}
  \label{fig:SAW-loops}
\end{figure}

%\subsection{Differential forms and bosonic-fermionic Gaussian integration}
\subsection{Differential forms}
\label{s:Gaussian-forms}

Our next goal is to modify the example
of Section~\ref{sec:G-saw-loops} with the help of differential
forms, which are versions of what physicists
call \emph{fermions}, to
obtain an integral representation for
the generating function for self-avoiding walks \emph{without} the
loop background.  A gentle introduction to differential forms can be
found in \cite{Rudi76}.

The \emph{Grassmann algebra} ${\mathcal N}$ of differential forms is generated by the
one-forms $du_1, dv_1, \dots, du_M, dv_M$,
with anticommutative product $\wedge$. A $p$-form (a differential
form of degree $p$) is a function of the variables $(u,v)$ times a product of
$p$ differentials or sum of these.
%The maximal degree is $2M$,
%due to anticommutativity.
Because of anticommutativity, $du_x \wedge du_x = dv_x \wedge dv_x = 0$, and any $p$-form with $p>2M$ must be zero. A form
of maximal degree can thus be written uniquely as
\begin{equation} \label{eq:K-maximal}
  K = f(u,v) \; du_1 \wedge dv_1 \wedge \cdots \wedge du_M \wedge dv_M,
\end{equation}
where $du_1 \wedge dv_1 \wedge \cdots \wedge du_M \wedge dv_M$ is
the standard volume form on $\R^{2M}$.
A general differential form is a linear combination of $p$-forms,
where different terms in the sum can have different values of $p$.
Together, the differential forms constitute the algebra ${\mathcal N}$.

We will omit the wedge $\wedge$ from the notation from now on, and write
simply $du_1 \, dv_1$ for $du_1 \wedge dv_1$, but it should be borne in
mind that
%the multiplication of differentials is anticommutative:
%interchanging the order of two adjacent factors introduces a minus sign,
%and $du_i \wedge du_i = 0$.  On the other hand, this implies that
%two forms of even degree do commute.
order is significant in such an expression: $du_x \, dv_y = -dv_y \, du_x$.
On the other hand, two forms of \emph{even} degree commute.

We again use complex variables, and write
\begin{align} \label{eq:boson-def}
\begin{split}
  \varphi_x = u_x  + i v_x
  , \qquad&
  \bar\varphi_x = u_x -i v_x
  ,
  \\
  d\varphi_x = du_x  + i \, dv_x
  ,\qquad&
  d\bar\varphi_x = du_x -i \, dv_x
  .
\end{split}
\end{align}
Then
\begin{equation}
  d\bar\varphi_x \, d\varphi_x = 2i \, du_x \, dv_x
  .
\end{equation}
Given any fixed choice of the complex square root, we introduce the notation
\begin{equation} \label{eq:fermion-def}
  \psi_x = \frac{1}{\sqrt{2\pi i}} d\varphi_x
  ,\quad
  \bar\psi_x = \frac{1}{\sqrt{2\pi i}} d\bar\varphi_x
  .
\end{equation}
The collection of differential forms
\begin{equation}
%\label{}
(\psi, \bar\psi) = (\psi_x, \bar\psi_x)_{x\in\{1,\dots,M\}}
\end{equation}
is called the \emph{fermion} field.
It follows that
\begin{equation}
  \bar\psi_x \psi_x = \frac{1}{\pi} du_x \, dv_x.
\end{equation}

Let $\Lambda = \{1,\ldots,M\}$.
Given an $M \times M$ matrix $A$, we define the differential form
\begin{equation}
\lbeq{SAdef}
  S_A = \varphi A \bar\varphi + \psi A \bar\psi
  = \sum_{x,y \in\Lambda} \varphi_x A_{xy} \bar\varphi_y
  + \sum_{x,y\in\Lambda} \psi_x A_{xy} \bar\psi_y
  .
\end{equation}
An example of special interest is the case where
$A_{uv} = \delta_{ux}\delta_{vx}$ for some fixed
$x \in \Lambda$.  In this case, we write $\tau_x$ in place of $S_A$,
i.e.,
\begin{equation}
\lbeq{tauxdef}
  \tau_x = \varphi_x \bar\varphi_x + \psi_x \bar\psi_x
  .
\end{equation}

\subsection{Functions of forms and integrals of forms}
\label{s:functions-forms}

%Many constructions for the ordinary ``bosonic'' field $(\varphi,\bar\varphi)$
%have fermionic counterparts, even though their probabilistic interpretation is lost.
%%
%In particular, functions of (even) differential forms will serve as a replacement for random variables.

The following definition tells us how to integrate a differential form.

\begin{definition} \label{def:int-K}
  Let $F$ be a differential form whose term
  $K$ of maximal degree is as in \eqref{eq:K-maximal}.
  The \emph{integral of} $F$ is then defined to be
  \begin{equation}
    \int F = \int K =
    \int_{\R^{2M}} f(u,v) \; du_1 \, dv_1 \cdots du_M \, dv_M.
  \end{equation}
  % \begin{equation}
  %   \int K = \begin{cases}
  %     \int_{\R^{2M}} f(u,v) \; du_1 \, dv_1 \cdots du_M \, dv_M & \text{if $K$ as above}\\
  %     0 & \text{if $K$ is a $p$-form with $p < 2M$}
  %   \end{cases}
  % \end{equation}
  In particular, if $F$ contains no term of degree $2M$ then its integral is zero.
\end{definition}

We also need to define functions of even differential forms.

\begin{definition}
  Let $K = (K_j)_{j\in J}$ be a finite collection
  of differential forms, with each $K_j$ even (a sum of forms of even degrees).
  Let $K_j^{(0)}$ be the degree zero part of $K_j$.
  Given a $C^\infty$ function $F : \R^J \to \C$, we define $F(K)$ to
  be the form given by the Taylor polynomial (a polynomial in $\psi$
  and $\bar\psi$)
  \begin{equation} \label{eq:def-F(K)}
    F(K) = \sum_{\alpha} \frac{1}{\alpha!} F^{(\alpha)}(K^{(0)})(K-K^{(0)})^\alpha
  \end{equation}
  where $\alpha = (\alpha_1, \dots, \alpha_j)$ is a multi-index and
  \begin{equation} \label{eq:def-F(K)-2}
    \alpha! = \prod_{j \in J} \alpha_j!
    ,\quad
    (K-K^{(0)})^\alpha = \prod_{j \in J} (K_j - K^{(0)}_j)^{\alpha_j}
    .
  \end{equation}
  The sum in \eqref{eq:def-F(K)} is finite due to anticommutativity,
  and the product in \eqref{eq:def-F(K)-2} is well-defined because
  all factors are even and thus commute.
\end{definition}

\begin{example}
  A simple but important example is $J=1$ and $F(t) = e^{-t}$, for which
  we obtain, e.g.,
  \begin{align}
    e^{-\tau_x} &= e^{-\varphi_x\bar\varphi_x-\psi_x\bar\psi_x}
    = e^{-\varphi_x\bar\varphi_x}(1-\psi_x\bar\psi_x)
    ,
  \\
  \label{eq:exp-S_A}
    e^{-S_A} &= e^{-\varphi A \bar\varphi - \psi A \bar\psi} = e^{-\varphi A \bar\varphi} \sum_{n=0}^M \frac{(-1)^n}{n!} (\psi A \bar\psi)^n
    .
  \end{align}
\end{example}

The following lemma displays a remarkable self-normalisation property
of these integrals.

\begin{lemma}
  \label{lem:self-normalisation}
  If $A$ is a complex $M\times M$ matrix with positive Hermitian part, then
  \begin{equation}
    \int e^{-S_A} = 1.
  \end{equation}
\end{lemma}

\begin{proof}
  Using \eqref{eq:exp-S_A} and Definition~\ref{def:int-K},
  \begin{align}
    \int e^{-S_A}
    &= \int_{\R^{2M}} e^{-\varphi A\bar\varphi} \frac{1}{M!} (-1)^M
    (\psi A \bar\psi)^M\nonumber\\
    &= \frac{1}{M!}
    \left(\frac{-1}{2\pi i}\right)^M \int_{\R^{2M}} e^{-\varphi A \bar\varphi}
    (d\varphi A d\bar\varphi)^M.
  \end{align}
  By definition,
  \begin{align}
    (d\varphi A d\bar\varphi)^M
    &= \sum_{x_1,y_1} \cdots \sum_{x_M,y_M} A_{x_1 y_1} \cdots A_{x_M y_M}
    \; d\varphi_{x_1}\,d\bar\varphi_{y_1} \cdots d\varphi_{x_M}\, d\bar\varphi_{y_M}.
  \end{align}
  Due to the antisymmetry, non-zero contributions to the above sum require
  that $x_1,\ldots,x_M$ and $y_1,\ldots, y_M$ each be a permutation of
  $\{1,\ldots,M\}$.  Thus, by interchanging the (commuting) pairs
  $d\varphi_{x_i}d\bar\varphi_{y_i}$
  so as to place the $x_i$ in the order $1,\ldots,M$, and
  then relabelling the $y_i$, we obtain
  \begin{align}
    (d\varphi A d\bar\varphi)^M
    &=   M! \sum_{y_1,\dots, y_M} A_{1 y_1} \cdots A_{M y_M} \; d\varphi_{1}\,d\bar\varphi_{y_1} \cdots d\varphi_{M} \, d\bar\varphi_{y_M} \nonumber\\
    &= M! \sum_{y_1,\dots, y_M} \epsilon_{y_1, \dots, y_M}  A_{1 y_1} \cdots A_{M y_M} \; d\varphi_{1}\, d\bar\varphi_{1} \cdots d\varphi_{M} \, d\bar\varphi_{M} \nonumber\\
    &= M! \, (-1)^M (\det A) \, d\bar\varphi \, d\varphi,
  \end{align}
  where $\epsilon_{y_1, \dots, y_M}$ is the sign of the
  permutation $(y_1, \dots, y_M)$ of $\{1, \dots, M\}$.
  With Lemma~\ref{lem:ZC}, it follows that
  \begin{equation}
    \int e^{-S_A} = \frac{\det A}{(2\pi i)^M} \int_{\R^{2M}} e^{-\varphi A \bar\varphi} \; d\bar\varphi\, d\varphi = 1. \qedhere
  \end{equation}
\end{proof}

\begin{remark}
\label{rem:gi}
  More generally, the calculation in the previous proof
  also shows that for a function $f=f(\varphi,\bar\varphi)$,
  a form of degree zero,
  \begin{equation}
    \int e^{-S_A} f = \int f \; d\mu_C  \qquad (C=A^{-1}),
  \end{equation}
  provided $f$ is such that the integral on the right-hand side converges.
  In our present setup, we have defined
  $\int e^{-S_A} F$  for more general forms $F$, so this provides
  an extension of the Gaussian integral of
  Section~\ref{s:Gaussian-integrals}.
\end{remark}

The self-normalisation property of Lemma~\ref{lem:self-normalisation}
has the following beautiful extension.  The precise hypotheses needed
on $F$ can be found in \cite[Proposition~4.4]{BIS09}.

\begin{lemma}
  \label{lem:extension-normalisation}
  If $A$ is a complex $M\times M$ matrix with positive Hermitian part,
  and $F:\R^M \to \C$ is a nice function (exponential growth at infinity is
  permitted), then
  \begin{equation}
    \int e^{-S_A} F(\tau) = F(0)
    ,
  \end{equation}
  where we regard $\tau$ as the vector $(\tau_1,\ldots,\tau_M)$.
\end{lemma}

\begin{proof}[Proof (sketch)]
  If $F$ is Schwartz class, e.g., then it can be expressed in
  terms of its Fourier transform as
  \begin{equation}
    F(t) = \frac{1}{(2\pi)^M} \int_{\R^M} \hat F(k) e^{-i k \cdot t} \; dk_1 \cdots dk_M.
  \end{equation}
  It then follows that
  \begin{equation}
    \int e^{-S_A} F(\tau) = \frac{1}{(2\pi)^M} \int \hat F(k) \left( \int e^{-S_A-ik\cdot \tau}\right) dk = F(0)
  \end{equation}
  because $S_A + ik\cdot \tau = S_{A+iK}$ with $K = \mathrm{diag}(k_x)_{x=1}^M$,
  and thus $\int e^{-S_{A+iK}} = 1$ by Lemma~\ref{lem:self-normalisation}.
%  This argument can be modified to allow for $F$ which are not Schwartz-class by using part of the exponential
%  from the Gaussian measure; see \cite{BIS09}.
\end{proof}

It is not difficult to extend
the integration by parts formula for Gaussian measures,
Lemma~\ref{lem:IBP}, to the present more general setting;
see \cite{BIS09} for details.
The result is the following.

\begin{lemma}
\label{lem:IBPfermions}
  For $a\in \Lambda$,
  for $C=A^{-1}$ with positive Hermitian part, and for forms $F$ for which the
  integrals exist,
  \begin{equation}
    \int e^{-S_A} \bar\varphi_a F =
    \sum_{x \in \Lambda} C_{ax} \int e^{-S_A} \ddp{F}{\varphi_x}.
  \end{equation}
\end{lemma}

\subsection{Integral representation for self-avoiding walk}
\lbsect{repsaw}

Let $\Lambda$ be a finite set and
let $a,b \in \Lambda$.
In Section~\ref{sec:G-saw-loops},
we showed that the integral
\begin{equation}
  \int \bar\varphi_a \varphi_b \prod_{x \neq a,b}
  (1+\varphi_x\bar\varphi_x) \; d\mu_C
\end{equation}
is the generating function for
SAWs in a background
of self-avoiding loops.
The following theorem shows that the loops are eliminated
if we replace the factors $(1+\varphi_x\bar\varphi_x)$ by
$(1+\tau_x) = (1+ \varphi_x\bar\varphi_x + \psi_x\bar\psi_x)$
and replace the Gaussian measure $d\mu_C$ by $e^{-S_A}$ with $A=C^{-1}$.

\begin{theorem}
\label{thm:repsaw}
  For $C=A^{-1}$ with positive Hermitian part, and for $a,b \in \Lambda$,
  \begin{equation}
  \lbeq{sawrep}
    \sum_{\omega \in \SAWs_{a,b}(\Lambda)} C^\omega
    =
    \int e^{-S_A} \bar\varphi_a \varphi_b \prod_{x \neq a,b} (1+\tau_x) .
  \end{equation}
\end{theorem}

\begin{proof}
  Exactly as in Section~\ref{sec:G-saw-loops},
   but now using the integration by parts formula
  of Lemma~\ref{lem:IBPfermions}, we obtain
  \begin{equation}
    \int e^{-S_A} \bar\varphi_a \varphi_b \prod_{x \neq a,b} (1+\tau_x)
    = \sum_{\omega \in \SAWs_{a,b}(\Lambda)} C^\omega
    \int e^{-S_A} \prod_{x \in \Lambda\setminus \omega} (1+\tau_x) .
  \end{equation}
  However, the integral on the right-hand side, which formerly generated
  loops, is now equal to $1$ by Lemma~\ref{lem:extension-normalisation}.
\end{proof}

\section{Renormalisation group analysis in dimension \texorpdfstring{$4$}{4}}
% Note: \texorpdfstring{a}{b} indicates that hyperref should substitute ASCII text b for the LaTeX text a
%       (in the .pdf file table of contents, for instance)
\lbsect{ctwsaw}

The integral representation of Theorem~\ref{thm:repsaw} opens up
the following possibility for studying SAWs on $\Z^d$:  approximate
$\Z^d$ by a large finite set $\Lambda$, rewrite the SAW two-point
function as an integral as in \refeq{sawrep}, and apply methods
of analysis to compute the asymptotic behaviour of the integral
uniformly in the limit $\Lambda \nearrow \Z^d$.
In this section, we sketch how such a program can be carried out
for a particular model of
continuous-time weakly SAW on the 4-dimensional lattice $\Z^4$,
using a variant of Theorem~\ref{thm:repsaw}. In this approach,
once the integral representation has been invoked, the original SAWs
no longer appear and play no further role in the analysis.
The method of proof is a rigorous renormalisation group
method \cite{BS11,BS11a}.
There is work in progress, not discussed further here, to
attempt to extend
this program to a particular spread-out
version of the discrete-time strictly SAW model on $\Z^4$
using Theorem~\ref{thm:repsaw}.

We begin in Section~\ref{sec:ctwsaw} with the definition of the
continuous-time weakly SAW and a statement of the main result
for its two-point function, followed by some commentary on related
results.  The approximation of the two-point function on $\Z^d$ by
a two-point function on a $d$-dimensional finite torus $\Lambda$ is
discussed in Section~\ref{sec:fv}, and the integral representation
of the two-point function on $\Lambda$ is explained in Section~\ref{sec:ir}.
The discussion of integration of differential forms from
Section~\ref{s:functions-forms} is developed further in
Section~\ref{sec:s-e}.  At this point, the stage is set for
the application of the renormalisation group method, and this
is described briefly in
Sections~\ref{sec:decomp}--\ref{sec:remsteps}.
A more extensive account of all this can be found in \cite{BS11,BS11a}.

\subsection{Continuous-time weakly self-avoiding walk}
\label{sec:ctwsaw}

The definition of the discrete-time weakly
self-avoiding walk was given in \refSSect{SAW}.
With an unimportant change in our conventions, and
writing $z=e^{-\nu}$ and using the
parameter $g>0$ of  \refeq{cnlambdaxExp}
rather than $\lambda$,
the two-point function \refeq{TwoPointDefn} can be rewritten as
\begin{equation}
\lbeq{GDT}
    G_\nu^{(g),{\rm DT}}(x)
    =
    \sum_{n=0}^\infty \sum_{\omega \in \Walks_n(0,x)}
    \exp\Bigl( -g \sum_{i,j=0}^{n} \indicator{\omega(i)=\omega(j)} \Bigr)
    e^{-\nu n}
    ,
\end{equation}
where ``${\rm DT}$'' emphasises the fact that the walks are in discrete time.
The \emph{local time} at $v\in \Z^d$ is defined as the number of visits to $v$ up
to time $n$, i.e.,
\begin{equation}
    L_{v,n} = L_{v,n}(\omega) = \sum_{i=0}^n \indicator{\omega(i)=v}.
\end{equation}
Note that $\sum_{v\in\Z^d} L_{v,n} = n$ is independent of the walk $\omega$, and that
\begin{align}
  \sum_{v\in\Z^d} L_{v,n}^2
  &= \sum_{v\in\Z^d} \sum_{i,j=0}^n   \indicator{\omega(i)=v} \indicator{\omega(j)=v}
  =\sum_{i,j=0}^n    \indicator{\omega(i)=\omega(j)}.
\end{align}
Thus,  writing $z=e^{-\nu}$,
the two-point function can be rewritten as
\begin{equation}
\lbeq{Ggkappa}
    G_{\nu}^{(g),{\rm DT}}(x)
    =
    \sum_{n=0}^\infty \sum_{\omega \in \Walks_n(0,x)}
    e^{ -g \sum_{v\in \Z^d} L_{v,n}^2 }
    e^{-\nu n}.
\end{equation}

The two-point function of the
continuous-time weakly SAW is a modification of
\refeq{Ggkappa} in which the
underlying random walk model has continuous, rather than discrete, time.
To define the modification, we consider the
\emph{continuous-time} random walk $X$ which takes nearest-neighbour steps
like the usual SRW, but whose jumps occur after independent
${\rm Exp}(2d)$ holding times at each vertex.  In other words, the
steps occur at the events of a rate-$2d$ Poisson process, rather than
at integer times.  We write $\E_0$ for the expectation associated to the
process $X$ started at $X(0)=0\in\Z^d$.
The local time of $X$ at $v$ up to time $T$ is now defined by
  \begin{equation}
    L_{v,T} =\int_0^T \indicator{X(s)=v} \; ds.
  \end{equation}

The probabilistic structure of
\refeq{cnlambdaxDefn}--\refeq{cnlambdaExpectation} extends
naturally to the continuous-time setting.
%In analogy with \refeq{Ggkappa}, we define
With this in mind, we define the \emph{two-point function}
of continuous-time weakly SAW
by
\begin{equation}
\lbeq{Ggnux}
  G^{(g)}_{\nu}(x) = \int_0^\infty \E_0(e^{-g\sum_v L_{v,T}^2} \indicator{X(T)=x})
  e^{-\nu T} \; dT ;
\end{equation}
this is a natural modification of \refeq{GDT}.
%\end{definition}
%This definition makes sense on any graph where the continuous-time simple random walk and
%local time is defined; in particular, on $\Z^d$ or the discrete torus $\Z^d / R \Z^d$.
The continuous-time SAW is predicted to lie in
the same universality class as the discrete-time SAW.

Using a subadditivity argument as in \refSect{Subadd}, it is
not difficult to see that the limit
\begin{equation}
  \lim_{T\to\infty} \left( \E_0(e^{-g\sum_v L_{v,T}^2}) \right)^{1/T}
  = e^{\nu_c(g)}
  %\quad \text{exists,}
\end{equation}
exists, for some $\nu_c(g) \le 0$.
We leave it as an exercise to show that $\nu_c(g) > -\infty$.
In particular,
%$G^{\mathrm{wsaw}}_{ab}(g, \nu)$
$G^{(g)}_{\nu}(x)$
is well-defined
for $\nu > \nu_c(g)$.  The following theorem of Brydges and
Slade \cite{BS11,BS11a} shows that the critical exponent
$\eta$ is equal to $0$ for this model, in dimensions $d \ge 4$.

\begin{theorem}
  \label{thm:BS-wsaw}
  Let $d \geq 4$. For $g\geq 0$ sufficiently small, there exists $c_g > 0$ such that
  \begin{equation}
    G^{(g)}_{\nu_c(g)}(x) = \frac{c_g}{|x|^{d-2}}(1+o(1))
    \quad \text{as $|x|\to\infty$.}
  \end{equation}
\end{theorem}

Theorem~\ref{thm:BS-wsaw} should be compared with the result of
Theorem~\ref{thm:SAWLaceGreenConclusion} for $d\geq 5$.
The main point in Theorem~\ref{thm:BS-wsaw} is the inclusion of the
upper critical
dimension $d=4$.
In particular, there is no logarithmic correction to the leading
asymptotic behaviour of the critical two-point function when $d=4$.
The case $g=0$ is the classical result that the
SRW Green
function obeys $G^{(0)}_0(x) \sim c_0 |x|^{-(d-2)}$, which in fact
holds in all dimensions $d>2$.

The proof of Theorem~\ref{thm:BS-wsaw} is based on an
integral representation combined with a rigorous renormalisation
group method, and is inspired by the methods used in
\cite{BEI92,BI03c,BI03d} for the continuous-time weakly self-avoiding
walk on the 4-dimensional
\emph{hierarchical lattice}.  The hierarchical lattice is
a modification of the lattice $\Z^d$ that is particularly amenable
to a renormalisation group approach.  It is predicted that the models
on the hierarchical lattice and $\Z^d$ lie in the same universality class.
Strong evidence for  this is the result of
Brydges and Imbrie \cite{BI03c} that on the
4-dimensional hierarchical lattice the typical
end-to-end distance after time $T$ is given,
for small $g>0$ and as $T \to\infty$, by
\begin{equation}
  \frac{\E_0(|\omega(T)|\, e^{-g\sum_v L_{v,T}^2})}
  {\E_0(e^{-g\sum_v L_{v,T}^2})}= c\, T^{1/2} (\log T)^{1/8}
  \left[ 1+ \frac{\log\log T}{32 \log T}  + O\left(\frac{1}{\log T}\right)\right].
\end{equation}
This matches the prediction \refeq{nuPrediction4} for $\Z^4$.
There are related results by Hara and Ohno  \cite{HO10}, proved
with a completely different renormalisation group approach,
for the critical two-point function, susceptibility and correlation
length of the \emph{discrete}-time weakly self-avoiding walk on the $d$-dimensional
hierarchical lattice for $d \ge 4$.

Recently, Mitter and Scoppola \cite{MS08} used the integral
representation and renormalisation group analysis to study a
continuous-time weakly self-avoiding walk with long-range steps.
In the model of \cite{MS08},
each step of length $r$ has a weight decaying like $r^{-d-\alpha}$, with
$\alpha = \frac{1}{2}(3+\epsilon)$ for small $\epsilon >0$,
in dimension $d=3$.  This is
\emph{below} the upper critical dimension $2\alpha = 3+\epsilon$
(recall the discussion below Theorem~\ref{thm:SAWLaceGreenConclusion}).
The main result is a control of the renormalisation
group trajectory, a first step towards the computation of the
asymptotic
behaviour of the critical two-point function below the upper critical
dimension.
This is a rigorous version, for the weakly self-avoiding walk, of the
expansion in $\epsilon = 4-d$ discussed in \cite{WK74}.

\subsection{Finite-volume approximation}
\label{sec:fv}

Integral representations of the type discussed in \refSect{repsaw}
are for walks on a finite set.  In preparation for the
integral representation, we first discuss the approximation
of the two-point function $G_{\nu_c}^{(g)}(x)$
on $\Z^d$ by a two-point function
on the finite torus $\Lambda = \Z^d / R\Z^d$
with side length $R \in \Z_+$.  For later convenience, we will always
take $R=L^N$ with $L$ a large dyadic integer.
The parameter $g$ is
regarded as a fixed positive number and will sometimes be omitted
in what follows, to simplify
the notation.
We denote by $G^\Lambda$ the natural
modification of \refeq{Ggnux} in which the random walk on $\Z^d$
is replaced by the random walk on $\Lambda$.

\begin{theorem}
\label{thm:SL}
  Let $d\geq 1$, $g >0$, and $x\in \Z^d$. Then for all $\nu \ge \nu_c$,
  \begin{equation}
    G_{\nu}(x)
    =
    \lim_{\nu' \searrow \nu} \lim_{N \to \infty} G^\Lambda_{\nu'}(x),
  \end{equation}
  where, on the right-hand side, $x$ is the canonical representative of $x$
  in $\Lambda$ for $L^N$ large compared to $x$.
\end{theorem}

\begin{proof}
  This follows from a version of the Simon--Lieb inequality
  \cite{Simo80,Lieb80}
  for the continuous-time weakly self-avoiding walk.
  In the problems of
  \refSect{Tut4} below, we develop the corresponding argument in the discrete-time
  setting.  With a little more work, the same approach can be adapted to continuous
  time.
\end{proof}

We are most interested in the case $\nu=\nu_c$ in
Theorem~\ref{thm:SL}.  The theorem allows for
the study of the critical two-point function on $\Z^d$ via the
subcritical two-point function in finite volume,
provided sufficient control is maintained to take the limits.
Since SRW is recurrent in finite volume, its Green function
is infinite, and the flexibility of taking $\nu$ slightly larger
than $\nu_c$ helps bypass this concern.

\subsection{Integral representation}
\label{sec:ir}

We recall the introduction of the boson field
$(\varphi_x,\bar\varphi_x)$ in \eqref{eq:boson-def} and
the fermion field $(\psi_x,\bar\psi_x)$ in \eqref{eq:fermion-def}, and now
index these fields with $x$ in the torus
$\Lambda = \Z^d / L^N\Z^d$.
%which we identify with the set $\{1, \dots, |\Lambda|\}$.
We also recall from \refeq{tauxdef} the definition, for $x \in \Lambda$,
of the differential form
\begin{equation}
  \tau_x = \varphi_x\bar\varphi_x + \psi_x\bar\psi_x.
\end{equation}
The Laplacian $\Delta$
applies to the boson and fermion fields according to
\begin{equation}
  (\Delta\varphi)_x = \sum_{y:y\sim x} (\varphi_y -\varphi_x),\quad
  (\Delta\psi)_x = \sum_{y:y \sim x} (\psi_y -\psi_x),
\end{equation}
where the sum is over the neighbours $y$ of $x$ in the torus $\Lambda$.
We also define the differential forms
\begin{equation}
  \tau_{\Delta,x} = \frac{1}{2}(
  \varphi_x(-\Delta\bar\varphi)_x + (-\Delta\varphi)_x\bar\varphi_x
  + \psi_x(-\Delta\bar\psi)_x + (-\Delta\psi)_x\bar\psi_x
  ).
\end{equation}
The following theorem is proved in \cite{BI03c}; see also
\cite[Theorem~5.1]{BIS09} for a self-contained proof.  Its requirement that
$G_\nu^\Lambda(x)<\infty$ for large $\Lambda$
is a consequence of Theorem~\ref{thm:SL}.

\begin{theorem}
  For $\nu > \nu_c$ and $0,x\in\Lambda$,
  and for $\Lambda$ large enough that $G_\nu^\Lambda(x)<\infty$,
  the finite-volume two-point function has the
  integral representation
  \begin{equation}
  \lbeq{Gint}
    G^{\Lambda}_{\nu}(x)
    = \int
    e^{-\sum_{v\in\Lambda}(\tau_{\Delta,v}+g\tau_v^2+\nu\tau_v)}
    \bar\varphi_0\varphi_x
    .
  \end{equation}
\end{theorem}

It is the goal of the method to show that the infinite-volume
critical two-point function is asymptotically equal to a multiple
of the inverse Laplacian on $\Z^d$, for $d \ge 4$.  To exhibit
an explicit factor to account for this multiple, we introduce
a parameter $z_0>-1$
by making the
change of variables $\varphi_x \mapsto (1+z_0)^{1/2} \varphi_x$.
With this change of variables, the integral representation \refeq{Gint}
becomes
\begin{equation}
  G^{\Lambda}_{\nu}(x) = (1+z_0) \int e^{-S(\Lambda)}
  e^{-\widetilde V_0(\Lambda)} \bar\varphi_0\varphi_x ,
\end{equation}
where
\begin{align}
  S(\Lambda) &= \sum_{v\in\Lambda} (\tau_{\Delta,v}+m^2\tau_v),\\
  \widetilde V_0(\Lambda)
  &= \sum_{v\in\Lambda} (g_0\tau_v^2+\nu_0\tau_v+z_0\tau_{\Delta,v}),
\end{align}
with
\begin{equation}
  g_0 = (1+z_0)^2g, \quad \nu_0 = (1+z_0)\nu_c, \quad m^2 = (1+z_0)(\nu-\nu_c).
\end{equation}
In particular,
the limit $\nu \searrow \nu_c$ corresponds to $m^2 \searrow 0$.

It is often convenient in statistical mechanics to obtain a correlation
function by differentiation of a partition function with respect to
an external field, and we will follow this approach here.
Introducing an \emph{external field} $\sigma \in \C$, we define
\begin{equation}
  V_0(\Lambda) =
  \widetilde V_0(\Lambda) + \sigma\bar\varphi_0 + \bar\sigma \varphi_x.
\end{equation}
Then the two-point function is given by
\begin{equation}
\lbeq{Ggoal}
  G^{\Lambda}_{\nu}(x) = (1+z_0) \frac{\partial^2}{\partial\sigma\partial\bar\sigma}\Big|_{\sigma=\bar\sigma=0} \int_{\C^\Lambda} e^{-S(\Lambda)-V_0(\Lambda)} .
\end{equation}
Our goal now is the evaluation of the large-$x$ asymptotic behaviour
of
\begin{equation}
\lbeq{Ggoal2}
    G_{\nu_c}(x)=
  \lim_{m^2 \searrow 0}
  \lim_{N \to \infty}
  (1+z_0) \frac{\partial^2}{\partial\sigma\partial\bar\sigma}\Big|_{\sigma=\bar\sigma=0} \int_{\C^\Lambda} e^{-S(\Lambda)-V_0(\Lambda)} .
\end{equation}

For the case $\widetilde V_0=0$ (so in particular $z_0=0$),
in view of Remark~\ref{rem:gi} the right-hand side becomes
\begin{equation}
\lbeq{G0:1}
  \lim_{m^2 \searrow 0}
  \lim_{N \to \infty}
  \int_{\C^\Lambda} e^{-S(\Lambda)}\bar\varphi_0\varphi_x
  =
  \lim_{m^2 \searrow 0}
  \lim_{\Lambda \nearrow \Z^d}
  \int  \bar\varphi_0\varphi_x d\mu_{(-\Delta_\Lambda+m^2)^{-1}},
\end{equation}
and by Lemma~\ref{lem:Wick} this is equal to
\begin{equation}
\lbeq{G0:2}
  \lim_{m^2 \searrow 0}
  \lim_{\Lambda \nearrow \Z^d}
  (-\Delta_\Lambda+m^2)^{-1}_{0x} = (-\Delta_{\Z^d})^{-1}_{0x}
  \sim c_0|x|^{-(d-2)}
\end{equation}
(we have added subscripts to the Laplacians to emphasise
where they act).  The goal of the forthcoming analysis is
to show that for small $g>0$, and with the correct choice of $z_0$,
the effect of $\widetilde V_0$ is a small perturbation in the sense
that its presence does not change the power in this $|x|^{-(d-2)}$ decay.

\subsection{Superexpectation}
\label{sec:s-e}

We will need some further development of the theory of
integration of differential forms discussed in
Section~\ref{s:functions-forms}.
As before, we denote the algebra of differential forms, now with index
set $\Lambda$, by $\mathcal{N}$.
Let $C$ be a $\Lambda
\times \Lambda$ matrix, with positive-definite Hermitian
part, and with inverse $A=C^{-1}$.
The \emph{Gaussian superexpectation} with covariance matrix $C$ is defined by
\begin{equation}
  \E_CF = \int e^{-S_A} F \quad \text{for $F \in \mathcal{N}$.}
\end{equation}
The name ``superexpectation'' comes from the fact that the integral
representation for the two-point function is actually a
\emph{supersymmetric} field theory; super-symmetry is discussed
in \cite{BIS09}.

Note that, by Lemma~\ref{lem:self-normalisation} and Remark~\ref{rem:gi},
$\E_C 1 = 1$, and more generally $\E_C f = \int fd\mu_C$ if $f$ is a
zero-form.  The latter property shows that the
Gaussian superexpectation
extends the ordinary Gaussian expectation, and we wish to take this further.
Recall the elementary fact that if $X_1 \sim N(0, \sigma_1^2)$ and
$X_2 \sim N(0,\sigma_2^2)$ are independent normal random variables,
then $X_1+X_2 \sim N(0, \sigma_1^2+\sigma_2^2)$.  In particular,
if $X \sim N(0,\sigma_1^2+\sigma_2^2)$ then we can evaluate
$\E(f(X))$ in stages as
\begin{equation}
\lbeq{EX1X2}
    \E(f(X)) = \E( \E( f(X_1+X_2) \, | \, X_2) ).
\end{equation}
It will be a crucial
ingredient of the following analysis that this
has an extension to the superexpectation, as we describe next.

%The latter property shows that the super-expectation is related to usual
%Gaussian integrals, but the relation goes farther than that.
%In particular,
%recall the elementary fact that if $X_1 \sim N(0, \sigma_1^2)$ and
%$X_2 \sim N(0,\sigma_2^2)$ are independent normal random variables,
%then $X_1+X_2 \sim N(0, \sigma_1^2+\sigma_2^2)$.  In other words,
%the convolution of two normal densities is a normal density.
%It will be a crucial ingredient of the following analysis that this
%has an extension to the super-expectation, as we describe next.

By definition,
any form $F \in \mathcal{N}$ is a linear combination of
products of factors $\psi_{x_i}$ and $\bar\psi_{\bar{x}_i}$, with
$x_i,\bar{x}_i\in \Lambda$ and with coefficients given by functions
of $\varphi$ and $\bar\varphi$.  The coefficients may also depend on
the external field $(\sigma,\bar\sigma)$,
but we leave the dependence on $\sigma,\bar\sigma$ implicit
in the notation.
%We write this symbolically as
%\begin{equation}
%  F = \sum_{y \in \Lambda^*} \frac{1}{y!} F_y(\varphi,\bar\varphi) \psi^y
%\end{equation}
%where $\Lambda^*$ consists of all sequences with
%elements in $\Lambda \times \{0,1\}$, and, e.g.,
%for $y = ((y_1,0), (y_1,1), (y_2,0))$,
%\begin{equation}
%  \psi^y = \psi_{y_1}\bar\psi_{y_1}\psi_{y_2}, \quad y! = 3! .
%\end{equation}
%%The coefficients $F_y$ are
%%functions of $\varphi,\bar\varphi,\sigma,\bar\sigma$.
%The form $\psi^y$ is antisymmetric under interchange of neighbouring
%factors, and by convention we take
%the coefficients $F_y$ also to be antisymmetric in the components
%of $y$; then $F_y\psi^y$ is independent of the order of the components
%of $y$ and this multiplicity is cancelled by the factor $1/y!$.
We also define an algebra $\mathcal{N}^\times$
with twice as many fields as $\mathcal{N}$, namely
with boson fields
$(\phi,\xi)$ and fermion fields $(\psi,\eta)$,
where $\phi = (\varphi,\bar\varphi)$,
$\xi = (\zeta,\bar\zeta)$,
$\psi = \frac{1}{\sqrt{2\pi i}} (d\varphi, d\bar\varphi)$,
$\eta = \frac{1}{\sqrt{2\pi i}} (d\zeta, d\bar\zeta)$.
%Then we define $\theta: \mathcal{N} \to \mathcal{N}^\times$ by
%\begin{equation}
%\lbeq{thetadef}
%  \theta F
%  = \sum_{y \in \Lambda^*} \frac{1}{y!}
%  F_y(\phi+\xi) (\psi+\eta)^y.
%\end{equation}
For a form $F=f(\varphi, \bar\varphi) \psi^x \bar\psi^y$
(where $\psi^x$ denotes a product
$\psi_{x_1}\cdots\psi_{x_j}$), we define
\begin{equation}
\lbeq{thetadef}
    \theta F = f(\varphi+\xi,\bar\varphi+\bar\xi) (\psi+\eta)^x (\bar\psi+\bar\eta)^y,
\end{equation}
and we extend this to a map $\theta : \mathcal{N} \to \mathcal{N}^\times$
by linearity.
Then we understand the map $\E_C \circ \theta : \mathcal{N}
\to \mathcal{N}$ as the integration with respect to the
\emph{fluctuation fields} $\xi$ and $\eta$, with the fields $\phi$ and $\psi$
left fixed.  This is like a conditional expectation.
However, this
is not standard probability theory, since $\E_C$ does not arise
from a probability measure and takes
values in the (non-commutative) algebra of forms.

The superexpectation has the following important
convolution property,
analogous to \refeq{EX1X2} (see \cite{BI03c,BS11a}).

\begin{prop}\lbprop{Gaussian-conv}
Let $F \in \mathcal{N}$, and suppose that $C_1$ and $C'$ have positive-definite
Hermitian parts.  Then
  \begin{equation} \label{eq:Gaussian-conv}
    \E_{C'+C_1}F = \E_{C'}(\E_{C_1}\theta F).
  \end{equation}
\end{prop}

Suppose $C$ and $C_{j},\,j=1,\dots ,N,$ are $\Lambda
\times \Lambda$ matrices with positive-definite  Hermitian parts,
such that
\begin{equation}
\label{e:CCN}
    C = \sum_{j=1}^N C_j
    .
\end{equation}
Then, by the above proposition,
\begin{equation}
    \label{e:progressive}
    \E_{C}F
    =
    \big( \E_{C_N}\circ \E_{C_{N-1}}\theta \circ \cdots
    \circ \E_{C_{1}}\theta\big) F
    .
\end{equation}
%and the $\E_{C}$ expectation can be
%performed progressively by applying the proposition iteratively to
%evaluate the $\E_{C'}$ expectation on the right-hand side.
In the next section, we describe a particular choice of
the decomposition \eqref{e:CCN},
which will allow us to control the progressive
integration in \eqref{e:progressive}.

\subsection{Decomposition of the covariance}
\label{sec:decomp}

Our goal is to compute the large-$x$ asymptotic
behaviour of the two-point function using \refeq{Ggoal},
which we can now rewrite as
\begin{equation}
\lbeq{GgoalC}
  G^{\Lambda}_{\nu}(x)
  = (1+z_0)
  \frac{\partial^2}{\partial\sigma\partial\bar\sigma}\Big|_{\sigma=\bar\sigma=0}
  \E_C e^{-V_0(\Lambda)} ,
\end{equation}
with $C=(-\Delta +m^2)^{-1}$.  The Laplacian is on the torus $\Lambda$,
and we must take the limits as $\Lambda$ approaches $\Z^d$ and $m^2$
approaches zero, so $C$ is an approximation to $(-\Delta_{\Z^d})^{-1}$.
The operator $(-\Delta_{\Z^d})^{-1}$ decays as $|x|^{-2}$ in dimension
$d=4$, and such long-range correlations make the analysis difficult.
The renormalisation group approach takes the long-range correlations
into account progressively, by making a good decomposition of the
covariance $C$ into a sum of terms with \emph{finite} range,
together with progressive integration as in \eqref{e:progressive}.
The particular decomposition used is given in the
following theorem, which extends a result of
Brydges, Guadagni and Mitter \cite{BGM04}; see also \cite{Bryd09,BS11a}.
In its statement,
$\nabla_x^\alpha=\nabla_{x_1}^{\alpha_1} \dotsb \nabla_{x_d}^{\alpha_d}$ for a
multi-index $\alpha=(\alpha_1,\dotsc,\alpha_d)$, where
$\nabla_{x_k}$ denotes the
finite-difference operator $\nabla_{x_k}f(x,y)=f(x+e_k,y)-f(x,y)$.

% Note that with $C=(-\Delta+m^2)^{-1}$,
% \begin{equation}
%   \int_{\C^\Lambda} e^{-(\Lambda)-V_0(\Lambda)} = \E_C e^{-V_0(\Lambda)}
% \end{equation}
% The principle idea of the renormalisation group method is to decompose the Gaussian integral
% into an iteration using \eqref{eq:Gaussian-conv}, and a good decomposition of the convariance.
% The following decomposition was originally obtained for the Green function on $\Z^d$, but can be
% extended to apply to the torus $\Lambda$.

\begin{theorem}
\label{thm:Cdecomp}
  Let $d>2$ and $N\in \Z_+$,
  and let $\Lambda$ be the torus $\Z^d/L^N\Z^d$, with $L$ a
  sufficiently large dyadic integer.
  Let $m^2>0$ and let
  $C=(-\Delta+m^2)^{-1}$ on $\Lambda$.
  There exist positive-definite $\Lambda \times \Lambda$ matrices
  $C_1, \dots, C_N$ such that:
  \begin{enumerate}
  \item $C= \sum_{j=1}^N C_j$,
  \item $C_j(x,y) = 0$ if $|x-y| \geq \frac{1}{2} L^j$,
  \item  for multi-indices $\alpha,\beta$ with
  $\ell^1$ norms $|\alpha|_1,|\beta|_1$ at most
  some fixed value $p$, and for $j <N$,
      \begin{equation}
        |\nabla_x^\alpha \nabla_y^\beta C_{j}(x,y)|
        \leq cL^{-(j-1)(2[\phi]-(|\alpha|_1+|\beta|_1))} ,
      \end{equation}
      where $[\phi]=\half(d-2)$,
      and $c$ is independent of $j$ and $N$.
  \end{enumerate}
\end{theorem}

%The infinite volume limit now becomes the limit $N \to \infty$.
The decomposition in Theorem~\ref{thm:Cdecomp}(a)
is called a \emph{finite-range}
decomposition because of item (b):
the covariance $C_j$ has range $\frac 12 L^j$, and fields at points
separated beyond that range are uncorrelated under $\E_{C_j}$.

To compute the important expectation $\E_C e^{-V_0(\Lambda)}$
in \refeq{GgoalC},
we use Theorem~\ref{thm:Cdecomp} and \refprop{Gaussian-conv} to evaluate
it progressively.
Namely, if we define
\begin{equation}
  Z_0 = e^{-V_0(\Lambda)}, \quad Z_{j+1} = \E_{C_{j+1}}\theta Z_j \;\;\;
  (j+1<N), \quad Z_N = \E_{C_N} Z_{N-1},
\end{equation}
then the desired expectation is equal to
$Z_N = \E_C e^{-V_0(\Lambda)}$.
Thus we are led to study the recursion $Z_j \mapsto Z_{j+1}$.

In the expectation $Z_{j+1}=\E_{C_{j+1}}\theta Z_j$,
on the right-hand side we write $\varphi_j = \varphi_{j+1}+\zeta_{j+1}$, as in \refeq{thetadef},
and similarly for
$\bar\varphi_{j}$, $d\varphi_{j}$,
$d\bar\varphi_{j}$.  The expectation $\E_{C_{j+1}}\theta$ integrates
out $\zeta_{j+1}$, $\bar\zeta_{j+1}$, $d\zeta_{j+1}$, $d\bar\zeta_{j+1}$ leaving
dependence of $Z_{j+1}$ on $\varphi_{j+1}$, $\bar\varphi_{j+1}$, $d\varphi_{j+1}$,
$d\bar\varphi_{j+1}$.  This process is repeated.
The $\zeta_j$ fields that are integrated out are the fluctuation fields.
%The progressive integration corresponds to a decomposition of the boson
%and fermi fields into \emph{fluctuation fields},
%e.g., $\varphi_x = \sum_{j=1}^N \zeta_{j,x}$, and similarly for $\bar\varphi,
%\psi,\bar\psi$.

It follows from Remark~\ref{rem:gi} and
Lemma~\ref{lem:Wick} that
$\E_{C_{j+1}}|\zeta_{j,x}|^2 = C_{j+1}(x,x)$.  With
Theorem~\ref{thm:Cdecomp}(c), this indicates that the typical
size of the fluctuation field $\zeta_j$ is of order $L^{-j[\phi]}$;
the number $[\phi]=\frac 12 (d-2)$ is referred to as the
\emph{scaling dimension} or \emph{engineering dimension}
of the field.
Moreover, Theorem~\ref{thm:Cdecomp}(c)
also indicates that the derivative of $\zeta_{j,x}$
is typically smaller than the field itself by a factor
$L^{-j}$, so that the fluctuation field remains approximately constant
over a distance $L^{j}$.
%Thus the fluctuation fields become smaller and
%smoother as $j$ increases.

To make systematic use of this behaviour of the fields, we introduce
nested pavings of $\Lambda$ by sets of \emph{blocks}
$\mathcal{B}_j$ on scales $j=0,\ldots,N$.
The blocks in $\mathcal{B}_0$ are simply the points in $\Lambda$.
The blocks in $\mathcal{B}_1$ form a disjoint paving of $\Lambda$ by boxes of
side $L$.  More generally, each block in $\mathcal{B}_j$ has side $L^j$
and consists of $L^d$ disjoint blocks in $\mathcal{B}_{j-1}$.
A \emph{polymer} on scale $j$ is any union of blocks in $\mathcal{B}_j$,
and we denote the set of scale-$j$ polymers by $\mathcal{P}_j$.
(This terminology is standard but these polymers have nothing to
do with physical polymers or random walks, they merely provide
a means of organising subsets in the pavings of the torus.)

\begin{figure}[h]
  \input{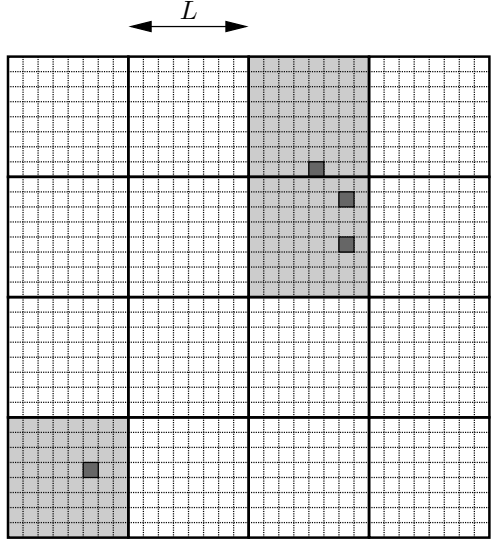}
  \caption{The four small shaded squares represent a polymer in $\mathcal{P}_0$,
  and the three larger shaded squares represent its closure in $\mathcal{P}_1$.}
  \label{fig:polymers}
\end{figure}

For a block $B \in \mathcal{B}_j$,
the above considerations concerning the typical size of the fluctuation
field suggest that, at each of the $L^{dj}$ points $x\in B$,
$\zeta_{j,x}$ has typical size $L^{-j[\phi]}$, and hence
\begin{equation}
  \sum_{x\in B} \zeta_{j,x}^p \approx L^{dj} L^{-pj[\phi]}
  = L^{(d-p[\phi])j}.
\end{equation}
The above sum is \emph{relevant} (growing exponentially in $j$)
for $p[\phi]<d$, \emph{irrelevant} (decaying exponentially in $j$)
for $p[\phi]>d$,
and \emph{marginal} (neither growing or decaying) for $p[\phi]=d$.
Since $\tau_x = \varphi_x\bar\varphi_x +\psi_x\bar\psi_x$ is
quadratic in the fields, it corresponds to $p=2$.
Thus $p[\phi]=2[\phi] = d-2 < d$ and
$\tau_x$
is relevant in all dimensions. Similarly, $\tau_x^2$
corresponds to $p=4$ with
$p[\phi]=4[\phi]=2d-4$, so that $\tau_x^2$ is irrelevant for $d>4$, marginal
for $d=4$, and relevant for $d<4$.
%This makes a renormalisation group
%analysis more
%difficult in dimensions below four.
The monomial $\tau_{\Delta,x}$ is marginal in all dimensions.
%We will now focus on the most interesting case of $d=4$, where $\tau_x^2$ is marginal.
%For $d=4$, the above sum is \emph{relevant} (growing exponentially in $j$)
%for $p<4$, \emph{irrelevant} (decaying exponentially in $j$) for $p>4$,
%and \emph{marginal} (neither growing or decaying) for $p=4$.
%With this terminology, $\tau_x^2$ is marginal ($p=4$) and $\tau_x$ is
%relevant ($p=2$).
In fact, the three monomials $\tau_x^2$, $\tau_x$ and $\tau_{\Delta,x}$,
which constitute the initial potential
$\widetilde V_0$, are precisely the marginal and relevant local monomials
that are Euclidean invariant and obey an additional symmetry between
bosons and fermions
called \emph{supersymmetry} (see \cite{BIS09}).

\subsection{The map \texorpdfstring{$Z_0 \mapsto Z_1$}{Z_0 -> Z_1}}
\label{sec:Z0Z1}

For an idea of how the recursion $Z_j \mapsto Z_{j+1}$ might be
studied, let us %consider the case $j=0$, so the map
take $j=0$ and consider the map
$Z_0 \mapsto Z_1=\E_{C_1}\theta Z_0$.

For simplicity, we set $\sigma=\bar\sigma=0$, so that $V_0=
g_0\tau^2+\nu_0\tau+z_0\tau_\Delta$ is
translation invariant.  As usual, the monomials in $V_0$ depend on the
fields $\varphi, \bar\varphi, \psi, \bar\psi$.
As discussed above, we decompose the field $\varphi$ as
$\varphi = \varphi_1 + \zeta_1 $, and similarly for $\bar\varphi, \psi,\bar\psi$.
The operation $\E_{C_1}\theta$ integrates out the fields $\zeta_1,\bar\zeta_1,
d\zeta_1,d\bar\zeta_1$.
Recall that, by definition, ${\mathcal P}_0$ is the set of subsets of $\Lambda$.
We write $I_0(x) = e^{-V_0(x)}$, and, for $X \in {\mathcal P}_0$, write
$I_0^X = \prod_{x\in X} I_0(x)=e^{-V_0(X)}$ where $V_0(X)
= \sum_{x \in X} V_0(x)$.
In this notation, the dependence on the fields is left implicit.
Let
  \begin{equation}
    V_1  =
    g_1\tau^2+\nu_1 \tau + z_1 \tau_{\Delta}
  \end{equation}
denote a modification of $V_0$ in which the coupling constants in $V_0$
have been
adjusted, or \emph{renormalised}, to some new values
$g_1,\nu_1 , z_1$. This is the origin
of the term ``renormalisation'' in the renormalisation group.
We set $I_1^X = e^{-V_1(X)}$, but with the
fields in $V_1$ given by $\varphi_1,\bar\varphi_1,d\varphi_1,d\bar\varphi_1$.
Let $\delta I_1^X =\prod_{x\in X}(I_1(x)-\theta I_0(x))$; this
is an element of $\mathcal{N}^\times$ since $I_1$ depends on the
fields $\varphi_1$ and so on, while
$\theta I_0$ depends on $\varphi_1+\zeta_1$
and so on.

%Now we drop the operator $\theta$ from the notation,
%and equivalently regard $\E_{C_1}$ as the expectation that integrates out
%$\zeta_1,\bar\zeta_1,d\zeta_1,d\bar\zeta_1$ while treating
%$\varphi_1,\bar\varphi_1,d\varphi_1,d\bar\varphi_1$
%as constant.
%
%
Then we obtain
\begin{align}
  Z_1(\Lambda) &= \E_{C_1} \theta I_0(\Lambda)
  = E_{C_1} \prod_{x\in \Lambda} (I_1(x)+\delta I_1(x))
  \nonumber\\
  &= \E_{C_1} \sum_{X\in {\mathcal P}_0} I_1^{\Lambda\setminus X}
  \delta I_1^{X}
  = \sum_{X\in {\mathcal P}_0} I_1^{\Lambda\setminus X} \E_{C_1}
  \delta I_1^{X}
  .
%  \nonumber\\
%  &= \sum_{X\subset \Lambda} I_1^{\Lambda\setminus X}
%  \E_{C_1} \prod_{x\in X}(I_1(x)-\theta I_o(x)) \nonumber\\
%  &
%
%  = \sum_{U \in \mathcal{P}_1} I_1^{\Lambda\setminus U} K_1(U),
\lbeq{Z1}
\end{align}
%
%This expresses
Here we have expressed
$Z_1$ as a sum over a polymer on scale 0; we wish to express it
as a sum over a polymer on scale 1.  To this end, for a polymer $X$
on scale 0, we define the \emph{closure} $\overline{X}$ to be the
smallest polymer on scale 1 containing $X$: see Figure~\ref{fig:polymers}.
We can now write
\begin{equation}\lbeq{Z1Final}
Z_1(\Lambda) = \sum_{U \in \mathcal{P}_1} I_1^{\Lambda\setminus U} K_1(U),
\end{equation}
where
\begin{equation}\lbeq{K1Definition}
K_1(U) = \sum_{X\in {\mathcal P}_0 : \overline{X}=U}
I_1^{U \setminus X} \E_{C_1}  \delta I_1^X.
\end{equation}

\begin{definition}
  For $j=0,1,2, \dots, N$, and for
  $F,G:\mathcal{P}_j \to \mathcal{N}_{\rm even}$,
  where $\mathcal{N}_{\rm even}$ denotes the
  forms of even degree, the \emph{circle product} of $F,G$ is
  \begin{equation}
    (F \circ G)(\Lambda) =
    \sum_{U \in \mathcal{P}_j(\Lambda)} F(\Lambda \setminus U) G(U).
  \end{equation}
  Note that the circle product depends on the scale $j$.
\end{definition}

The circle product is associative and commutative (the latter due to
the restriction to forms of even degree).
With the circle product, we can
encode the formula \refeq{Z1Final} compactly
as $Z_1(\Lambda) = (I_1 \circ K_1)(\Lambda)$, with the convention that $I_1(U)=I_1^U$.
The identity element for the circle product is $\indicator{U = \varnothing}$.
Thus, if we define $K_0(X) = \indicator{X=\varnothing}$, then
$Z_0(\Lambda)=I_0(\Lambda) = (I_0\circ K_0)(\Lambda)$.

All later stages of the recursion proceed inductively
from $Z_j = (I_j \circ K_j) (\Lambda)$.
The interaction $I_j$ continues to be defined by a potential $V_j$,
but the form of the dependence will not, in general, be as simple
as $I=e^{-V}$.  The interaction does, however, obey
$I_j(X) = \prod_{B\in\mathcal{B}_j(X)} I_j(B)$, for
all $X \in {\mathcal P}_j$ and for all $j$.
The following factorisation property of
$K_1$, which can be verified from \refeq{K1Definition},
allows the induction to
proceed.  If $U\in\mathcal{P}_1$ has connected components
$U_1, \dots, U_k$, then $K_1(U) = \prod_{i=1}^k K_1(U_i)$;
the notion of connectivity here includes blocks touching at a corner.
%Also, due to the finite range property of $C_1$, the factors $K_1(U_i)$
%are independent in the sense that
%%
%\begin{equation}
%  \E_{C_2} K_1(U) = \prod_{i=1}^k \E_{C_2}  K_1(U_i).
%\end{equation}
%%
The induction will preserve this key
property for $K_j$ and ${\mathcal P}_j$,
for all $j$.

%where, for $U \in \mathcal{P}_1$, we define
%\begin{equation}
%  K_1(U) = \sum_{X \in {\overline{\mathcal{P}}}_0(U)}
%  I_1^{\Lambda \setminus X} \E_{C_1} \theta \delta I_1^X,
%\end{equation}
%with ${\overline{\mathcal{P}}}_0(U)$ defined to be the subsets of $\Lambda$
%contained in $U$ but not in any smaller polymer in $\mathcal{P}_1$.
%We refer to those $X \in \mathcal{P}_0$ in ${\overline{\mathcal{P}}}_0(U)$
%as having \emph{closure} $U$; see Figure~\ref{fig:polymers}.

%It can be verified that if $U\in\mathcal{P}_1$ has connected components
%$U_1, \dots, U_k$, then $K_1(U) = \prod_{i=1}^k K_1(U_i)$.
%Also, due to the finite range property of $C_1$, the factors $K_1(U_i)$
%are independent in the sense that
%\begin{equation}
%  \E_{C_2} K_1(U) = \prod_{i=1}^n \E_{C_2}  K_1(U_1).
%\end{equation}

%The final member of \refeq{Z1} is an instance of the associative
%and commutative product defined
%in the following definition.

%\begin{definition}
%  For $j=0,1,2, \dots, N$, and for even-degree forms
%  $F,G:\mathcal{P}_j \to \mathcal{N}$, their \emph{circle product} is
%  \begin{equation}
%    (F \circ G)(\Lambda) =
%    \sum_{U \in \mathcal{P}_j(\Lambda)} F(\Lambda \setminus U) G(U).
%  \end{equation}
%  Note that the circle product depends on the scale $j$.
%\end{definition}

%Then we can rewrite \refeq{Z1} as $Z_1(\Lambda) = (I_1 \circ K_1)(\Lambda)$.
%The identity element for the circle product is $\indicator{U = \varnothing}$,
%and thus if we define $K_0(X) = \indicator{X=\varnothing}$ then we
%also have $Z_0(\Lambda)=I_0(\Lambda) = (I_0\circ K_0)(\Lambda)$.

\subsection{Remaining steps in the proof}
\label{sec:remsteps}

Our goal is to prove Theorem~\ref{thm:BS-wsaw}.  According to
\refeq{Ggoal2}, we need to show that there is a choice of $z_0$
such that, for $g$ small and positive,
\begin{equation}
\lbeq{ZNx}
  G_{\nu_c}(x)=\lim_{m^2 \searrow 0}
  \lim_{N \to \infty}
  (1+z_0)
  \frac{\partial^2}{\partial\sigma\partial\bar\sigma}
  \Big|_{\sigma=\bar\sigma=0}
  Z_N(\Lambda)
  \sim c_g |x|^{-(d-2)}.
\end{equation}
%where we restore the dependence on the external field in $Z_N(\Lambda)
%=\E_C e^{-V_0(\Lambda)}$.
In particular, we see from this that
the correct choice of $z_0$ will appear in the value of the constant $c_g$.
The remaining steps in the proof of \refeq{ZNx} are summarised,
imprecisely, as follows.
Much is left unsaid here, and details can be found in \cite{BS11a}.

\begin{theorem}
Let $d \ge 4$, and let $g> 0$ be sufficiently small.
  There is a choice of
  $V_1,\ldots,V_N$ given, for $X \subset \Lambda$, by
  \begin{equation}
    V_j(X) = \sum_{v\in X}
    (g_j\tau_v^2+\nu_j \tau_v + z_j \tau_{\Delta,v})
    + \lambda_j(\sigma\bar\varphi_0 + \bar\sigma \varphi_x)
    + q_j^2  \sigma\bar\sigma,
  \end{equation}
  with $V_j$ determining $I_j$,
  and a choice
  of $K_1,\ldots,K_N$ with $K_j : \mathcal{P}_j \to \mathcal{N}$
  obeying the key factorisation property mentioned above, such that
  \begin{equation}
    Z_j(\Lambda) = (I_j \circ K_j)(\Lambda)
   \end{equation}
   obeys the recursion $Z_{j+1}=\E_{C_{j+1}}\theta Z_j$.
   Moreover, $(V_j,K_j)_{0 \le j \le N}$ obeys the flow equations
   \begin{align}
     g_{j+1} &= g_j - c g_j^2 + r_{g,j} \\
     \nu_{j+1} &= \nu_j + 2g_j C_{j+1}(0,0) + r_{\nu,j} \\
     z_{j+1} &= z_j + r_{z,j} \\
     K_{j+1} &= r_{K,j}
   \end{align}
   where the $r$ terms represent error
   terms.  Further equations define the evolution of $\lambda_j$ and
   $q_j$.
%   The effect of $K_j$ appears only in the error terms.
\end{theorem}

The previous theorem represents the recursion $Z_j \mapsto Z_{j+1}$ as
a dynamical system.  A fixed-point theorem is used to make the
correct choice of the initial value $z_0$ so that the $r$ terms remain
small on all scales, and so that $(g_j,\nu_j,z_j,K_j)$ flows to
$(0,0,0,0)$.  The latter is referred to as \emph{infrared asymptotic
freedom}, and is the effect anticipated below \refeq{G0:2}.
This final ingredient is summarised in the following
theorem.

\begin{theorem}
\label{thm:fp}
  If $g> 0$ is sufficiently small (independent of $N$ and $m^2$),
%  and with our choice $K_0(X) = \indicator{X=\varnothing}$,
  there exists $z_0$
  such that
  \begin{equation}
    \lim_{m^2 \searrow 0}\lim_{N\to\infty}
    V_N =\lambda_\infty(\sigma\bar\varphi_0 + \bar\sigma\varphi_x)
    + q_\infty \sigma\bar\sigma,
    \end{equation}
  with $\lambda_\infty >0$ and, as $x\to\infty$,
  $q_\infty \sim \lambda_\infty^2(-\Delta_{\Z^d})^{-1}_{0x}$.
  Moreover, in an appropriately defined Banach space,
\begin{equation}
    \lim_{m^2 \searrow 0}\lim_{N\to\infty}  K_N(\Lambda) = 0.
\end{equation}
\end{theorem}

At scale $N$ there are only two polymers, namely the single block
$\Lambda$ and the empty set $\varnothing$.
By definition,
$I_N(\varnothing)=K_N(\varnothing)=1$.
Also, the field has been entirely
integrated out at scale $N$, and from Theorem~\ref{thm:fp} and the definition
of the circle product,  we obtain
\begin{equation}
  Z_N(\Lambda) = I_N(\Lambda) + K_N(\Lambda) \approx I_N(\Lambda)
  \approx e^{-q_N \sigma \bar\sigma}.
\end{equation}
Let $z_0^*=\lim_{m^2 \searrow 0}z_0$.
With \refeq{ZNx} and $q_N \to q_\infty$, this gives
\begin{equation}
 G_{\nu_c}(x) =  (1+z_0^*) q_\infty \sim
 (1+z_0^*)\lambda_\infty^2(-\Delta_{\Z^d})^{-1}_{0x}
 \sim
 (1+z_0^*)\lambda_\infty^2 c_0  |x|^{-(d-2)}.
\end{equation}
This is the desired conclusion of Theorem~\ref{thm:BS-wsaw}.

\subsection{Tutorial}
\lbsect{Tut4}

These problems develop a proof of the discrete-time version
of Theorem~\ref{thm:SL}.
The proof makes use of a Simon--Lieb inequality---this
is now a generic term for inequalities of
  the sort introduced in \cite{Simo80,Lieb80} for the Ising model.
The approach developed here can be adapted to prove Theorem~\ref{thm:SL}.

Let $\Gamma$ represent either $\Gamma = \Z^d$
or the discrete torus $\Gamma = \Z^d/R\Z^d$.
Let $\E_x$ denote the expectation for the usual discrete-time
SRW  on $\Gamma$, which we denote now by $(X_n)_{n\geq 0}$, starting at $x$.
Let $I_{m,n}$ denote the number of self-intersections of
$X$ between times $m$ and $n$:
\begin{equation}
  I_{m,n} = \sum_{m\leq i < j \leq n} \indicator{X_i = X_j}, \quad I_n = I_{0,n}.
\end{equation}
We define
the two-point function of the weakly SAW in the domain $D \subset \Gamma$
by
\begin{equation}
  G_{\nu,D}(x,y)
  =
  \sum_{n\geq 0} \E_x(e^{-gI_n} \indicator{X_n=y,n < T_D}) e^{-\nu n},
  \quad x,y \in \Gamma, \; \nu \in \R,\\
\end{equation}
where $T_D = \inf\{n \geq 0: X_n \notin D \}$ is the exit time of $D$.
We define the boundary $\partial D =
\{x \not\in D: \text{$\exists y \in D$ s.t. } x \sim y\}$, and the closure
$\bar D = D \cup \partial D$.
The two-point function on the entire graph is written as $G_\nu$
rather than $G_{\nu,\Gamma}$.
% The two-point functions of the weakly self-avoiding walk killed when hitting and after hitting the set
% $U \subset V$ are defined by
% \begin{align}
%   G_{\nu,U}(x,y) &= \sum_{n\geq 0} E_x(e^{-gI_n} \indicator{X_n=y,n < H_U}) e^{-\nu n}, \quad x,y \in V, \; \nu \in \R,\\
%   G_{\nu,U+}(x,y) &= \sum_{n\geq 0} E_x(e^{-gI_n} \indicator{X_n=y,n \leq H_U}) e^{-\nu n}, \quad x,y \in V, \; \nu \in \R.
% \end{align}
% Here, $H_U = \inf\{n \geq 0: X_n \in U \}$ is the hitting time of $U$.
Let $c_n(x,y) = E_x(e^{-gI_n} \indicator{X_n=y})$, let
$c_n = \sum_{y \in \Gamma} c_n(0,y)$, and define the susceptibility by
\begin{equation}
  \chi(\nu) = \sum_{y \in \Gamma} G_{\nu}(0,y) = \sum_{n\geq 0} c_n e^{-\nu n}.
\end{equation}

\begin{problem} \label{problem:wsaw-submultiplicativity}
  Verify that $(c_n)_{n\geq 0}$ is a submultiplicative sequence,
  i.e.\ $c_{n+m} \leq c_{n}c_{m}$,
  and conclude that $\frac{1}{n} \log(c_n)$ converges to its infimum,
  which is $\nu_c$ %$e^{\nu_c}$
  by definition.
  In particular, notice that for $\nu < \nu_c$, $\chi(\nu) = \infty$
  and for $\nu > \nu_c$, $\chi(\nu) < \infty$.
\end{problem}

\begin{problem} \label{problem:kappa-torus}
  Let $\chi^R(\nu)$ be the susceptibility for $\Z^d/R'\Z^d$
  where $R'=2R+1$, and let $\chi(\nu)$ be the susceptibility for $\Z^d$.
  Prove that $\chi^R(\nu) \leq \chi(\nu)$ for $R'\geq 3$, and, in particular,
  that $\nu_c(\Z^d) \geq \nu_c(\Z^d/R\Z^d)$.
  Here, $\nu_c(\Gamma)$ denotes the critical point of the
  weakly SAW on $\Gamma$.
\end{problem}

\begin{problem} \label{problem:simon-lieb}
  Prove the following version of the Simon-Lieb inequality
  for the discrete-time weakly SAW on $\Gamma$.
%    have proved an analogous inequality for the
%    ferromagnetic Ising model on $\Z^d$. If $(\sigma_x)_{x\in\Z^d}$ is the spin field of the Ising model,
%    $\langle \,\cdot\, \rangle$ its expectation, $\langle \, \cdot \, \rangle_{D}$ the expectation of the spin system
%    inside $D \subset \Z^d$, $x,y \in \Z^d$, $A\subset \Z^d$ is a set that separates $x$ and $y$,
%    and $A^x \subset \Z^d$ is the connected component of $\Z^d \setminus A$ that contains $x$ (but not $y$),
%    then
%    \begin{equation*}
%      \langle \sigma_x \sigma_y \rangle
%      \leq \sum_{z \in A^x} \langle \sigma_x \sigma_z \rangle_{A^x}  \cdot \langle \sigma_z \sigma_y \rangle.
%    \end{equation*}
%    % where $\langle \, \cdot \, \rangle_{D}$ is the expectation of the spin system on $D \subset \Z^d$ and $A^x \subset \Z^d$
%    % is the connected component of $\Z^d \setminus A$ that contains $x$.
%    % For the strictly self-avoiding walk, the inequality is immediate.
%    Also, for the continuous-time weakly self-avoiding walk, a similar inequality holds.
%    Its statement is slightly different, but the essential idea is the same. We restrict to the case of discrete time
%    to keep it simple.}
  Given $D \subset \Gamma$, show that
  \begin{equation} \label{eq:sl}
    G_{\nu}(x,y) - G_{\nu,D}(x,y) \leq \sum_{z \in \partial D} G_{\nu,\bar D}(x,z) G_{\nu}(z,y).
  \end{equation}
  Note that if $x \in D$ and $y \in D^c$, then $G_{\nu,D}(x,y) = 0$.

  % Let $U \subset V$ be a set of vertices and $G_{\nu,U}$, $G_{\nu,U+}$ be the two-point functions
  % of the walk when and killed after hitting $U$. Then
  % \begin{equation} \label{eq:sl}
  %   G_{\nu}(x,y) - G_{\nu,U}(x,y) \leq \sum_{z \in U} G_{\nu,U+}(x,z) G_{\nu}(z,y).
  % \end{equation}
  % Note that if $U$ separates $x$ and $y$, then $G_{\nu,U}(x,y) = 0$.
\end{problem}

The following problem provides an approach to proving exponential
decay of a subcritical two-point function which, unlike
\refprop{TwoPointSubcritDecay}, adapts well to the
continuous-time setting.

\begin{problem} \label{problem:exp-decay}
  %Conclude that $G_\nu$ decays exponentially for $\nu > \nu_c$:
  Let $\Lambda_{R} = \{-R+1, \dots, R\}^d \subset \Z^d$.
%  with
%  $\partial\Lambda_{L} = \{ x \in \Z^d: \abs{x}_\infty = R + 1\}$.
  For $\nu > \nu_c$, $\sum_{y\in\Z^d} G_\nu(0,y)$ is finite, and thus
  $\theta = \sum_{y \in \partial\Lambda_R} G_\nu(0,y) < 1$ for $R$
  sufficiently large.
  Conclude from Problem~\ref{problem:simon-lieb} with $D=\Lambda_R$
  %$U_L=\partial\Lambda_{L}$
  that for $y \not\in\Lambda_R$,
  \begin{equation}
    G_\nu(0,y) \leq \theta^{\lfloor |y|_\infty/(R+1)\rfloor} \sup_{x\in\Z^d} G_\nu(0,x).
  \end{equation}
\end{problem}

\begin{problem} \label{problem:infinite-volume-limit}
  Let $(T_R)_{R\in \N}$ be a sequence of discrete tori
  with the vertex sets $V_R$ embedded in $\Z^d$
  by $V_R = \Lambda_R$ where $\Lambda_R$ is as in
  Problem~\ref{problem:exp-decay}; in particular, $V_R \subset V_{R+1}$.
  Let $G^R_\nu$ be the two-point function on $T_R$,
  and $G_\nu$ be the two-point function on $\Z^d$.
  Use Problem~\ref{problem:kappa-torus} and
  Problem~\ref{problem:exp-decay} to prove that for all
  $\nu > \nu_c = \nu_c(\Z^d)$, $x,y \in \Z^d$,
  \begin{equation}
    \label{eq:TorusToInfty}
    G^R_\nu(x,y) \to G_\nu(x,y) \quad\text{as $R\to\infty$.}
  \end{equation}
  Conclude that
  \begin{equation}
    \label{eq:kappaToCrit}
    G_{\nu_c}(x,y) = \lim_{\nu \searrow \nu_c} \lim_{R\to\infty} G^R_\nu(x,y).
  \end{equation}
\end{problem}

\appendix

\section{Solutions to the problems}

\subsection{Solutions for Tutorial~\ref{s:Tut1}}

\begin{solution}{Problem \ref{problem:subadditivity}}
  Let $M$ be an integer, and for every $n \in \N$, write $n = Mk+r$ with $0\leq r <M$. Then,
  \begin{equation}
    \frac{1}{n}a_n \leq \frac{k}{n} a_M + \frac{1}{n}a_r,
    \quad\text{and, thus,}\quad
    \limsup_{n\to\infty} \frac{1}{n} a_n \leq \frac{1}{M} a_M.
  \end{equation}
  In particular,
  \begin{equation}
    \limsup_{n\to\infty} \frac{1}{n} a_n \leq \inf_{M\in \N} \frac{1}{M} a_M \leq \liminf_{M\to \infty} \frac{1}{n} a_n,
  \end{equation}
  which implies both statements of the claim.
\end{solution}

\begin{solution}{Problem \ref{problem:connectivec}}
  The number of $n$-step walks with steps only in positive coordinate directions is $d^{n}$.
  The number of walks which do not reverse direction is $2d(2d-1)^{n-1}$. Thus,
  \begin{equation}
    d^n \leq c_n \leq 2d(2d-1)^{n-1}
    \quad \text{and therefore}\quad
    d \leq \mu \leq 2d-1.
  \end{equation}
  The upper bound can easily be improved by excluding more patterns that lead to self-intersecting walks
  than merely reversals of steps. For example, by considering walks which do not contain anti-clockwise
  ``unit squares'' (see \reffig{WalkWithHalfLoop}), we obtain
  \begin{equation}
    c_{3n+1} \leq 2d((2d-1)^3-1)^n = 4(26^{1/3})^{3n},
  \end{equation}
  giving $\mu \leq 26^{1/3} < 3$. Similarly, the lower bound can be improved by considering
  walks that take steps either in positive coordinate directions, i.e., north or east,
  or in an east-north-west-north pattern: see \reffig{WalkWithHalfLoop}. It follows that
  \begin{equation}
    c_{4n} \geq (d^4+1)^n = (17^{1/4})^{4n},
  \end{equation}
  where $17^{1/4} > 2$. In particular, $2< 17^{1/4} \leq \mu \leq 26^{1/3} < 3$.
\end{solution}
\begin{figure}[h]
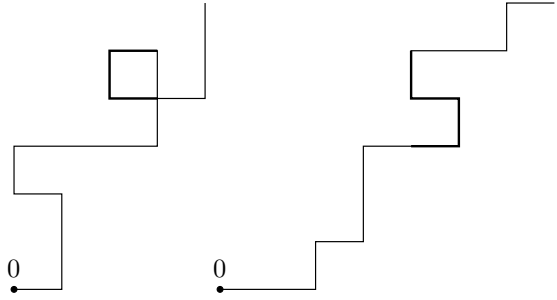

  \input{walk-square.pspdftex}
  \input{walk-with-half-loop.pspdftex}
  \caption{Left: The walk does contain a unit square. Right: The walk only takes steps east, north, or in east-north-west-north patterns  (thick line).}
  \lbfig{WalkWithHalfLoop}
\end{figure}

\begin{solution}{Problem \ref{problem:consistent}}
  SAWs can get trapped: see \reffig{TrappedWalk}. A trapped walk $\omega$ of length $n$
  does not arise as the restriction of a walk $\rho$ of length $m > n$ to the first
  $n$ steps. Thus, under $\Q_n^{(1)}$, $\omega$ has positive probability, while
  $\sum_{\rho > \omega} \Q_{m}^{(1)}(\rho) = 0$.
\end{solution}
\begin{figure}[h]
  \input{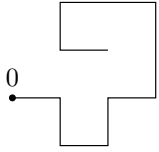}
  \caption{Trapped walk.}
  \lbfig{TrappedWalk}
\end{figure}

\begin{solution}{Problem \ref{problem:2point-1d}}
  $c_n(x) = \indicator{|x|=n}$, so $G_z(x) = \sum_{n\geq 0} c_n(x) z^n = z^{|x|}$, and
  \begin{align}
    \hat G_z(k)
    &=
    \sum_{x\in\Z} z^{|x|} e^{ikx}
    = -1 + \sum_{n\geq 0} z^n (e^{ikn}+e^{-ikn})
    \notag\\
    &=
    -1 + (1-ze^{ik})^{-1} + (1-ze^{-ik})^{-1}
    = \frac{1- z^2}{1-2z\cos k+z^2},
  \end{align}
  as claimed.
\end{solution}

\begin{solution}{Problem \ref{problem:tauberian}}
  The assumption implies
  \begin{equation}
    |f((1-1/n)e^{i\varphi})| \leq c|1-(1-1/n)e^{i\varphi}|^{-b}.
  \end{equation}
  Note that for $\varphi \in [0,\pi/2]$,
  \begin{equation}
    |{\rm Re}(1-(1-1/n)e^{i\varphi})|
    = 1-(1-1/n)\cos\varphi
    % = 1/n+(1-\cos\varphi)
    \geq 1/n,
  \end{equation}
  \begin{equation}
    |{\rm Im}(1-(1-1/n)e^{i\varphi})| = |(1-1/n)\sin\varphi|\\
    \geq (1-1/n) %\begin{cases}
      \frac{2\varphi}{\pi} . % \varphi \in [0, \pi/2],\\
%      2(\pi-\varphi)/\pi & \varphi \in [\pi/2,\pi].
%    \end{cases}
  \end{equation}
  Suppose $b>1$.
  The integral is estimated using $|z|^{n} \geq C e^{-1}$ for $|z| = 1-1/n$,
  % \begin{multline}
  %   \frac{1}{2\pi} \int_0^{\pi/2} |f((1-1/n)e^{i\varphi})| \; d\varphi
  %   \leq \frac{c}{2\pi} \int_0^{\pi/2} \left(\frac{1}{n} + \frac{2\varphi}{\pi}\right)^{-b} \; d\varphi
  %   = \frac{c}{4} \int_{1/n}^{1/n +1} \theta^{-b} \; d\theta \\
  %   = \frac{cb}{4} (n^{b-1} - (\frac{1}{n}+1)^{-b+1} ) \leq \frac{cb}{4} n^{b-1},
  % \end{multline}
  % and, similarly,
  % \begin{equation}
  %   \frac{1}{2\pi} \int_{\pi/2}^\pi |f((1-1/n)e^{i\varphi})| \; d\varphi
  %   \leq \frac{c}{2\pi} \int_{\pi/2}^\pi \left(\frac{1}{n} + \frac{2(\pi-\varphi)}{\pi}\right)^{-b} \; d\varphi
  %   %= \frac{c}{4} \int_{1/n}^{1+1/n} \theta^{-b} \; d\theta
  %   \leq \frac{cb}{4} n^{b-1}.
  % \end{equation}
  \begin{multline}
    \frac{1}{2\pi} \int_0^{\pi/2} |f((1-1/n)e^{i\varphi})| \; d\varphi
    \leq c \int_0^{\pi/2} \left(\frac{1}{n} + (1-1/n)\frac{2\varphi}{\pi}\right)^{-b} \; d\varphi\\
    = (1-1/n)^{-1} c \int_{1/n}^{1} t^{-b} \; dt
    = (1-1/n)^{-1} cb  (n^{b-1} - 1) \leq c n^{b-1},
  \end{multline}
  and, since $|f(z)|$ is bounded for $z$ bounded away from $1$,
  \begin{equation}
    \frac{1}{2\pi} \int_{\pi/2}^\pi |f((1-1/n)e^{i\varphi})| \; d\varphi
    %\leq \frac{cb}{4} n^{b-1}.
    \leq c .
  \end{equation}
  Likewise, the contributions for the interval $[\pi,2\pi]$ are estimated and we obtain
  \begin{equation}
    |a_n| \leq cn^{b-1}.
  \end{equation}
  The above assumed $b >1$ but the extension to $b=1$ is easy.
\end{solution}

\begin{solution}{Problem \ref{problem:transience}}
  (a) Let $T_0 = 0$ and $T_k = \inf \{ n > T_{k-1}: X_n = 0 \}$. Then $u = P(T_1 < \infty)$, and by induction
  and the strong Markov property,
  $P(T_k < \infty) = u^k$.
  %\begin{equation*}
  %  P(T_k < \infty) = P(T_k < \infty | T_{k-1} < \infty) P(T_{k-1} < \infty) = \cdots = u^k
  %\end{equation*}
  It follows that
  \begin{equation}
    m = \E(N) = \sum_{k \geq 0} \P(T_k< \infty) = (1-u)^{-1} .
  \end{equation}
  \medskip
  (b) The solution relies on the formula %$|\hat D(k)| \leq 1$ and
  \begin{equation}
    \P(X_n = 0) = \int_{[-\pi,\pi]^d} \hat D(k)^n \frac{d^dk}{(2\pi)^d}.
  \end{equation}
  Some care is required when performing the sum over $n$ since the best uniform bound on $\hat D^n$ is $1$
  which is not summable. A solution is to make use of monotone convergence first, and then apply the dominated
  convergence theorem, as follows,
  \begin{equation}
    m = \lim_{t\nearrow 1} \sum_{n \geq 0} \P(X_n=0)t^n = \lim_{t\nearrow 1} \int_{[-\pi,\pi]^d} \frac{1}{1- t\hat D(k)}\; \frac{d^dk}{(2\pi)^d}.
  \end{equation}
  Note that $\hat D$ is a real-valued function and that
  \begin{equation}
    \frac{1}{1- t\hat D(k)}
    \leq \frac{2}{1- \hat D(k)} \quad \text{for $t\in[1/2,1]$},
  \end{equation}
  so that if $(1-\hat D)^{-1} \in L^1$, %([-\pi,\pi]^d)$,
  then the claim follows by dominated convergence.
  In the case that
  $(1-\hat D)^{-1} \not\in L^1$, %([-\pi,\pi]^d)$,
  the claim follows from Fatou's lemma.

  \medskip

  (c) $\hat D(k) = \sum_{j=1}^d (e^{ik_j} + e^{-ik_j}) = 2 \sum_{j=1}^d \cos(k_j)$ and thus
  $1- \hat D(k) = O(1) |k|^2$ as $k \to 0$. Note further that
  \begin{equation}
    \int_{\R^d} f(|k|) \; dk = V_{d-1} \int_0^\infty f(r) \; r^{d-1} \; dr,
  \end{equation}
  where $V_{d-1}$ is the volume of the $(d-1)$-dimensional sphere, and in particular,
  \begin{equation}
    \int_{[-\epsilon,\epsilon]^d} |k|^{-p} \; dk \text{ is integrable} \quad  \text{if and only if}\quad d > p. \qedhere
  \end{equation}
\end{solution}

\begin{solution}{Problem \ref{problem:intersection}}
  Note that
  \begin{equation}
    I = \sum_{x \in \Z^d} \Big(\sum_{i\geq 0}
    \indicator{X^1_i = x}\Big)\Big( \sum_{j\geq 0} \indicator{X^2_j = x}\Big),
  \end{equation}
  and thus, by Parseval's theorem, if $f \in L^2(\Z^d)$,
  \begin{equation}
    \E(I) = \sum_{x \in \Z^d} f(x)^2 = \int_{[-\pi,\pi]^d} |\hat f(k)|^2 \frac{d^dk}{(2\pi)^d},
  \end{equation}
  where
  \begin{equation}
    f(x) =  \sum_{j\geq 0} \P\{X^1_j = x\} = \sum_{j \geq 0} D^{*j}(x),
    \quad
    \hat f(k)
    = \sum_{j \geq 0}  \hat D(k)^j
    = \frac{1}{1-\hat D(k)}.
  \end{equation}
  If $f \not\in L^2(\Z^d)$, then both sides must be infinite.
\end{solution}

\subsection{Solutions for Tutorial~\ref{s:Tut2}}

\begin{solution}{Problem \ref{problem:graph-connected-not-path-connected}}
The graph $\set{0n}$ is connected on $[0,n]$ in the above sense but not
path-connected.  Also, $\set{01,12,\dotsc,(n-1)n}$ is path-connected but
not connected in the above sense since the open intervals $(i-1,i)$ do
not overlap.
\end{solution}

\begin{solution}{Problem \ref{problem:Ksumprod}}
This is an application of the identity
\begin{equation}
\label{e:SubsetProduct}
\prod_{i\in I} (1+u_i)=\sum_{S\subset I}\prod_{i\in S} u_i
\end{equation}
with $I$ being the set of edges on $[a,b]$.
\end{solution}

\begin{solution}{Problem \ref{problem:KJidentity}}
The identity corresponds to a decomposition of $\mathcal{B}[a,b]$ by
connected components.  The term $K[a+1,b]$ corresponds to graphs $\Gamma$
for which $a\notin\Gamma$.

So assume $a\in\Gamma$.  We shall show that $\Gamma$ can be written
uniquely as $\Gamma=\Gamma'\union\Gamma''$ where $\Gamma'\in\mathcal{G}[a,j]$
and $\Gamma''\in\mathcal{B}[j,b]$ for some $j\in(a,b]$.  Informally, $\Gamma'$ is
the connected component of $\Gamma$ containing $a$, though we must verify
that this notion is well-defined. Conversely it is
clear that if $\Gamma'\in \mathcal{G}[a,j]$, $\Gamma''\in \mathcal{B}[j,b]$
for some $j\in(a,b]$, then $\Gamma=\Gamma'\union\Gamma''\in\mathcal{B}[a,b]$
with $a\in\Gamma$.  Then the result will follow since
\begin{equation}
%\label{}
\prod_{st\in\Gamma'\union\Gamma''} U_{st}=\prod_{st\in\Gamma'} U_{st}\prod_{st\in\Gamma''} U_{st}.
\end{equation}

Let
\begin{equation}
\label{e:jFormula}
j=\min\set{i\in(a,b]: i\notin(s,t)\text{ for some $st\in\Gamma$}}.
\end{equation}
The minimum is well defined since there can be no $st \in \Gamma$ for
which $b\in(s,t)$.
By construction, every edge $st\in\Gamma$ satisfies $t\leq j$ or $s\geq j$, so that we can write $\Gamma=\Gamma'\union\Gamma''$ where $\Gamma'\in\mathcal{B}[a,j]$, $\Gamma''\in\mathcal{B}[j,b]$.  We must show that $\Gamma'\in\mathcal{G}[a,j]$, i.e., that $\Gamma'$ is connected.  But $\union_{st\in\Gamma'}(s,t)=(a,j)\intersect\union_{st\in\Gamma}(s,t)=(a,j)$ by the minimality of $j$.

Finally we check that the decomposition $\Gamma=\Gamma'\union\Gamma''$ is unique: this follows because if $\Gamma'\in\mathcal{G}[a,j']$ and $\Gamma''\in\mathcal{B}[j',b]$ then the formula \eqref{e:jFormula} recovers $j=j'$.
\end{solution}

\begin{solution}{Problem \ref{problem:c_n-convolution}}
The convolutions correspond to summing over the values of
$\omega(1)$ and $\omega(m)$.  Namely, noting that a walk $\omega$
on $[0,n]$ is equivalent to a pair of walks $\omega_0$ on $[0,m]$
and $\omega_1$ on $[m,n]$ with $\omega_0(m)=\omega_1(m)$, we have
\begin{align}
c_n(x)
&=
\sum_{\omega\in\mathcal{W}_n(0,x)} K[1,n](\omega)
+\sum_{m=1}^n \sum_{\omega\in\mathcal{W}_n(0,x)} J[0,m](\omega) K[m,n](\omega)
\notag\\
&=
\sum_{y\in\Z^d}\sum_{\omega_0\in\mathcal{W}_1(0,y)} \sum_{\substack{\omega_1 : [1,n]\to\Z^d,\\ \omega_1(1)=y, \, \omega_1(n)=x}} K[1,n](\omega_1)
\notag\\
&\qquad
+\sum_{m=1}^n \sum_{y\in\Z^d} \sum_{\omega_0\in\mathcal{W}_m(0,y)} \sum_{\substack{\omega_1:[m,n]\to\Z^d,\\ \omega_1(m)=y, \, \omega_1(n)=x}} J[0,m](\omega_0) K[m,n](\omega_1)
\notag\\
&=
\sum_{y\in\Z^d} \indicator{y\in \Omega} c_{n-1}(x-y)
+\sum_{m=1}^n \sum_{y\in\Z^d} \pi_m(y) c_{n-m}(x-y)
%\notag
\end{align}
where we use the translation invariance (in time and space) of $K$.
Since $c_1(y)=\indicator{y\in \Omega}$, this is the desired equation.
\end{solution}

\begin{figure}
\begin{center}
\setlength{\unitlength}{0.0115in}%
\begin{picture}(215,270)(20,580)
\thicklines
%\put(300,810){\oval(21,20)[t]}
%\put(300,810){\oval(21,20)[b]}
%\put(280,740){\line( 1, 0){ 40}}
%\put(280,620){\line( 0,-1){ 20}}
%\put(280,620){\line( 2,-1){ 40}}
%\put(320,600){\line( 0, 1){ 20}}
%\put(300,602){\line( 0,-1){ 4}}
%\put(300,622){\line( 0,-1){ 4}}
%\put(280,680){\line( 1, 1){ 20}}
%\put(300,700){\line( 1,-1){ 20}}
%\put(300,682){\line( 0,-1){ 4}}
\put( 20,800){\makebox(0,0)[lb]{\raisebox{0pt}[0pt][0pt]{\circle*{2}}}}
\put( 30,800){\makebox(0,0)[lb]{\raisebox{0pt}[0pt][0pt]{\circle*{2}}}}
\put( 40,800){\makebox(0,0)[lb]{\raisebox{0pt}[0pt][0pt]{\circle*{2}}}}
\put( 50,800){\makebox(0,0)[lb]{\raisebox{0pt}[0pt][0pt]{\circle*{2}}}}
\put( 60,800){\makebox(0,0)[lb]{\raisebox{0pt}[0pt][0pt]{\circle*{2}}}}
\put( 70,800){\makebox(0,0)[lb]{\raisebox{0pt}[0pt][0pt]{\circle*{2}}}}
\put( 80,800){\makebox(0,0)[lb]{\raisebox{0pt}[0pt][0pt]{\circle*{2}}}}
\put( 90,800){\makebox(0,0)[lb]{\raisebox{0pt}[0pt][0pt]{\circle*{2}}}}
\put(100,800){\makebox(0,0)[lb]{\raisebox{0pt}[0pt][0pt]{\circle*{2}}}}
\put(110,800){\makebox(0,0)[lb]{\raisebox{0pt}[0pt][0pt]{\circle*{2}}}}
\put(120,800){\makebox(0,0)[lb]{\raisebox{0pt}[0pt][0pt]{\circle*{2}}}}
\put(130,800){\makebox(0,0)[lb]{\raisebox{0pt}[0pt][0pt]{\circle*{2}}}}
\put(140,800){\makebox(0,0)[lb]{\raisebox{0pt}[0pt][0pt]{\circle*{2}}}}
\put(150,800){\makebox(0,0)[lb]{\raisebox{0pt}[0pt][0pt]{\circle*{2}}}}
\put(160,800){\makebox(0,0)[lb]{\raisebox{0pt}[0pt][0pt]{\circle*{2}}}}
\put(170,800){\makebox(0,0)[lb]{\raisebox{0pt}[0pt][0pt]{\circle*{2}}}}
\put(180,800){\makebox(0,0)[lb]{\raisebox{0pt}[0pt][0pt]{\circle*{2}}}}
\put(190,800){\makebox(0,0)[lb]{\raisebox{0pt}[0pt][0pt]{\circle*{2}}}}
\put(200,800){\makebox(0,0)[lb]{\raisebox{0pt}[0pt][0pt]{\circle*{2}}}}
\put(210,800){\makebox(0,0)[lb]{\raisebox{0pt}[0pt][0pt]{\circle*{2}}}}
\qbezier(20,800)(115,830)(210,800)
%\put(115,800){\oval(190,15)[t]}
\put( 20,780){\makebox(0,0)[lb]{\raisebox{0pt}[0pt][0pt]{$s_1$}}}
\put(210,780){\makebox(0,0)[lb]{\raisebox{0pt}[0pt][0pt]{$t_1$}}}
%\put(300,800){\makebox(0,0)[lb]{\raisebox{0pt}[0pt][0pt]{\circle*{2}}}}
%\put(295,780){\makebox(0,0)[lb]{\raisebox{0pt}[0pt][0pt]{$0,T$}}}
\put( 20,740){\makebox(0,0)[lb]{\raisebox{0pt}[0pt][0pt]{\circle*{2}}}}
\put( 30,740){\makebox(0,0)[lb]{\raisebox{0pt}[0pt][0pt]{\circle*{2}}}}
\put( 40,740){\makebox(0,0)[lb]{\raisebox{0pt}[0pt][0pt]{\circle*{2}}}}
\put( 50,740){\makebox(0,0)[lb]{\raisebox{0pt}[0pt][0pt]{\circle*{2}}}}
\put( 60,740){\makebox(0,0)[lb]{\raisebox{0pt}[0pt][0pt]{\circle*{2}}}}
\put( 70,740){\makebox(0,0)[lb]{\raisebox{0pt}[0pt][0pt]{\circle*{2}}}}
\put( 80,740){\makebox(0,0)[lb]{\raisebox{0pt}[0pt][0pt]{\circle*{2}}}}
\put( 90,740){\makebox(0,0)[lb]{\raisebox{0pt}[0pt][0pt]{\circle*{2}}}}
\put(100,740){\makebox(0,0)[lb]{\raisebox{0pt}[0pt][0pt]{\circle*{2}}}}
\put(110,740){\makebox(0,0)[lb]{\raisebox{0pt}[0pt][0pt]{\circle*{2}}}}
\put(120,740){\makebox(0,0)[lb]{\raisebox{0pt}[0pt][0pt]{\circle*{2}}}}
\put(130,740){\makebox(0,0)[lb]{\raisebox{0pt}[0pt][0pt]{\circle*{2}}}}
\put(140,740){\makebox(0,0)[lb]{\raisebox{0pt}[0pt][0pt]{\circle*{2}}}}
\put(150,740){\makebox(0,0)[lb]{\raisebox{0pt}[0pt][0pt]{\circle*{2}}}}
\put(160,740){\makebox(0,0)[lb]{\raisebox{0pt}[0pt][0pt]{\circle*{2}}}}
\put(170,740){\makebox(0,0)[lb]{\raisebox{0pt}[0pt][0pt]{\circle*{2}}}}
\put(180,740){\makebox(0,0)[lb]{\raisebox{0pt}[0pt][0pt]{\circle*{2}}}}
\put(190,740){\makebox(0,0)[lb]{\raisebox{0pt}[0pt][0pt]{\circle*{2}}}}
\put(200,740){\makebox(0,0)[lb]{\raisebox{0pt}[0pt][0pt]{\circle*{2}}}}
\put(210,740){\makebox(0,0)[lb]{\raisebox{0pt}[0pt][0pt]{\circle*{2}}}}
%\put(280,740){\makebox(0,0)[lb]{\raisebox{0pt}[0pt][0pt]{\circle*{2}}}}
%\put(320,740){\makebox(0,0)[lb]{\raisebox{0pt}[0pt][0pt]{\circle*{2}}}}
%\put(300,740){\oval(40,10)}
\qbezier(20,740)(80,770)(140,740)
\qbezier(80,740)(145,770)(210,740)
%\put(80,740){\oval(120,10)[t]}
%\put(145,740){\oval(130,15)[t]}
%\put(250,740){\makebox(0,0)[lb]{\raisebox{0pt}[0pt][0pt]{$a,t_1$}}}
%\put(335,740){\makebox(0,0)[lb]{\raisebox{0pt}[0pt][0pt]{$s_2,b$}}}
\put( 20,720){\makebox(0,0)[lb]{\raisebox{0pt}[0pt][0pt]{$s_1$}}}
\put( 80,720){\makebox(0,0)[lb]{\raisebox{0pt}[0pt][0pt]{$s_2$}}}
\put(140,720){\makebox(0,0)[lb]{\raisebox{0pt}[0pt][0pt]{$t_1$}}}
\put(210,720){\makebox(0,0)[lb]{\raisebox{0pt}[0pt][0pt]{$t_2$}}}
\put( 20,680){\makebox(0,0)[lb]{\raisebox{0pt}[0pt][0pt]{\circle*{2}}}}
\put( 30,680){\makebox(0,0)[lb]{\raisebox{0pt}[0pt][0pt]{\circle*{2}}}}
\put( 40,680){\makebox(0,0)[lb]{\raisebox{0pt}[0pt][0pt]{\circle*{2}}}}
\put( 50,680){\makebox(0,0)[lb]{\raisebox{0pt}[0pt][0pt]{\circle*{2}}}}
\put( 60,680){\makebox(0,0)[lb]{\raisebox{0pt}[0pt][0pt]{\circle*{2}}}}
\put( 70,680){\makebox(0,0)[lb]{\raisebox{0pt}[0pt][0pt]{\circle*{2}}}}
\put( 80,680){\makebox(0,0)[lb]{\raisebox{0pt}[0pt][0pt]{\circle*{2}}}}
\put( 90,680){\makebox(0,0)[lb]{\raisebox{0pt}[0pt][0pt]{\circle*{2}}}}
\put(100,680){\makebox(0,0)[lb]{\raisebox{0pt}[0pt][0pt]{\circle*{2}}}}
\put(110,680){\makebox(0,0)[lb]{\raisebox{0pt}[0pt][0pt]{\circle*{2}}}}
\put(120,680){\makebox(0,0)[lb]{\raisebox{0pt}[0pt][0pt]{\circle*{2}}}}
\put(130,680){\makebox(0,0)[lb]{\raisebox{0pt}[0pt][0pt]{\circle*{2}}}}
\put(140,680){\makebox(0,0)[lb]{\raisebox{0pt}[0pt][0pt]{\circle*{2}}}}
\put(150,680){\makebox(0,0)[lb]{\raisebox{0pt}[0pt][0pt]{\circle*{2}}}}
\put(160,680){\makebox(0,0)[lb]{\raisebox{0pt}[0pt][0pt]{\circle*{2}}}}
\put(170,680){\makebox(0,0)[lb]{\raisebox{0pt}[0pt][0pt]{\circle*{2}}}}
\put(180,680){\makebox(0,0)[lb]{\raisebox{0pt}[0pt][0pt]{\circle*{2}}}}
\put(190,680){\makebox(0,0)[lb]{\raisebox{0pt}[0pt][0pt]{\circle*{2}}}}
\put(200,680){\makebox(0,0)[lb]{\raisebox{0pt}[0pt][0pt]{\circle*{2}}}}
\put(210,680){\makebox(0,0)[lb]{\raisebox{0pt}[0pt][0pt]{\circle*{2}}}}
\put( 20,660){\makebox(0,0)[lb]{\raisebox{0pt}[0pt][0pt]{$s_1$}}}
\put( 50,660){\makebox(0,0)[lb]{\raisebox{0pt}[0pt][0pt]{$s_2$}}}
\put( 80,660){\makebox(0,0)[lb]{\raisebox{0pt}[0pt][0pt]{$t_1$}}}
\put(120,660){\makebox(0,0)[lb]{\raisebox{0pt}[0pt][0pt]{$s_3$}}}
\put(160,660){\makebox(0,0)[lb]{\raisebox{0pt}[0pt][0pt]{$t_2$}}}
\put(210,660){\makebox(0,0)[lb]{\raisebox{0pt}[0pt][0pt]{$t_3$}}}
\qbezier(20,680)(50,710)(80,680)
\qbezier(50,680)(105,710)(160,680)
\qbezier(120,680)(165,710)(210,680)
%\put(50,680){\oval(60,15)[t]}
%\put(105,680){\oval(110,10)[t]}
%\put(165,680){\oval(90,15)[t]}
%\put(280,620){\makebox(0,0)[lb]{\raisebox{0pt}[0pt][0pt]{\circle*{2}}}}
%\put(280,600){\makebox(0,0)[lb]{\raisebox{0pt}[0pt][0pt]{\circle*{2}}}}
%\put(320,620){\makebox(0,0)[lb]{\raisebox{0pt}[0pt][0pt]{\circle*{2}}}}
%\put(320,600){\makebox(0,0)[lb]{\raisebox{0pt}[0pt][0pt]{\circle*{2}}}}
%\put(300,610){\oval(60,20)}
%\put(265,630){\makebox(0,0)[lb]{\raisebox{0pt}[0pt][0pt]{$s_2,t_2$}}}
%\put(320,630){\makebox(0,0)[lb]{\raisebox{0pt}[0pt][0pt]{$s_4,b$}}}
%\put(270,580){\makebox(0,0)[lb]{\raisebox{0pt}[0pt][0pt]{$a,t_1$}}}
%\put(320,580){\makebox(0,0)[lb]{\raisebox{0pt}[0pt][0pt]{$s_3,t_3$}}}
\put( 20,580){\makebox(0,0)[lb]{\raisebox{0pt}[0pt][0pt]{$s_1$}}}
\put( 50,580){\makebox(0,0)[lb]{\raisebox{0pt}[0pt][0pt]{$s_2$}}}
\put( 80,580){\makebox(0,0)[lb]{\raisebox{0pt}[0pt][0pt]{$t_1$}}}
\put(110,580){\makebox(0,0)[lb]{\raisebox{0pt}[0pt][0pt]{$s_3$}}}
\put(140,580){\makebox(0,0)[lb]{\raisebox{0pt}[0pt][0pt]{$t_2$}}}
\put(170,580){\makebox(0,0)[lb]{\raisebox{0pt}[0pt][0pt]{$s_4$}}}
\put(190,580){\makebox(0,0)[lb]{\raisebox{0pt}[0pt][0pt]{$t_3$}}}
\put(210,580){\makebox(0,0)[lb]{\raisebox{0pt}[0pt][0pt]{$t_4$}}}
\qbezier(20,600)(50,630)(80,600)
\qbezier(50,600)(95,630)(140,600)
\qbezier(110,600)(150,630)(190,600)
\qbezier(170,600)(190,630)(210,600)
%\put(50,600){\oval(60,15)[t]}
%\put(95,600){\oval(90,10)[t]}
%\put(150,600){\oval(80,15)[t]}
%\put(190,600){\oval(40,10)[t]}
\put( 20,600){\makebox(0,0)[lb]{\raisebox{0pt}[0pt][0pt]{\circle*{2}}}}
\put( 30,600){\makebox(0,0)[lb]{\raisebox{0pt}[0pt][0pt]{\circle*{2}}}}
\put( 40,600){\makebox(0,0)[lb]{\raisebox{0pt}[0pt][0pt]{\circle*{2}}}}
\put( 50,600){\makebox(0,0)[lb]{\raisebox{0pt}[0pt][0pt]{\circle*{2}}}}
\put( 60,600){\makebox(0,0)[lb]{\raisebox{0pt}[0pt][0pt]{\circle*{2}}}}
\put( 70,600){\makebox(0,0)[lb]{\raisebox{0pt}[0pt][0pt]{\circle*{2}}}}
\put( 80,600){\makebox(0,0)[lb]{\raisebox{0pt}[0pt][0pt]{\circle*{2}}}}
\put( 90,600){\makebox(0,0)[lb]{\raisebox{0pt}[0pt][0pt]{\circle*{2}}}}
\put(100,600){\makebox(0,0)[lb]{\raisebox{0pt}[0pt][0pt]{\circle*{2}}}}
\put(110,600){\makebox(0,0)[lb]{\raisebox{0pt}[0pt][0pt]{\circle*{2}}}}
\put(120,600){\makebox(0,0)[lb]{\raisebox{0pt}[0pt][0pt]{\circle*{2}}}}
\put(130,600){\makebox(0,0)[lb]{\raisebox{0pt}[0pt][0pt]{\circle*{2}}}}
\put(140,600){\makebox(0,0)[lb]{\raisebox{0pt}[0pt][0pt]{\circle*{2}}}}
\put(150,600){\makebox(0,0)[lb]{\raisebox{0pt}[0pt][0pt]{\circle*{2}}}}
\put(160,600){\makebox(0,0)[lb]{\raisebox{0pt}[0pt][0pt]{\circle*{2}}}}
\put(170,600){\makebox(0,0)[lb]{\raisebox{0pt}[0pt][0pt]{\circle*{2}}}}
\put(180,600){\makebox(0,0)[lb]{\raisebox{0pt}[0pt][0pt]{\circle*{2}}}}
\put(190,600){\makebox(0,0)[lb]{\raisebox{0pt}[0pt][0pt]{\circle*{2}}}}
\put(200,600){\makebox(0,0)[lb]{\raisebox{0pt}[0pt][0pt]{\circle*{2}}}}
\put(210,600){\makebox(0,0)[lb]{\raisebox{0pt}[0pt][0pt]{\circle*{2}}}}
%\put(280,680){\makebox(0,0)[lb]{\raisebox{0pt}[0pt][0pt]{\circle*{2}}}}
%\put(300,700){\makebox(0,0)[lb]{\raisebox{0pt}[0pt][0pt]{\circle*{2}}}}
%\put(320,680){\makebox(0,0)[lb]{\raisebox{0pt}[0pt][0pt]{\circle*{2}}}}
%\put(300,690){\oval(60,20)}
%\put(290,710){\makebox(0,0)[lb]{\raisebox{0pt}[0pt][0pt]{$s_2,t_2$}}}
%\put(270,665){\makebox(0,0)[lb]{\raisebox{0pt}[0pt][0pt]{$a,t_1$}}}
%\put(320,665){\makebox(0,0)[lb]{\raisebox{0pt}[0pt][0pt]{$s_3,b$}}}
%\put(170,530){\oval(100,20)}
%\put(130,540){\line(1,-2){10}}
%\put(150,540){\line(1,-2){10}}
%\put(170,540){\line(1,-2){10}}
%\put(190,540){\line(1,-2){10}}
%\put(130,540){\line(0,-1){20}}
%\put(150,540){\line(-1,-2){10}}
%\put(170,540){\line(-1,-2){10}}
%\put(190,540){\line(-1,-2){10}}
%\put(210,540){\line(-1,-2){10}}
%\put(210,540){\line(0,-1){20}}
%\put(135,522){\line(0,-1){4}}
%\put(150,522){\line(0,-1){4}}
%\put(170,522){\line(0,-1){4}}
%\put(190,522){\line(0,-1){4}}
%\put(205,522){\line(0,-1){4}}
%\put(140,542){\line(0,-1){4}}
%\put(160,542){\line(0,-1){4}}
%\put(180,542){\line(0,-1){4}}
%\put(200,542){\line(0,-1){4}}
\end{picture}
\end{center}
\caption{\label{LN} Laces in $\mathcal{L}^{(N)}[a,b]$
for $N = 1,2,3,4$, with $s_1=a$ and $t_N=b$.}
\end{figure}
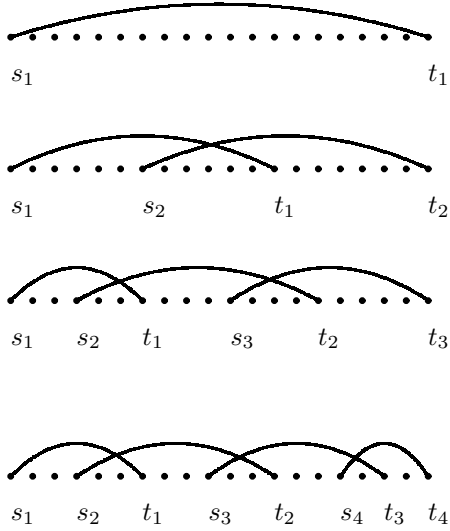

\begin{solution}{Problem \ref{problem:LaceIntervals}}
Figure~\ref{LN} is helpful.
Note first that if $L$ is a lace, then $s_l<s_{l+1}$ for each $l$.
Indeed, if $s_l=s_{l+1}$, we may assume that $t_l<t_{l+1}$.
But then $(s_l,t_l)\subset(s_{l+1},t_{l+1})$ so that $L\setminus\set{s_l t_l}$
is still connected.  A similar argument gives $t_l<t_{l+1}$.
The requirement that $L$ is connected implies that $a=s_1$ and $b=t_N$.

Suppose to the contrary that (1) $s_{l+1}\geq t_l$ ($1\leq l\leq N-1$)
or (2) $s_{l+2}< t_l$ ($1\leq l\leq N-2$).  In case (1), $L$ is not connected,
since $s_i\geq t_l$ for $i\geq l+1$ while $t_i\leq t_l$ for $i\leq l$.
In case (2), the edge $s_{l+1}t_{l+1}$ is redundant since
$(s_{l+1},t_{l+1})\subset(s_l,t_l)\union (s_{l+2},t_{l+2})=(s_l,t_{l+2})$.

For the converse, the hypotheses imply that $\union_{st\in L}(s,t)=(a,b)$,
so $L$ is connected.  Neither $s_1 t_1$ nor $s_N t_N$ can be removed from $L$
since they are the only edges containing the endpoints.  If $s_l t_l$ is
removed, $2\leq l\leq N-1$, then $t_{l-1}\leq s_{l+1}$ implies that
$\union_{st\in L} (s,t) = (a,t_{l-1})\union(s_{l+1},b)\neq(a,b)$.
So $L\setminus\set{st}$ is not connected.  Since connectedness is a
monotone property, no strict subset of $L$ can be connected, so $L$
is minimally connected, i.e., a lace.

Finally the intervals are as follows: the first and last intervals are
$[s_1,s_2]$ and $[t_{N-1},t_N]$; the $2i^\text{th}$ interval is
$[s_{i+1},t_i]$ ($1\leq i\leq N-1$); and the $(2i+1)^\text{st}$ interval
is $[t_i,s_{i+2}]$, $1\leq i\leq N-2$.  The inequalities above show that
the points $\set{s_i,t_i}$ do indeed form the intervals claimed, and that
the intervals $[t_i,s_{i+2}]$ can be empty while the other intervals must
be non-empty.
\end{solution}

\begin{figure}
\begin{center}
\setlength{\unitlength}{0.01in}%
\begin{picture}(225,212)(-25,600)
\thicklines
\put( 60,800){\makebox(0,0)[lb]{\raisebox{0pt}[0pt][0pt]{\circle*{2}}}}
\put( 80,800){\makebox(0,0)[lb]{\raisebox{0pt}[0pt][0pt]{\circle*{2}}}}
\put(100,800){\makebox(0,0)[lb]{\raisebox{0pt}[0pt][0pt]{\circle*{2}}}}
\put(120,800){\makebox(0,0)[lb]{\raisebox{0pt}[0pt][0pt]{\circle*{2}}}}
\put(140,800){\makebox(0,0)[lb]{\raisebox{0pt}[0pt][0pt]{\circle*{2}}}}
\put(160,800){\makebox(0,0)[lb]{\raisebox{0pt}[0pt][0pt]{\circle*{2}}}}
\put(180,800){\makebox(0,0)[lb]{\raisebox{0pt}[0pt][0pt]{\circle*{2}}}}
\put(200,800){\makebox(0,0)[lb]{\raisebox{0pt}[0pt][0pt]{\circle*{2}}}}
\put(220,800){\makebox(0,0)[lb]{\raisebox{0pt}[0pt][0pt]{\circle*{2}}}}
\put(240,800){\makebox(0,0)[lb]{\raisebox{0pt}[0pt][0pt]{\circle*{2}}}}
\qbezier(60,800)(100,830)(140,800)
\qbezier(80,800)(90,810)(100,800)
\qbezier(100,800)(140,830)(180,800)
\qbezier(160,800)(200,830)(240,800)
\qbezier(180,800)(200,820)(220,800)
%\put(100,800){\oval(80,20)[t]}
%\put(90,800){\oval(20,10)[t]}
%\put(150,800){\oval(60,10)[t]}
%\put(200,800){\oval(80,20)[t]}
%\put(200,800){\oval(40,10)[t]}
\put( 60,740){\makebox(0,0)[lb]{\raisebox{0pt}[0pt][0pt]{\circle*{2}}}}
\put( 80,740){\makebox(0,0)[lb]{\raisebox{0pt}[0pt][0pt]{\circle*{2}}}}
\put(100,740){\makebox(0,0)[lb]{\raisebox{0pt}[0pt][0pt]{\circle*{2}}}}
\put(120,740){\makebox(0,0)[lb]{\raisebox{0pt}[0pt][0pt]{\circle*{2}}}}
\put(140,740){\makebox(0,0)[lb]{\raisebox{0pt}[0pt][0pt]{\circle*{2}}}}
\put(160,740){\makebox(0,0)[lb]{\raisebox{0pt}[0pt][0pt]{\circle*{2}}}}
\put(180,740){\makebox(0,0)[lb]{\raisebox{0pt}[0pt][0pt]{\circle*{2}}}}
\put(200,740){\makebox(0,0)[lb]{\raisebox{0pt}[0pt][0pt]{\circle*{2}}}}
\put(220,740){\makebox(0,0)[lb]{\raisebox{0pt}[0pt][0pt]{\circle*{2}}}}
\put(240,740){\makebox(0,0)[lb]{\raisebox{0pt}[0pt][0pt]{\circle*{2}}}}
\qbezier(60,740)(100,770)(140,740)
\qbezier(100,740)(140,770)(180,740)
\qbezier(160,740)(200,770)(240,740)
%\put(100,740){\oval(80,20)[t]}
%\put(150,740){\oval(60,10)[t]}
%\put(200,740){\oval(80,20)[t]}
\put( 60,680){\makebox(0,0)[lb]{\raisebox{0pt}[0pt][0pt]{\circle*{2}}}}
\put( 80,680){\makebox(0,0)[lb]{\raisebox{0pt}[0pt][0pt]{\circle*{2}}}}
\put(100,680){\makebox(0,0)[lb]{\raisebox{0pt}[0pt][0pt]{\circle*{2}}}}
\put(120,680){\makebox(0,0)[lb]{\raisebox{0pt}[0pt][0pt]{\circle*{2}}}}
\put(140,680){\makebox(0,0)[lb]{\raisebox{0pt}[0pt][0pt]{\circle*{2}}}}
\put(160,680){\makebox(0,0)[lb]{\raisebox{0pt}[0pt][0pt]{\circle*{2}}}}
\put(180,680){\makebox(0,0)[lb]{\raisebox{0pt}[0pt][0pt]{\circle*{2}}}}
\put(200,680){\makebox(0,0)[lb]{\raisebox{0pt}[0pt][0pt]{\circle*{2}}}}
\put(220,680){\makebox(0,0)[lb]{\raisebox{0pt}[0pt][0pt]{\circle*{2}}}}
\put(240,680){\makebox(0,0)[lb]{\raisebox{0pt}[0pt][0pt]{\circle*{2}}}}
\qbezier(60,680)(100,710)(140,680)
\qbezier(100,680)(140,710)(180,680)
\qbezier(160,680)(200,710)(240,680)
%\put(100,680){\oval(80,20)[t]}
%\put(150,680){\oval(60,10)[t]}
%\put(200,680){\oval(80,20)[t]}
\thicklines
\qbezier[30](80,680)(120,650)(160,680)
\qbezier[30](140,680)(180,650)(220,680)
%\put(120,680){\oval(80,20)[b]}
%\put(180,680){\oval(80,30)[b]}
%\thinlines
\put( 60,620){\makebox(0,0)[lb]{\raisebox{0pt}[0pt][0pt]{\circle*{2}}}}
\put( 80,620){\makebox(0,0)[lb]{\raisebox{0pt}[0pt][0pt]{\circle*{2}}}}
\put(100,620){\makebox(0,0)[lb]{\raisebox{0pt}[0pt][0pt]{\circle*{2}}}}
\put(120,620){\makebox(0,0)[lb]{\raisebox{0pt}[0pt][0pt]{\circle*{2}}}}
\put(140,620){\makebox(0,0)[lb]{\raisebox{0pt}[0pt][0pt]{\circle*{2}}}}
\put(160,620){\makebox(0,0)[lb]{\raisebox{0pt}[0pt][0pt]{\circle*{2}}}}
\put(180,620){\makebox(0,0)[lb]{\raisebox{0pt}[0pt][0pt]{\circle*{2}}}}
\put(200,620){\makebox(0,0)[lb]{\raisebox{0pt}[0pt][0pt]{\circle*{2}}}}
\put(220,620){\makebox(0,0)[lb]{\raisebox{0pt}[0pt][0pt]{\circle*{2}}}}
\put(240,620){\makebox(0,0)[lb]{\raisebox{0pt}[0pt][0pt]{\circle*{2}}}}
\qbezier(60,620)(100,650)(140,620)
\qbezier(100,620)(140,650)(180,620)
\qbezier(160,620)(200,650)(240,620)
%\put(100,620){\oval(80,20)[t]}
%\put(150,620){\oval(60,10)[t]}
%\put(200,620){\oval(80,20)[t]}
\thicklines
\qbezier[30](100,620)(150,590)(200,620)
%\put(150,620){\oval(100,20)[b]}
%\thinlines
\put( 60,780){\makebox(0,0)[lb]{\raisebox{0pt}[0pt][0pt]{$a$}}}
\put(240,780){\makebox(0,0)[lb]{\raisebox{0pt}[0pt][0pt]{$b$}}}
\put( 60,720){\makebox(0,0)[lb]{\raisebox{0pt}[0pt][0pt]{$a$}}}
\put(240,720){\makebox(0,0)[lb]{\raisebox{0pt}[0pt][0pt]{$b$}}}
\put( 60,660){\makebox(0,0)[lb]{\raisebox{0pt}[0pt][0pt]{$a$}}}
\put(240,660){\makebox(0,0)[lb]{\raisebox{0pt}[0pt][0pt]{$b$}}}
\put( 60,600){\makebox(0,0)[lb]{\raisebox{0pt}[0pt][0pt]{$a$}}}
\put(240,600){\makebox(0,0)[lb]{\raisebox{0pt}[0pt][0pt]{$b$}}}
\put( -90,800){\makebox(0,0)[lb]{\raisebox{0pt}[0pt][0pt]{(a)}}}
\put( -90,680){\makebox(0,0)[lb]{\raisebox{0pt}[0pt][0pt]{(b)}}}
\put( -90,620){\makebox(0,0)[lb]{\raisebox{0pt}[0pt][0pt]{(c)}}}
\put( 15,800){\makebox(0,0)[lb]{\raisebox{0pt}[0pt][0pt]{$\Gamma$}}}
\put( 15,740){\makebox(0,0)[lb]{\raisebox{0pt}[0pt][0pt]{${\sf L}_\Gamma$}}}
\put( 15,680){\makebox(0,0)[lb]{\raisebox{0pt}[0pt][0pt]{$L$}}}
%\put( -75,645){\makebox(0,0)[lb]{\raisebox{0pt}[0pt][0pt]{two compatible edges}}}
%\put( 60,645){\line(1,0){100}}
%\put(100,645){\vector(1,1){20}}
%\put(160,645){\vector(1,1){15}}
\put( 15,620){\makebox(0,0)[lb]{\raisebox{0pt}[0pt][0pt]{$L$}}}
%\put( -75,585){\makebox(0,0)[lb]{\raisebox{0pt}[0pt][0pt]
%{an incompatible edge}}}
%\put( 60,585){\line(1,0){60}}
%\put(120,585){\vector(1,1){20}}
\end{picture}
\end{center}
\caption{(a) A connected graph $\Gamma$ and its associated
lace $L={\sf L}_ \Gamma$.  (b)  The dotted edges are compatible with the lace $L$.
(c)  The dotted edge is not compatible with the lace $L$.}
\label{lacepresc}
\end{figure}
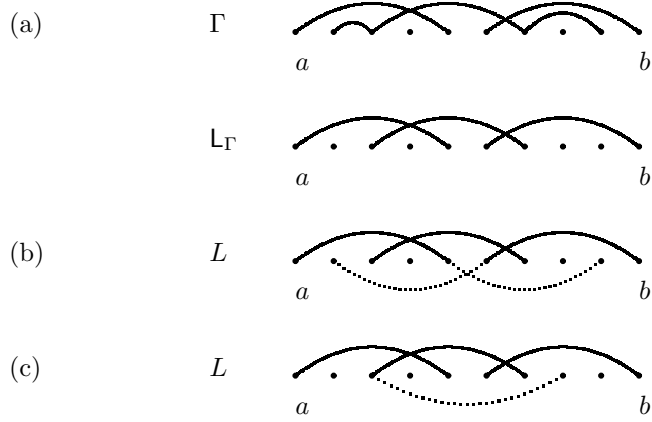

\begin{solution}{Problem \ref{problem:Compatible}}
Figure~\ref{lacepresc} is helpful.
First, since necessarily ${\sf L}_\Gamma\subset\Gamma$, we may assume
that $L\subset\Gamma$, and we write $\Gamma=L\union A$ with
$A\intersect L=\emptyset$.

Next, we reformulate the inductive procedure for selecting the edges
of ${\sf L}_\Gamma$.  At each step, the edge $s_{i+1}t_{i+1}$ is,
among all edges $st\in\Gamma$ satisfying $s<t_i$, the one that is
\emph{maximal} with respect to the following order relation:
$st\succ s't'$ if and only if $t>t'$ or $t=t'$ and $st$ is longer than $s't'$
(i.e., $t-s>t'-s'$).

The result follows at once from this observation.
Indeed, ${\sf L}_\Gamma=L$ means that at each inductive step,
$s_{i+1}t_{i+1}\in L$ is the maximal edge $st$ satisfying $s<t_i$,
among all edges of $L\union A$.  This is equivalent to saying that
for each $s't'\in A$, at each inductive step, $s_{i+1}t_{i+1}$ is
the maximal edge among all edges of $L\union \set{s't'}$.
But this is precisely the condition that $A\subset \mathcal{C}(L)$.
%
%Equivalently, we must prove that for any $A\subset\mathcal{C}(L)$, if $\Gamma=L\union A$ then $L_\Gamma=L$.  In the inductive definition of $L_\Gamma$, the edge $s_{i+1}t_{i+1}$ is chosen to be that edge $st$ satisfying $s<t_i$ that is \emph{minimal} with respect to the following order relation: $st\prec s't'$  if and only if $t>t'$ or $t=t'$ and $s<s'$.  But $\prec$ is a total ordering on the edges, so the claim follows.
%
%(In more detail: the condition $A\subset\mathcal{C}(L)$ implies that, for each $st\in A$, and at each stage of the inductive definition, the edge $s_{i+1}t_{i+1}\in L$ is minimal among the edges $s't'\in L\union\set{st}$ satisfying $s'<t_i$.  But this implies that at each stage $s_{i+1}t_{i+1}$ is minimal among the allowed edges from $L\union A$ as well.  So the same edges will be chosen inductively at each step.)
\end{solution}

\begin{solution}{Problem \ref{problem:Jsumprod}}
The first equation is simply a decomposition of
$\Gamma\in\mathcal{G}[a,b]$ according to the value of ${\sf L}_\Gamma$.
The second equation follows using \eqref{e:SubsetProduct} because
Problem~\ref{problem:Compatible} shows that $\Gamma$ for which
${\sf L}_\Gamma=L$ can be identified as $L$ together with
an arbitrary subset of edges from $\mathcal{C}(L)$.
The last equation is immediate from the preceding ones.
\end{solution}

\begin{solution}{Problem \ref{problem:J-laces}}
(a)
By definition,
\begin{align}
%\label{}
\sum_{N=1}^\infty (-1)^N \pi_m^{(N)}(x)
&=
\sum_{\omega\in\mathcal{W}_m(0,x)} \sum_{N=1}^\infty J^{(N)}[0,m](\omega)
\notag\\
&=
\sum_{\omega\in\mathcal{W}_m(0,x)} J[0,m](\omega) =\pi_m(x).
\end{align}
Each of the $N$ factors $U_{st}$, $st\in L$, contributes $-1$, so
$\pi_m^{(N)}(x)\geq 0$.

\smallskip (b)
$N=1$:
The only lace with 1 edge is $L=\set{0m}$, and every edge except $0m$ is
compatible with $L$.  So $J^{(1)}[0,m]$ contains the single factor $U_{0m}$,
and the factor $1+U_{s't'}$ for each $s't'\neq 0m$.  So a contributing
$\omega$ must have $\omega(s')\neq\omega(t')$ whenever $s't'\neq 0m$,
as well as $0=\omega(0)=\omega(m)=x$.  Hence $\pi_m^{(1)}(x)=0$ for $x\neq 0$,
and $\pi_m^{(1)}(0)$ is the number of $m$-step self-avoiding returns.

$N=2$:
For $L=\set{0t_1,s_1 m}$, the factors $U_{st}$, $st\in L$, require that
$\omega$ should start at $0$, visit $x$ (at step $s_1$), return to $0$
(at step $t_1$), then return to $x$.  The compatible edges consist of
every edge except the edges of $L$ and the edges $0t$, $t>t_1$ and $sm$, $s<s_1$.  This implies that each of the three intervals in $\omega$ must be
self-avoiding and mutually avoiding, except for the intersections required above.  (In particular, $x\neq 0$.)  (Intersections of the form $\omega(0)=\omega(t)$, $t>t_1$, might not appear to be forbidden, but actually they are impossible since we require $\omega(t)\neq \omega(t_1)=\omega(0)$.)

$N=3,4,\dotsc$:
As for $N=2$, $\omega$ must have self-intersections corresponding to
the edges of the lace and self-avoidance corresponding to each compatible
edge.  It is convenient to recall the $2N-1$ intervals from
Problem~\ref{problem:LaceIntervals}.  Because of compatible edges,
$\omega$ is required to be self-avoiding on each of these intervals.
In addition, certain of these intervals are required to be mutually avoiding,
but not all of them need be, corresponding to the fact that an edge spanning
too many intervals cannot be compatible.  The pattern of mutual avoidance is
described as follows: for $N=3$,
\begin{equation}
%\label{}
[1234][345]
\end{equation}
and for $N=4$,
\begin{equation}
%\label{}
[1234][3456][567]
\end{equation}
where, for instance $[3456]$ indicates that the third to sixth interval
must be mutually self-avoiding, except for the required intersections.
These intersections require that at the endpoints of the intervals,
$\omega$ must visit the following points (for the case $N=4$):
\begin{equation}
%\label{}
0,x_1,0,x_2,x_1,x_3,x_2,x_3
\end{equation}
where $x_3=x$, corresponding to the intervals $[s_1,s_2]$, $[s_2,t_1]$,
$[t_1,s_3]$, $[s_3,t_2]$, $[t_2,s_4]$, $[s_4,t_3]$, $[t_3,t_4]$.

To prove the avoidance patterns amounts to analysing
exactly which edges are compatible.  For instance, it is easy to verify
that if $s_{i+1}\leq s <t_i$ for $1\leq i\leq N-1$, then $st\in\mathcal{C}(L)$
if and only if $t\leq t_{i+1}$ (assuming $st\notin L$).

\smallskip (c)
The possibly empty intervals indicate that the
$3^\text{rd},5^\text{th},\dotsc,(2N-3)^\text{rd}$ segments of the
diagrams above can be empty, whereas all other segments must have non-zero length.
The picture for $N=11$ is
    \begin{center}
      \includegraphics[width=0.5\hsize]{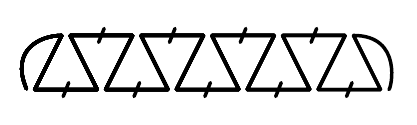}
     \end{center}
where the lines that are slashed are exactly the lines that are permitted
to have length zero.
\end{solution}

\subsection{Solutions for Tutorial~\ref{s:Tut3}}

\begin{solution}{Problem \ref{problem:1/d}}
  (a) Using the hint and Fubini's theorem (which is applicable because $d > 2m$),
  \begin{align}
    \int_{[-\pi,\pi]^d} \frac{1}{[1-\hat D(k)]^m} \frac{d^dk}{(2\pi)^d}
    &= \Gamma(m)^{-1} \int_0^\infty \left(\int_{[-\pi,\pi]^d} e^{-u[1-\hat D(k)]} \frac{d^dk}{(2\pi)^d} \right) u^{m-1} \; du
    \notag\\
    &= \Gamma(m)^{-1} \int_0^\infty \left(\int_{-\pi}^\pi e^{-u(1-\cos k)/d} \; \frac{dk}{2\pi} \right)^d u^{m-1} \; du.
  \end{align}
  The inner integral is decreasing as a function of $d$ by H\"older's inequality.
  \medskip

  (b) Relaxing the self-avoidance for the first $j$ steps
  gives the inequality
  \begin{equation}
  %\label{}
    H^{(j)}_z(x) \leq (z|\Omega| D)^{*j} * G_z(x),
  \end{equation}
  and thus, by Cauchy-Schwarz,
  \begin{equation}
    \|H^{(j)}_z\|_\infty
    \leq \|\hat H^{(j)}_z\|_1
    \leq (z|\Omega|)^j \|\hat D^{j}\|_2 \|\hat G_z\|_2.
  \end{equation}
  The claim now follows from $z_c|\Omega|\leq a$,
  \begin{equation}
    \|\hat D^{j}\|_2 = \|\hat D^{2j}\|_1^{1/2} \leq O((2d)^{-j/2}),
  \end{equation}
  and
  \begin{equation}
    \|\hat G_z\|_2 \leq a \|\hat C_{p(z)}\|_2 \leq a\|\hat C_{1/|\Omega|}\|_2 \leq O(1)
  \end{equation}
  by the infrared bound \eqref{eq:infrared} and (a).
  \medskip

  (c) Calculating the first two terms explicitly, we obtain
  \begin{align}
    \hat\Pi^{(1)}_{z}(0)
    &= (2d) z^2 + (2d)(2d-2) z^4 +  \sum_{m \geq 6} \hat \pi^{(1)}_m(0) z^m
    \notag\\
    &= (2d) z^2 + (2d)(2d-2) z^4 + O((2d)^{-3}),
  \end{align}
  where the remainder was estimated as in (b), and using the symmetry of $D$:
  \begin{align}
    \sum_{m \geq j} \hat \pi^{(1)}_m(0) z^m
    &=
    (H^{(j-1)}_z * z|\Omega|D) (0)
    \notag\\
    &\leq
    (z|\Omega|)^j \, \|D^{*j} * G_z\|_\infty
    \leq O((2d)^{-j/2}).
  \end{align}
  Using \eqref{eq:Pi-lace-bound}, we obtain
  \begin{equation}
    \hat \Pi_z(0) = - \hat \Pi^{(1)}_z(0) + O((2d)^{-2}).
  \end{equation}
 Equation \eqref{eq:G-z_c} gives
 $z_c=(2d)^{-1}-(2d)^{-1} \hat\Pi_{z_c}(0) = (2d)^{-1}+O((2d)^{-2})$,
 from which we obtain
  \begin{align}
    z_c &= (2d)^{-1} + (2d)^{-1} \hat\Pi^{(1)}_{z_c}(0) + O((2d)^{-3})
    \notag\\
    &= (2d)^{-1} + (2d)^{-2} +O((2d)^{-3}).
  \end{align}
  Finally, we obtain
  \begin{equation}
    \hat\Pi^{(1)}_{z_c}(0) = [(2d)^{-1} + 2(2d)^{-2}] + (2d)^{-2} + O((2d)^{-3}).
  \end{equation}
  \medskip

  (d) The generating function for $\theta$-walks from $0$ to $x$ can be written as
  \begin{equation}
    \hat\Pi^{(2)}_z(0) = (2d)z^3 + 3(2d)(2d-2) z^5 + \sum_{m\geq 7} \hat\pi^{(2)}_m(0)z^m.
  \end{equation}
  Using $z_c = (2d)^{-1} + (2d)^{-2} +O((2d)^{-3})$ from part (c), we obtain
  \begin{equation}
    (2d)z_c^3 + 3(2d)(2d-2) z_c^5   = (2d)^{-2} + O((2d)^{-3}).
  \end{equation}
  The remainder is estimated using the fact that in a
  $\theta$-walks from $0$ to $x$ of length $m\geq 7$, either two of the
  subwalks take just one step and the other takes at least 5 steps,
  or at least two of the subwalks take at least 3 steps.
  Thus there is a combinatorial constant $K$ such that
  \begin{align}
    \begin{split}
    \sum_{m\geq 7} \hat\pi^{(2)}_m(0) z^m
    &\leq
    K \sum_{e} H^{(5)}_z(e)z|\Omega|D(e)z|\Omega|D(e)
    \\
    &\qquad
    %+ \sum_e z|\Omega|D(e)H^{(j-2)}_z(-e)z|\Omega|D(e) \\
    %+ 3 \sum_e H^{(1)}_z(e) H^{(3)}_z(-e)H^{(3)}_z(e)
    + K \sum_x H^{(1)}_z(x) H^{(3)}_z(x)H^{(3)}_z(x).
    \end{split}
  \end{align}
  The first term can be estimated by an $L^\infty$ bound
  (use $z|\Omega| \leq a$ and $|\Omega|=2d$):
  \begin{align}
    \sum_{e} H^{(5)}_z(e)z|\Omega|D(-e)z|\Omega|D(e)
    &\leq \|H^{(5)}_z\|_\infty a^2(2d)^{-1}
    \notag\\
    &\leq O((2d)^{-7/2}) \leq O((2d)^{-3}).
  \end{align}
  % \begin{multline*}
  %   \sum_{e} (H^{(j-2)}_z)(e)z|\Omega|D(-e)z|\Omega|D(e)
  %   \leq (2d)^{-1} (H^{(j-2)}_z * z|\Omega|D)(0)
  %   \leq (2d)^{-1} ((D)^{*(j-1)} * G_z)(0)\\
  %   \leq (2d)^{-1} \|\hat D^{(j-1)} \hat G_z\|_1
  %   \leq (2d)^{-1} \|\hat D^{(j-1)}\|_2 \|\hat G_z\|_2
  %   \leq (2d)^{-1} O((2d)^{-(j-1)/2}) O(1)
  %   \leq O((2d)^{-4}),
  % \end{multline*}
  %and the second term is very similar;
  % the third term is
  % \begin{equation*}
  %   \sum_e H^{(1)}_z(e) H^{(3)}_z(-e)H^{(3)}_z(e) \leq (2d) O((2d)^{-1}) O((2d)^{-3/2}) O((2d)^{-3/2}) \leq O((2d)^{-3}).
  % \end{equation*}
  The second term is estimated in the spirit of (b):
  \begin{align}
    \sum_x H^{(1)}_z(x) H^{(3)}_z(-x)H^{(3)}_z(x)
    &\leq O((2d)^{-1/2}) (H^{(3)}_z * H^{(3)}_z)(0)
    \notag\\
    &\leq O((2d)^{-1/2}) ((z|\Omega|D)^{*6} * G_z^{*2})(0)
    \notag\\
    &\leq O((2d)^{-1/2}) \|\hat D^6\hat G_z^2\|_1
    \notag\\
    & \leq O((2d)^{-1/2}) \|\hat D^6\|_2 \|\hat G_z^2\|_2
    %\leq O((2d)^{-1}) O((2d)^{-3})
    \leq O((2d)^{-7/2}).
  \end{align}
  \medskip

  (e) Using \eqref{eq:Pi-lace-bound}, we obtain
  \begin{equation}
    \hat \Pi_{z_c}(0) = - (2d)^{-1} - 2(2d)^{-2} + O((2d)^{-3}).
  \end{equation}
  From \eqref{eq:G-z_c}, it then follows that
  \begin{equation}
    z_c=(2d)^{-1}+(2d)^{-2} + 2(2d)^{-3} + O((2d)^{-4}).
  \end{equation}
  Inverting this finally yields
  \begin{equation}
    \mu = 2d -1 - (2d)^{-1} + O((2d)^{-2}). \qedhere
  \end{equation}
\end{solution}

% First, recall that
% \begin{equation} \label{eq:Pi}
%   \Pi^{(N)}_z(x) = \sum_{n\geq 0} \pi^{(N)}_n(x) z^n, \quad \pi^{(N)}_n(x) \geq 0.
% \end{equation}

% \begin{problem} \label{problem:correlation-length}
%   \begin{enumerate}
%   \item
%     Let $p(z)$ be defined by $\hat C_{p(z)}(0) = \hat G_{z}(0)$. Show that
%     \begin{equation}
%       \sum_{x\in\Z^d} [1-\cos(k \cdot x)] |\Pi_z(x)| \leq \bar c_K \beta (1-\hat D(k)) .
%     \end{equation}
%   \item
%     Let $\sigma^2 = \sum_{x\in\Z^d} |x|^2 D(x)$. Prove that, uniformly in $z < z_c$,
%     \begin{equation}
%       \sum_{x\in\Z^d} |x|^2 |\Pi_z(x)| \leq O(\beta\sigma^2).
%     \end{equation}

%   \end{enumerate}
% \end{problem}

% \begin{solution}{Solution \ref{problem:correlation-length}}
%   (a) Note that
%   \begin{equation}
%     \|[1-\cos(k\cdot x)] H_z\|_\infty
%     = \|[1-\cos(k\cdot x)] G_z\|_\infty
%     \leq \frac{1}{2} \|\Delta_k \hat G_z\|_1
%   \end{equation}

%   We see from the definition \eqref{eq:Pi} that $\Pi^{(N)}_z(x) \leq \Pi^{(N)}_{z_c}(x)$ (for $z \leq z_c$).
%   Hence,
%   \begin{equation}
%     \sum_{x\in\Z^d} [1-\cos(k \cdot x)] |\Pi^{(N)}_z(x)| \leq \bar c_K \beta (1-\hat D(k)) .
%   \end{equation}

% \end{solution}

\begin{solution}{Problem \ref{problem:chi-asymptotic}}
  (a) This is a straightforward calculation.
  \medskip

  (b) This requires an extension of the diagrammatic estimates.  The
  argument is sketched in \cite[Section~5.4]{Slad06}.
  \medskip

  (c) Note that
  \begin{equation}
    \frac{d[z\chi(z)]^{-1}}{dz}
    = -[z\chi(z)]^{-2} \frac{d[z\chi(z)]}{dz}
    \sim -\frac{V(z_c)}{z^2}
    \sim -\frac{V(z_c)}{z_c^2} .
  \end{equation}
  Integrating this asymptotic relation, we obtain
  \begin{equation}
    \lim_{z\to z_c} [z\chi(z)]^{-1} - [z\chi(z)]^{-1}
    \sim -\frac{V(z_c)}{z_c^2} (z_c-z).
  \end{equation}
  The limit vanishes and thus
  \begin{equation}
    \chi(z)^{-1}
    \sim  V(z_c)^{-1}  (1-z/z_c)^{-1}.
  \end{equation}
  The claim then follows from the definition of $V(z_c)$ and \eqref{eq:G-z_c}.
\end{solution}

\subsection{Solutions for Tutorial~\ref{s:Tut4}}

\begin{solution}{Problem \ref{problem:wsaw-submultiplicativity}}
For $m,n\geq 0$,
\begin{align}
  I_{0,n+m}
  &= \sum_{0\leq i < j \leq n+m} \indicator{X_{i} = X_{j}}
  \notag\\
  &\geq \sum_{0\leq i < j \leq m}\indicator{X_{i} = X_{j}}+\sum_{m \leq i < j \leq n+m} \indicator{X_{i} = X_{j}}
  = I_{0,m} + I_{m,n+m}.
\end{align}
%because $L_t^{x} \geq L_s^{x} \geq 0$ and
%\begin{equation}
%  (L_t^{x})^2 = (L_s^{x} + (L_t^{x}-L_s^{x}))^2 \geq (L_s^{x})^2 + (L_t^{x}-L_s^{x})^2.
%\end{equation}
By translation invariance and the Markov property, $I_{m,n+m}$ is independent of $I_{0,m}$ and has the same law as $I_{0,n}$.  Therefore
\begin{equation}
\label{cSubmultiplicative}
  c_{n+m} \leq \E_0(e^{-gI_{0,m}} e^{-gI_{m,m+n}})
  = \E_0(e^{-gI_{0,m}}) \E_0(e^{-gI_{0,n}})=c_m c_n
\end{equation}
as claimed.  The remaining statements follow since \eqref{cSubmultiplicative}
implies that $(\log c_n)_{n\geq 0}$ is a subadditive sequence, and
\reflemma{Subadd} can be applied.
%It follows, using independence of $I_{0,m}$ and $I_{m,n+m}$, that
%\begin{equation}
%  c_{n+m} \leq E(e^{-gI_{0,m}} e^{-gI_{m,m+n}}) = E(e^{-gI_{0,m}}) E(e^{-gI_{m,m+n}}),
%\end{equation}
%and
%\begin{equation}
%  c_{n+m} \leq E(e^{-gI_{0,m}}) E(e^{-gI_{0,n}}) = c_m c_n,
%\end{equation}
%by the Markov property.
\end{solution}

\begin{solution}{Problem \ref{problem:kappa-torus}}
  Note that there is a one-to-one correspondence between nearest-neighbour walks on $\Z^d$
  and such walks on the torus $\Z^d / R\Z^d$, $R \geq 3$,
  by folding a walk on $\Z^d$
  (the image under the canonical projection $\Z^d \twoheadrightarrow \Z^d/R\Z^d$), and
  corresponding unfolding of walks on $\Z^d/R\Z^d$ (unique for the
  nearest-neighbour step distribution
  provided $R \geq 3$).
  Given a walk $X=(X_n)_{n\geq 0}$ on $\Z^d$ starting at $0$, we denote the folded
  (or projected) walk by $X'$. Write $\Lambda_R= \{-R+1, \dots, R\}^d$
  and $R'=2R+1$; then
  \begin{align}
    I_{n}(X)
    &=
    \sum_{0\leq i < j \leq n} \sum_{x\in\Z^d}  \indicator{X_i = X_j = x}
    = \sum_{0\leq i < j \leq n} \sum_{x\in \Lambda_R} \sum_{y \in \Z^d}
    \indicator{X_i = X_j = x+yR'}
    \notag\\
    &\leq
    \sum_{0\leq i < j \leq n} \sum_{x\in \Lambda_R}
    \sum_{y_1,y_2 \in \Z^d} \indicator{X_i = x+y_1R'} \indicator{X_j = x+y_2 R'}
    \notag\\
    &=
    I_{n}(X'),
  \end{align}
  and thus
  \begin{equation}
    \E(e^{-gI_n}) \geq \E^R(e^{-gI_n}).
  \end{equation}
  The desired inequalities both follow from this one.
\end{solution}

\begin{solution}{Problem \ref{problem:simon-lieb}}
Note that
\begin{equation}
  G_{\nu}(x,y)-G_{\nu,D}(x,y)
  = \sum_{n=0}^\infty \E_x(e^{-gI_n} \indicator{X_n=y,\, n\geq T_D}) e^{-\nu n},
\end{equation}
and, by partitioning in $T_D$ and $X_{T_D}$, we obtain
\begin{equation}
  \E_x(e^{-gI_n}\indicator{X_n=y, n\geq T_D})
  = \sum_{z\in \partial D}\sum_{m=0}^{n}  \E_x(e^{-gI_n} \indicator{X_n=y}\indicator{X_{T_D}=z}\indicator{T_D=m}).
\end{equation}
Using $I_{n} \geq I_{m}+I_{m,n}$ and the Markov property, it follows that
\begin{align}
  &\E_x(e^{-gI_n} \indicator{X_n=y} \indicator{X_{T_D} = z}\indicator{T_D=m})
  \nonumber \\ & \quad
  \leq \E_x(e^{-gI_{m}} \indicator{X_m=z} \indicator{T_D=m} e^{-gI_{m,n}} \indicator{X_n=y}) \nonumber\\
  % &= E(e^{-gI_{0,m}} \indicator{X_m=z} \indicator{H_U=m}) E(e^{-gI_{m,n}} \indicator{X_n=y} | X_m=z)\nonumber\\
  &\quad
  = \E_x(e^{-gI_{m}} \indicator{X_m=z} \indicator{T_D=m}) \E_z(e^{-gI_{n-m}} \indicator{X_{n-m}=y}).
\end{align}
Thus, because $\{ T_D = m, X_m =z \} = \{m \leq T_D, X_m =z\}$ for $z \in \partial D$,
\begin{align}
  G_{\nu}(x,y)-G_{\nu, D}(x,y)
  &\leq
  \sum_{z\in \partial D} \sum_{n=0}^\infty
  \sum_{m= 0}^{n} \E_x(e^{-gI_{m}} \indicator{X_m=z}
  \indicator{T_D=m})
  \notag\\
  &\qquad\qquad\cdot
  \E_z(e^{-gI_{n-m}} \indicator{X_{n-m}=y}) e^{-\nu n}
  \notag\\
  &\leq
  \sum_{z \in \partial D} G_{\nu,\bar D}(x,z) G_{\nu}(z,y),
\end{align}
%  \\
%  \nonumber
%  &\quad= \sum_{z\in \partial D} \sum_{m=0}^\infty  E_x(e^{-gI_{m}}
%  \indicator{X_m=z} \indicator{m\leq T_D}) e^{-\nu m}
%    \sum_{n=m}^{\infty}  E_z(e^{-gI_{n-m}} \indicator{X_{n-m}=y}) e^{-\nu (n-m)}
as claimed.
\end{solution}

\begin{solution}{Problem \ref{problem:exp-decay}}
  Let $m= \lfloor |y|_\infty/(R+1)\rfloor$.
  By the Simon-Lieb inequality \eqref{eq:sl}, translation invariance,
  and the bound $G_{\nu,\bar D}(x,z)\leq G_{\nu}(x,z)$, we have
  \begin{align}
    G_\nu(x,y) &\leq \sum_{z_1 \in x+\partial \Lambda_R} G_\nu(x, z_1) G_{\nu}(z_1, y) \nonumber\\
    &\leq \cdots \leq \sum_{z_1 \in z_0+\partial \Lambda_R}
    \dots \sum_{z_m \in z_{m-1} + \partial\Lambda_R}
    G_\nu(x, z_1) G_{\nu}(z_1, z_2) \cdots G_\nu(z_m, y)
    \nonumber \\
    &\leq \theta^m \sup_{x \in \Z^d} G_\nu(0,x).
  \end{align}
  Note that we applied \eqref{eq:sl}
  in such a manner that the term $G_{\nu,D}(x,y)$ vanishes.
\end{solution}

\begin{solution}{Problem \ref{problem:infinite-volume-limit}}
Fix $\nu>\nu_c$, and let $D_R %= \mathrm{int}(\Lambda_R)
= \{-R+2,\ldots,R-1\}^d$
be the interior of $\Lambda_R$.
%Let $U_L = \partial\Lambda_L = \{ x \in \Z^d: \abs{x}_\infty = L \}$.
By monotone convergence,
\begin{equation}
  G_{\nu}(x,y) = \lim_{R \to \infty} G_{\nu,D_R}(x,y).
\end{equation}
Hence, to prove \eqref{eq:TorusToInfty}, it suffices to show that
$\lim_{R \to \infty}G^R_{\nu}(x,y)- G_{\nu,D_R}(x,y)=0$.  Now,
\begin{equation}
  G^R_{\nu}(x,y)- G_{\nu,D_R}(x,y)
  = G^R_{\nu}(x,y) - G^R_{\nu,D_R}(x,y),
\end{equation}
and thus, from the Simon-Lieb inequality (Problem~\ref{problem:simon-lieb}), it follows that
\begin{align}
  G^R_{\nu}(x,y)- G_{\nu,D_R}(x,y)
  &\leq \sum_{z\in \partial D_R} G^R_{\nu,\bar D_R}(x,z) G^R_{\nu}(z,y)
  \notag\\
  &\leq \Big(\sup_{z\in \partial D_R}G^R_{\nu,\bar D_R}(x,z)\Big)
  \Big(\sum_{z\in \partial D_R} G^R_{\nu}(z,y)\Big).
\end{align}
By Problem~\ref{problem:kappa-torus},
\begin{equation}
  \sum_{z\in \partial D_R} G^R_{\nu}(z,y)
  \leq \sum_{z \in \Lambda_R} G^R_{\nu}(z,y)
  = \chi^R(\nu) \leq  \chi(\nu) < \infty,
\end{equation}
and, by Problem~\ref{problem:exp-decay} and the fact that $G_\nu(0,x)$ is
uniformly bounded since the susceptibility is finite,
\begin{align}
  \sup_{z\in \partial D_R} G^R_{\nu,\bar D_R}(x,z)
  &= \sup_{z\in \partial D_R} G_{\nu,\bar D_R}(x,z)
  \leq \sup_{z\in \partial D_R} G_\nu(x,z)
  \notag\\
  &\leq  \sup_{z\in \partial D_R} Ce^{- \gamma  |z-x|}
  \leq Ce^{-\gamma (R-\abs{x})}\to 0
\end{align}
as $R\to\infty$.  Therefore
$\lim_{R \to \infty}G^R_{\nu}(x,y)- G_{\nu,\bar D_R}(x,y)=0$,
proving \eqref{eq:TorusToInfty}.  Finally \eqref{eq:kappaToCrit}
follows since $G_{\nu_c}(x,y)=\lim_{\nu\searrow \nu_c}G_{\nu}(x,y)$
by monotone convergence.
\end{solution}

\section*{Acknowledgements}

The authors are grateful to the organisers of both the Clay Mathematics Institute
Summer School and the XIV Brazilian Probability School for having delivered
such a stimulating and energetic event.  We are particularly grateful to
Vladas Sidoravicius for his tireless efforts on behalf of the School and its
participants.

Hugo Duminil-Copin thanks Stanislav Smirnov for very stimulating discussions and
encouragement related to their joint work presented in \refSect{cchex}.
Gordon Slade thanks David Brydges for many discussions
related to their joint work presented in \refSect{ctwsaw}.

The work of Hugo Duminil-Copin was supported by the ANR grant
BLAN06-3-134462, the EU Marie-Curie RTN CODY, the ERC AG CONFRA, and the
Swiss FNS.
The work of Gordon Slade was supported in part by NSERC of Canada.

%\bibliographystyle{amsplain}
%\bibliography{../bibdef/bib}

\providecommand{\bysame}{\leavevmode\hbox to3em{\hrulefill}\thinspace}
\providecommand{\MR}{\relax\ifhmode\unskip\space\fi MR }
\providecommand{\MRhref}[2]{%
  \href{http://www.ams.org/mathscinet-getitem?mr=#1}{#2}
}
\providecommand{\href}[2]{#2}

\end{document}